\documentclass[11pt]{siamart171218}

\usepackage[T1]{fontenc}
\usepackage[english]{babel}\usepackage{csquotes}\usepackage[activate={true,nocompatibility},tracking=true]{microtype}\usepackage{hyphenat}\usepackage{siunitx}
\usepackage{placeins}
\usepackage[sort]{cite}\usepackage[begintext=``, endtext='']{quoting}
\SetBlockEnvironment{quoting}
\SetBlockThreshold{1}
\usepackage{hyperref}

\usepackage{amsmath}
\usepackage{amsfonts}
\usepackage{amssymb}
\usepackage{amscd}
\usepackage{mathtools}
\usepackage[mathscr]{eucal}
\usepackage{calligra}
\DeclareMathAlphabet{\mathcalligra}{T1}{calligra}{m}{n}
\DeclareFontShape{T1}{calligra}{m}{n}{<->s*[1.1]callig15}{}
\usepackage{bm}
\usepackage[nice]{nicefrac}
\usepackage{dsfont}

\usepackage{tikz}
\usepackage{pgf}
\usepackage{pgfplots}
\usepackage{pgfplotstable}
\usepackage{shellesc}
\pgfplotsset{compat=newest}
\usetikzlibrary{spy}
\usetikzlibrary{shapes}
\usetikzlibrary{backgrounds}
\usetikzlibrary{shadows}
\usetikzlibrary{matrix}
\usetikzlibrary{external}
\usetikzlibrary{positioning}
\usetikzlibrary{plotmarks}
\usepgfplotslibrary{fillbetween}
\usepackage{comment}
\usepackage{graphicx}
\usepackage[colorinlistoftodos,prependcaption,textsize=small]{todonotes}
\usepackage[capitalize]{cleveref}

\crefname{equation}{}{}
\crefname{figure}{Figure}{Figures}
\usepackage{csvsimple}
\usepackage{booktabs}
\usepackage{environ}
\usepackage{adjustbox}
\usepackage{relsize}
\usepackage[font=small,figurewithin=section]{caption}\usepackage[subrefformat=simple,labelformat=simple]{subcaption}
\usepackage{afterpage}
\usepackage{booktabs}
\usepackage{arydshln}
\usepackage[ruled,vlined,algo2e]{algorithm2e}

\newtheorem{remark}{Remark}
\crefname{algocf}{algorithm}{algorithms}
\crefname{definition}{definition}{definitions}

\newcommand{\RomanNumeralCaps}[1]
\linenumbers

\usepackage{multirow}

\newtheorem{example}{Example}

\let\min\relax \DeclareMathOperator*\min{\vphantom{p}min}
\let\max\relax \DeclareMathOperator*\max{\vphantom{p}max}
\let\subset\relax \DeclareMathOperator{\subset}{\subseteq}
 
\let\tilde\widetilde

\DeclareMathOperator{\ndofs}{ndofs}

\DeclareMathOperator{\mean}{E}
\DeclareMathOperator{\variance}{V}
\DeclareMathOperator{\sd}{SD}
\DeclareMathOperator{\rms}{RMS}

\newcommand{\R}{\mathbb{R}}      \newcommand{\mesh}{\mcT} \newcommand{\dd}{\,\mathrm{d}}

\newcommand{\mcT}{\mathcal{T}}

\newcommand{\bbE}{\mathbb{E}}

\newcommand{\mfem}{{\texttt{MFEM}}}

\newcolumntype{?}{!{\vrule width 1.2pt}}

\makeatletter
\newsavebox{\measure@tikzpicture}
\NewEnviron{scaletikzpicturetowidth}[1]{	\tikzifexternalizingnext{		\def\tikz@width{#1}		\begin{lrbox}{\measure@tikzpicture}			\tikzset{external/export next=false,external/optimize=false}			\BODY
		\end{lrbox}		\pgfmathparse{#1/\wd\measure@tikzpicture}				\BODY
	}{		\BODY
	}}
\makeatother

\definecolor{color0}{rgb}{0.7843, 0.7843, 0.7843}
\definecolor{color1}{rgb}{0, 0.4470, 0.7410}
\definecolor{color2}{rgb}{0.8500, 0.3250, 0.0980}
\definecolor{color3}{rgb}{0.9290, 0.6940, 0.1250}
\definecolor{color4}{rgb}{0.7060, 0.3840, 0.7650}
\definecolor{color5}{rgb}{0.4660, 0.6740, 0.1880}
\definecolor{color6}{rgb}{0.3010, 0.7450, 0.9330}
\definecolor{color7}{rgb}{0.6350, 0.0780, 0.1840}
\definecolor{color8}{rgb}{0.0, 0.4078, 0.3412}

\pgfplotsset{
  log x ticks with fixed point/.style={
      xticklabel={
        \pgfkeys{/pgf/fpu=true}
        \pgfmathparse{exp(\tick)}        \pgfmathprintnumber[fixed relative, precision=3]{\pgfmathresult}
        \pgfkeys{/pgf/fpu=false}
      }
  },
  log y ticks with fixed point/.style={
      yticklabel={
        \pgfkeys{/pgf/fpu=true}
        \pgfmathparse{exp(\tick)}        \pgfmathprintnumber[fixed relative, precision=3]{\pgfmathresult}
        \pgfkeys{/pgf/fpu=false}
      }
  }
}

\tikzset{
  ashadow/.style={opacity=.25, shadow xshift=0.07, shadow yshift=-0.07},
}

\definecolor{CustomGreen}{RGB}{65,169,50}

\makeatletter
\newdimen\commentwd
\let\oldtcp\tcp
\def\tcp*[#1]#2{\setbox0\hbox{#2}\ifdim\wd\z@>\commentwd\global\commentwd\wd\z@\fi
\oldtcp*[r]{\leavevmode\hbox to \commentwd{\box0\hfill}}}

\let\oldalgorithm\algorithm
\def\algorithm{\oldalgorithm
\global\commentwd\z@
\expandafter\ifx\csname commentwd@\romannumeral\csname c@\algocf@float\endcsname\endcsname\relax\else
\global\commentwd\csname commentwd@\romannumeral\csname c@\algocf@float\endcsname\endcsname
\fi
}
\let\oldendalgorithm\endalgorithm
\def\endalgorithm{\oldendalgorithm
\immediate\write\@auxout{\gdef\expandafter\string\csname commentwd@\romannumeral\csname c@\algocf@float\endcsname\endcsname{\the\commentwd}}}

\newcommand{\red}[1]{{#1}}

\newcommand{\rev}[1]{{#1}}

\pgfplotsset{
    legend image with text/.style={
        legend image code/.code={            \node[anchor=center] at (0.3cm,0cm) {#1};
        }
    },
}

\makeatletter
\tikzset{
    scale plot marks/.is choice,
    scale plot marks/false/.code={
        \def\pgfuseplotmark##1{\pgftransformresetnontranslations\csname pgf@plot@mark@##1\endcsname}
    },
    scale plot marks/true/.style={},
    scale plot marks/.default=true
}
\makeatother
\graphicspath{{./Figures/}}

\begin{document}

\title{
	    Learning robust marking policies for adaptive mesh refinement
            \thanks{Submitted to the editors \today.
        }
}

\author{
	Andrew~Gillette\thanks{\protect
        Center for Applied Scientific Computing,
        Lawrence Livermore National Laboratory,
        Livermore, CA 94550 USA
        (\email{gillette7@llnl.gov}, \email{petrides1@llnl.gov}).}
    \and Brendan~Keith\thanks{\protect
        Division of Applied Mathematics,
        Brown University,
        Providence, RI 02912 USA
        (\email{brendan\_keith@brown.edu}).}
    \and Socratis~Petrides\footnotemark[2]
}

\date{\today}

\maketitle

\begin{abstract}
In this work, we revisit the marking decisions made in the standard adaptive finite element method (AFEM).
Experience shows that a na\"{i}ve marking policy leads to inefficient use of computational resources for adaptive mesh refinement (AMR).
Consequently, using AMR in practice often involves ad-hoc or time-consuming offline parameter tuning to set appropriate parameters for the marking subroutine.
To address these practical concerns, we recast AMR as a Markov decision process in which refinement parameters can be selected on-the-fly at run time, without the need for pre-tuning by expert users.
In this new paradigm, the refinement parameters are also chosen adaptively via a marking policy that can be optimized using methods from reinforcement learning.
We use the Poisson equation to demonstrate our techniques on $h$- and $hp$-refinement benchmark problems, and our experiments suggest that superior marking policies remain undiscovered for many classical AFEM applications.
Furthermore, an unexpected observation from this work is that marking policies trained on one family of PDEs are sometimes robust enough to perform well on problems far outside the training family.
For illustration, we show that a simple $hp$-refinement policy trained on 2D domains with only a single re-entrant corner can be deployed on far more complicated 2D domains, and even 3D domains, without significant performance loss.
For reproduction and broader adoption, we accompany this work with an open-source implementation of our methods.
\end{abstract}

\section{Introduction} \label{sec:introduction}

A longstanding challenge for adaptive finite element methods (AFEMs) is the creation of strategies or \textit{policies} to guide the iterative refinement process.
An ideal policy should balance the competing goals of maximizing error reduction against minimizing growth in number of degrees of freedom.
The modern tools of reinforcement learning have the potential to discover optimal policies in an automated fashion, once a suitable connection between the finite element problem and the reinforcement learning environment has been established.

In this work, we focus on a very simple connection to the reinforcement learning community, based exclusively on the \textsc{mark} step of the traditional AFEM process:
\begin{equation}
\label{SEMR}
	\textsc{solve}~~\rightarrow~~\textsc{estimate}~~\rightarrow~~\textsc{mark}~~\rightarrow~~\textsc{refine}
\end{equation}
We refer to a complete pass through the above sequence as one iteration of the process.
For each iteration, the \textsc{mark} step receives a list of error estimates for each element in a mesh and must produce a list of  elements to mark for $h$-refinement (subdivide geometrically), $p$-refinement (raise the local polynomial approximation order), or de-refinement (undo a previous refinement).
Common practice is to leave parameters that control the \textsc{mark} step fixed, allowing users to focus on analyzing other aspects of the process or application problem.
Here, we treat the selection of parameters controlling the \textsc{mark} step as a decision that can be optimized by reinforcement learning and demonstrate that how such a treatment can improve the overall efficiency and accuracy of the AFEM process.

We motivate the potential gains from such an approach in the heavily studied context of $h$-refinement AFEM for Laplace's equation over an L-shaped domain with Dirichlet boundary conditions defined to match a known singular solution.
A standard marking policy in this setting is to mark all elements whose error estimate is greater than or equal to $\theta$ times the maximum error in the list, where $\theta\in[0,1]$ is a fixed parameter.
In \Cref{fig:motivation}, we show final meshes at the end of an AFEM workflow employing either $\theta=0.1$ or $\theta=0.9$, where the refinement process is stopped once the global error estimate is below $1.07\times 10^{-3}$.
While the meshes have visibly similar refinement patterns, the computational cost is dramatically different: for $\theta=0.1$, the final mesh occurs after 11 iterations and 2441 degrees of freedom (dofs), while for $\theta=0.9$ the final mesh occurs after 36 iterations and 2169 dofs. 
This simple example highlights the trade-off between iteration count and dof count that presents an opportunity for optimization.
Despite the obvious sensitivity of iteration count to the selection of $\theta$, the tuning of $\theta$ to a particular problem setting is often neglected in practice and has not---to the best of our knowledge---been studied directly as an optimization problem.

\begin{figure}
\centering
\begin{tabular}{ccc}
	\includegraphics[width=0.28\textwidth]{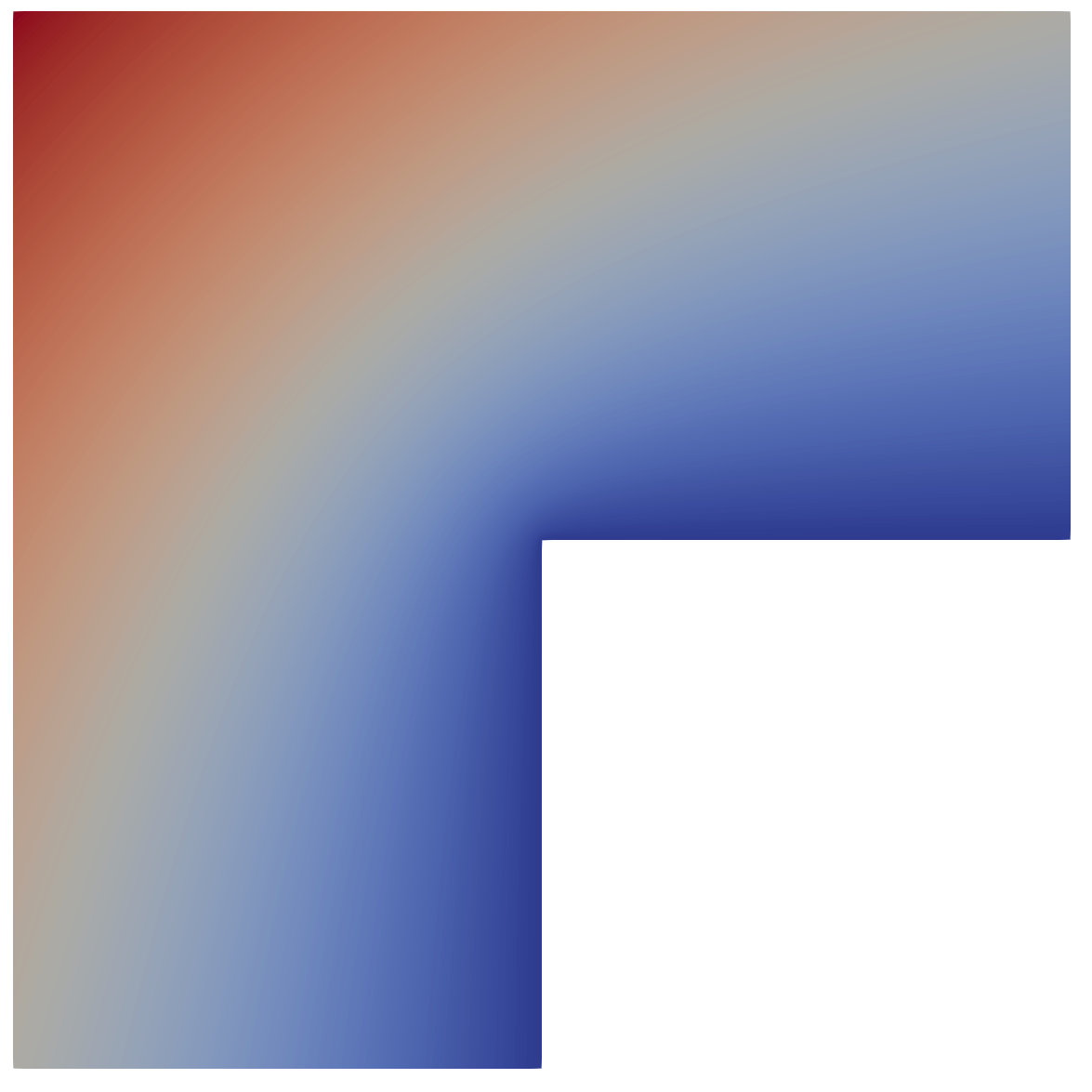} &
	\quad
	\includegraphics[width=0.28\textwidth]{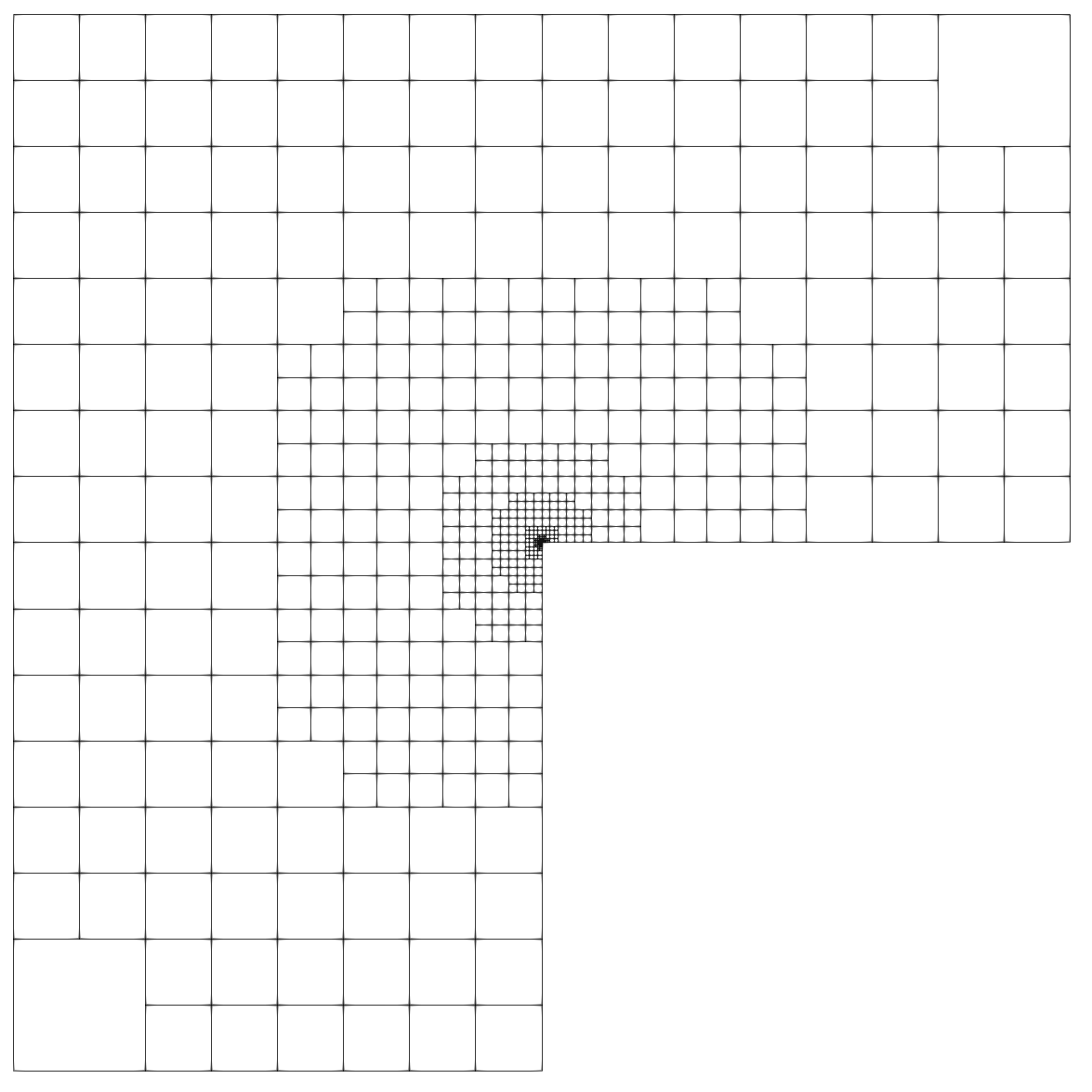}
	\quad
	&
	\includegraphics[width=0.28\textwidth]{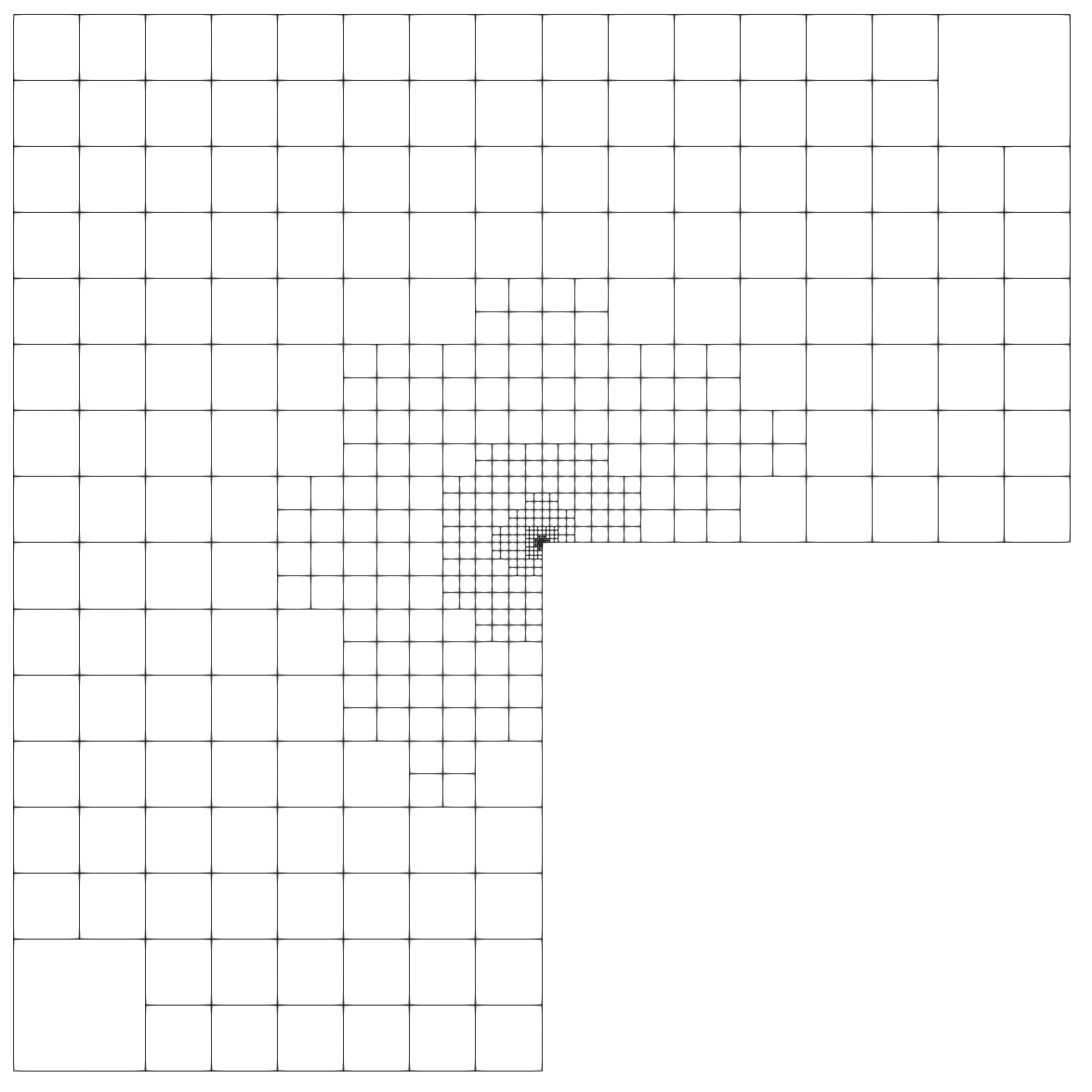}
	\\
	PDE solution & $\theta=0.1$ & $\theta=0.9$
\end{tabular}

	\caption{The solution to Laplace's equation on the L-shaped domain (left) is a classical AFEM test-bed problem.  Parameters in the \textsc{mark} step that are typically fixed by the user in a heuristic fashion can have a dramatic effect on the total computational cost required to achieve a desired accuracy. The two meshes shown correspond to different parameters for the same marking policy, yielding similar global error and similar meshes after either 11 iterations and 2441 dofs (middle) or 36 iterations and 2169 dofs (right). Automating the learning of optimal parameters for the \textsc{mark} step in~\cref{SEMR} is the primary goal of this work.
}
	\label{fig:motivation}
\end{figure} 

The example from \Cref{fig:motivation} highlights an additional axis on which the parameter $\theta$ can be optimized: each individual refinement iteration.
Tuning for a fixed choice of $\theta$ can bring some computational benefits for $h$-refinement, but allowing $\theta$ to change after each iteration opens a much broader landscape for optimization.
By searching over the entire \textit{sequence} of refinement steps used to produce a suitable adaptive mesh, we can search for total computational cost minima and final global error minima as independent objectives.
Furthermore, by introducing an independent parameter $\rho \in [0,1]$ to control the number of $p$-refinements at a given iteration, we can add an additional dimension to the space of refinement processes, enabling an even larger space of optimization possibilities for $hp$-refinement.

The space of possible $h$- or $hp$-refinement processes and the search for optimal decisions in this space is very naturally treated as a reinforcement learning (RL) problem.
At a high level, a marking policy receives \textit{state} information about the distribution of errors after an \textsc{estimate} step, as well as the remaining ``budget'' for searching, and returns an \textit{action} that controls the subsequent \textsc{mark} step.
Pictorially, 
\begin{equation}
\label{SEDMR}
\textsc{solve}~~\rightarrow~~\textsc{estimate}~~\rightarrow~~\fbox{\textsc{decide}}~~\rightarrow~~\textsc{mark}~~\rightarrow~~\textsc{refine}
\end{equation}
We implement the \textsc{decide} step by querying a trained marking policy.
During training, the policy receives a reward based on either how much the error decreased or how few dofs were added, given a particular action.
Training is effectively aiming to induce optimal behavior of a Markov decision process (a type of discrete-time stochastic control process).
Once trained, the policy can be deployed in previously unseen AFEM environments and then compared directly against other marking policies in the literature.

As we will demonstrate, trained policies of this type can produce refinement paths for AFEM processes that are (1) superior to optimal fixed-parameter marking policies on a fixed geometry, (2) robust to changes in the domain geometry, and (3) suitable for 3D geometries, even when the policy is trained exclusively in 2D settings.
\red{We motivate the selection of an observation space consisting of three quantities: a measure of ``budget'' (either the distance to a target error threshold or the cumulative degrees of freedom already employed) and normalized mean and standard deviations of error estimates at the previous step; all of the observables are trivial to compute within any existing AFEM framework.}
To be abundantly clear, we emphasize that marking policies are only involved in parameter selection; marking policies are not meant to interact in any other way with the underlying finite element code.
For the purpose of reproduction and wider adoption, this work is accompanied by an open-source Python-based implementation \cite{Code}.

\subsection{Related work} \label{sub:literature_review}

Machine learning for adaptive mesh refinement (AMR) was first studied in the early 1990s \cite{dyck1992determining}, with only a small number of subsequent contributions \cite{chedid1996automatic,manevitz2005neural} appearing in the literature until a few years ago.
Since then, attention has grown significantly \cite{bao2019data,bohn2021recurrent,chen2020output,pfaff2020learning,zhang2020meshingnet,huang2021machine,paszynski2021deep,chen2021output,chakraborty2021multigoaloriented,roth2022neural,sprave2021evaluating,patel2021smart,yang2021reinforcement,antonietti2022machine,antonietti2022refinement,twumasi2021towards}.
To date, machine learning has been used to design goal-oriented AMR strategies \cite{chen2020output,chen2021output,chakraborty2021multigoaloriented,roth2022neural}, $hp$-refinement strategies \cite{paszynski2021deep}, and refinement strategies for polygonal discretizations \cite{antonietti2022machine,antonietti2022refinement}.

The earliest data-driven approaches to AMR involved learning a mesh density function \cite{dyck1992determining,chedid1996automatic,tan2003proper} or ``areas of interest'' in the physical domain for refinement \cite{manevitz2005neural}.
This class of approaches continues to be actively developed \cite{pfaff2020learning,huang2021machine,zhang2020meshingnet,song2022m2n,chan2022locally}.
We note that density-based AMR typically relies on a down-stream mesh generator (see, e.g., \cite{dobrzynski2008anisotropic}) or some form of mesh transport (see, e.g., \cite{budd2009moving}).
Among other reasons, this makes density-based AMR fundamentally different from marked element AMR, which is the paradigm considered in this work.

The first instance of machine learning for marked element AMR that we are aware of is \cite{bohn2021recurrent}.
In \cite{bohn2021recurrent}, it is shown that error estimation (cf.~\textsc{estimate}) and element marking (cf.~\textsc{mark}) for an elliptic partial differential equation (PDE) can be represented by a recurrent neural network.

Of all of the references above, only \cite{yang2021reinforcement} characterizes AMR as a sequential decision-making problem.
\red{In particular, the authors of \cite{yang2021reinforcement} seek to learn a mesh-dependent map between approximate PDE solutions and elementwise mesh refinement decisions}.
In turn, they formulate AMR as a Markov decision process (MDP) with \emph{variable}-size state and action spaces.
The authors then propose novel policy architectures to support their unique variable-size spaces.
An additional novelty in \cite{yang2021reinforcement} is that the authors do not rely on \emph{a posteriori} error estimators from the finite element literature; \red{see, e.g., \cite{ainsworth2000aposteriori}}.
Instead, they replace the \textsc{estimate} and \textsc{mark} steps in~\cref{SEMR} with an elementwise decision based on a direct view of the solution about each element.
\red{Unfortunately, the approach is limited to refining only one element at a time~\cite[Section 7]{yang2021reinforcement}, which is impractical for most applications.}

\red{Although we also characterize AMR as a sequential decision-making problem, our approach is different from \cite{yang2021reinforcement} in numerous ways: for example, (1) we require the user to provide an \emph{a posteriori} error estimator that can deliver reliable and efficient local error estimates (cf.~\Cref{sub:marked_element_amr}); (2) our policy observes a fixed number of statistics derived from the local error estimates (cf.~\Cref{sub:observables}) and returns a fixed number of bulk refinement parameters (cf.~\Cref{sec:putting_it_all_together}); in turn, (3) we formulate AMR as an MDP with \emph{fixed}-size state and action spaces (i.e., mesh-independent); (4) we rely on policy architectures that are often used in other reinforcement learning applications and are, therefore, easy to construct and train with contemporary software libraries (cf.~\Cref{rem:SoftwareImplementation}); finally, unlike \cite{yang2021reinforcement}, (5) our approach is not limited to refining only one element at a time and, in fact, naturally supports refining any number of elements per refinement step.}
Our contributions, as well as those in \cite{yang2021reinforcement}, align with a recent trend in reinforcement learning to improve adaptive algorithms in scientific computing---for example, adaptive time-stepping for numerical integration \cite{dellnitz2021efficient} and adaptive coarsening for algebraic multigrid methods \cite{taghibakhshi2021optimization}.

\section{Preliminaries} \label{sec:preliminaries}

In this section, we introduce the fundamental concepts and basic notation used throughout the paper.
In particular, we first describe classical concepts from marked element AMR for stationary PDEs.
We then recapitulate marked element AMR as MDP.
Finally, we introduce the concept of \textit{marking policies}, which can be used to control the associated MDP.

\subsection{Marked element AMR} \label{sub:marked_element_amr}

Our focus is on AMR for PDE-based boundary value problems posed on open Lipschitz domains $\Omega\subset\R^d$, where $d=2$ or $3$.
In this work, we let $\mesh$ denote any shape regular \emph{mesh} subordinate to the domain, $\bigcup_{T\in\mesh} \overline{T} = \overline{\Omega}$ and $\bigcap_{T\in\mesh} = \emptyset$, where every \emph{element} $T\in\mesh$ is Lipschitz.

The canonical application scenario begins with an equation of the form
\begin{equation}
\label{eq:CanoncialPDE}
	\mathcal{L} u = f
	\quad \text{in } \Omega,
	\qquad
	u = 0 \quad \text{on } \partial\Omega,
\end{equation}
where $\mathcal{L} \colon V \to V^\prime$ is a bijective differential operator on some Hilbert space $V$ with norm $\|\cdot\|_V$.
A popular method to solve such equations is the finite element method \cite{ciarlet2002finite}.
This method involves constructing a discrete space $V(\mesh) \subset V$ and solving a discrete version of~\cref{eq:CanoncialPDE}:
\begin{equation}
\label{eq:DiscreteCanoncialPDE}
	\text{Find~} u_\mesh \in V(\mesh) \text{~such that}
	\quad
	\langle \mathcal{L} u_\mesh, v\rangle = \langle f, v \rangle
	\quad
	\text{for all~} v \in V(\mesh),
\end{equation}
where $\langle\cdot,\cdot\rangle$ denotes the $V^\prime \times V$ duality pairing.

In most application scenarios, we have one of two complementary objectives:
\smallskip

(i) \textbf{Optimal efficiency.}
Solve~\cref{eq:DiscreteCanoncialPDE} to a prescribed accuracy (e.g., $\|u-u_\mesh\|_V \leq \mathtt{tol}$) as efficiently as possible.
\smallskip

(ii) \textbf{Optimal accuracy.}
Solve~\cref{eq:DiscreteCanoncialPDE} to the optimal accuracy allowed by a prescribed computing budget or time constraint.
\smallskip

\noindent
Objectives (i) and (ii) are difficult to achieve because they involve the solution of an optimization problem in a complex, infinite-dimensional set of possible discretizations \cite{demkowicz1985h,demkowicz2002fully,patra2001systematic}.
Instead of trying to reach optimal efficiency or accuracy as defined by  (i) and (ii), one typically finds a satisfactory solution through an AMR process.
This type of iterative process begins with a user-defined initial mesh $\mesh_0$ that is sequentially refined, generating a sequence of meshes $\mesh_0,\mesh_1,\ldots,\mesh_k$ with improving accuracy.
The process then stops once the computing budget is exhausted or the target accuracy is reached and the convergence rate of the solution error is used to assess its effectiveness \cite{morin2002convergence,stevenson2007optimality,carstensen2014axioms}.
A benefit of the RL approach employed here is that we can directly address (i) and (ii), while still adhering to the traditional marked AMR process described above.

Hereafter, we assume that~\cref{eq:DiscreteCanoncialPDE} admits a global relative error estimator
\begin{equation*}
	\eta_\mesh = \sqrt{\sum_{T\in \mathcal{T}} \eta_T^2} \approx \|u - u_\mesh\|_V/\|u\|_V,
\end{equation*}
where $\eta_T$ denotes the local error estimator applied to a single element $T\in\mathcal{T}$.
\rev{We require a relative error estimator to normalize the errors over trivial length scale changes of the mesh geometry.}
After constructing such an estimator, we must also select a marking rule.
Two common choices are the greedy and D\"orfler marking rules \cite{dorfler1996convergent}.
Both of these choices are parameterized by a bulk parameter $\theta \in [0,1]$ that determines the how conservative the refinement will be ($\theta = 1$ being the most conservative).
In a greedy rule, we refine all elements $T \in\mathcal{T}$ satisfying
\begin{equation}
\label{eq:GreedyMeshRefinement}
      \theta\cdot\max_{S\in\mathcal{T}}\{\eta_{S}\} \leq \eta_T
  .
\end{equation}
In a D\"orfler rule, we refine a minimal cardinality subset of elements $\mathcal{S} \subset \mathcal{T}$ satisfying
\begin{equation}
    \label{eq:DorflerMeshRefinement}
  \theta\cdot\sum_{T\in \mathcal{T}} \eta_T^2
  \leq
  \sum_{S\in \mathcal{S}} \eta_S^2
  .
\end{equation}

We now state the standard marked element AMR algorithm for achieving a target error estimate $\eta_\infty > 0$; cf. objective (i).
Here and onward, we denote $\eta_k = \eta_{\mesh_k}$ and $u_k = u_{\mesh_k}$.
An example follows immediately afterward.

\begin{algorithm2e}[H]
\DontPrintSemicolon
	\caption{\label{alg:BasicAMR} Marked element AMR with a target error estimate.}
	\SetKwInOut{Input}{input}
	\SetKwInOut{Output}{output}
	\SetKw{Break}{break}
	\Input{Initial mesh $\mesh_0$, fixed parameter $\theta \in (0,1)$, target error estimate $\eta_\infty > 0$.}
	\Output{Discrete solution $u_k$.}
	$k \leftarrow 0$.\;
	\While{$\eta_k > \eta_\infty$}
	{
		Solve~\cref{eq:DiscreteCanoncialPDE} with $\mesh = \mesh_k$. \tcp*[l]{\textsc{solve}~~~~~~}
		Compute error estimates $\{\eta_T\}_{T\in\mesh_k}$. \tcp*[l]{\textsc{estimate}}
		Mark all $T\in\mesh_k$ satisfying~\cref{eq:GreedyMeshRefinement} (or~\cref{eq:DorflerMeshRefinement}). \tcp*[l]{\textsc{mark}}
		Form $\mesh_{k+1}$ by refining all marked elements in $\mesh_k$. \tcp*[l]{\textsc{refine}}
		$k \leftarrow k+1$.\;
	}
\end{algorithm2e}

\begin{example}
\label{ex:ZZ}
Assume that $\mesh$ is made up of simplices and, for every $T\in\mesh$, denote the space of polynomials of degree less than or equal to $p$ by $P_{p}(T)$.
We may now consider the classical order-$p$ finite element variational formulation of the Poisson equation $-\Delta u = f$ with homogeneous Dirichlet boundary conditions:
\begin{equation}
\label{eq:FEMPoisson}
	\text{Find~} u_\mesh \in V_{p}(\mesh) \text{~such that}
	\quad
	\int_\Omega \nabla u_\mesh\cdot\nabla v \dd x = \int_\Omega f v \dd x
		\quad
	\text{for all~} v \in V_{p}(\mesh),
\end{equation}
where $V_{p}(\mesh) = \{ v \in H^1_0(\Omega) \,\colon\, v|_T \in P_{p}(T)~ \forall\, T\in\mesh\}$.
In all numerical experiments in this work, we utilize the Zienkiewicz--Zhu error estimator \cite{zienkiewicz1992superconvergent1,zienkiewicz1992superconvergent2} for~\cref{eq:FEMPoisson}.
However, there are numerous other equally well-qualified candidates in the literature that one may use instead \cite{ainsworth2000aposteriori}.
\end{example}

\subsection{AMR as a Markov process} \label{sub:episodic_environments}

The remainder of this section is centered on characterizing an optimal marking policy $\pi$ through which a new value of $\theta$ in~\cref{eq:GreedyMeshRefinement} (or~\cref{eq:DorflerMeshRefinement}) can be selected after every AMR iteration.
The key to our approach is to identify every \textsc{solve}--\textsc{estimate}--\textsc{mark}--\textsc{refine} (SEMR) iteration of \Cref{alg:BasicAMR} (cf.~\cref{SEMR}) with the transition of an unknown \emph{Markov process} that advances the current state of the discretization to a new state with a transition probability dependent on $\theta$.
This stochastic perspective is appealing in part because it allows us to characterize a \emph{robust} marking policy that performs well on a distribution of PDEs (cf.~\Cref{ssub:ex2c,ssub:ex_transfer}).

The SEMR process in \Cref{alg:BasicAMR} can be steered to even more computationally efficient results by modifying the choice of $\theta$ at each iteration $k$.
Doing so is equivalent to enacting an ``Adaptive Marking AMR'' process, which we will denote by (AM)$^2$R.
The method proceeds via a \textsc{solve}--\textsc{estimate}--\textsc{decide}--\textsc{mark}--\textsc{refine} (SEDMR) loop that we will now construct (cf.~\cref{SEDMR} and~\Cref{fig:flowchart}).
Since we are permitted a parameter decision before each state transition (i.e., refinement), SEDMR is an example of a discrete-time stochastic control process called a \emph{Markov decision process} (MDP) \cite{sutton2018reinforcement}.

\subsection{Marking policies as probability distributions} \label{sub:refinement_policies}

A marking policy can be any map between a set of system observables $\mathcal{O}$ and refinement actions $\mathcal{A}$.
However, experience from the reinforcement learning literature indicates several advantages of defining the observable-to-action map through a probability distribution \cite{sutton2018reinforcement}.
In turn, we define a \textit{marking policy} to be a family of probability distributions $\pi\colon \mathcal{O} \times \mathcal{A} \to [0,1]$ from which we can sample the bulk parameter $\theta \sim \pi(\theta|o)$ for any state observation $o\in\mathcal{O}$.

The most important reason to define a marking policy as a probability distribution is that it enables the use of stochastic algorithms for optimizing the associated MDP; cf.~\Cref{sec:putting_it_all_together}.
Furthermore, the distribution perspective provides a way to encode multivalued maps between $\mathcal{O}$ and $\mathcal{A}$, which are helpful when the observation space is not rich enough to encode all state information.

In the context of marking rules like \cref{eq:DorflerMeshRefinement,eq:GreedyMeshRefinement}, the natural \emph{action space} $\mathcal{A}$ for AMR is the set of all admissible bulk parameters $\theta$; that is, $\mathcal{A} = [0,1]$.
Unlike this clearly defined action space, we are free to incorporate any features we deem important to the PDE 
discretization into the definition of the \emph{observation space} $\mathcal{O}$.
For example, any subset of the physical or geometry parameters of the underlying PDE could be used to define $\mathcal{O}$.
Such a choice may be helpful to arrive at a policy appropriate for a range of PDEs.
In this work, we focus on a more generic observation space derived solely from the local error estimates and refinement objective.
In order to focus now on more general aspects of the policy $\pi$, we defer the precise description of our observation space to \Cref{sub:observables}.

\subsection{From problem statements to algorithms} \label{sub:optimality}

By changing the value of $\theta = \theta_k$ within every SEDMR iteration $k$, we seek to induce a doubly adaptive refinement process that is optimal with respect to some prescribed objective function.
In problems (i) and (ii), the objective is to minimize either the final error estimate $\eta_k$ or some surrogate of the total simulation cost $J_k$.
We rewrite these problems as follows, \smallskip

\textbf{Efficiency problem.}
Given the target error estimate $\eta_\infty>0$, seek
\begin{equation}
\label{eq:err_thresh}
	\min_{\pi,k}~ \bbE_{\theta\sim\pi}\big[ \log_2 J_k\big]
	\quad
	\text{subject to~}
	\eta_k \leq \eta_\infty
	~\text{a.s.}
	\end{equation}

\textbf{Accuracy problem.}
Given the computational budget $J_\infty>0$, seek
\begin{equation}
\label{eq:dof_thresh}
	\min_{\pi,k}~ \bbE_{\theta\sim\pi}\big[\log_2 \eta_k\big]
	\quad
	\text{subject to~}
	J_k \leq J_\infty
	~\text{a.s.}
	\end{equation}
In the two problems above, $\bbE_{\theta\sim\pi}[X(\theta)]$ denotes the expected value of the random variable $X(\theta)$ when $\theta$ is drawn from $\pi$ and a.s. stands for ``almost surely'' with respect to the probability measure defining the expected value.

We are free to define $J_k$ as we choose; however, one basic principle is that the cost should depend on the entire refinement history, $J_k = J(\mesh_0,\ldots,\mesh_k)$.
The most direct measures of cost may be the cumulative time-to-solution or the cumulative energy expended.
Both of these are stochastic quantities that are difficult to estimate precisely.
In this work, we use the cumulative degrees of freedom (dofs) to define the cost function.
More precisely,
\begin{equation}
	J_k
	:=
	\sum_{i=0}^k \ndofs(\mesh_i)
	,
\end{equation}
where $\ndofs(\mesh)$ is the number of dofs in the discrete space $V(\mesh)$.
If the PDE solver scales linearly with number of dofs and the overhead costs of assembly and refinement are neglected, we believe this is a reasonable surrogate for simulation cost.
Future work may be devoted to designing optimal policies based on other cost functions.

\Cref{alg:AM2R_1,alg:AM2R_2} describe the (AM)$^2$R process, as applied to the efficiency problem and accuracy problem, respectively.
In~\Cref{sec:putting_it_all_together}, we describe how to optimize for the scalar outputs of these process; i.e., the cost $J_k$ or final global error estimate $\eta_k$.

\begin{algorithm2e}
\DontPrintSemicolon
	\caption{\label{alg:AM2R_1} (AM)$^2$R with a target error estimate.}
	\SetKwInOut{Input}{input}
	\SetKwInOut{Output}{output}
	\SetKw{Break}{break}
	\Input{Initial mesh $\mesh_0$, marking policy $\pi$, target error estimate $\eta_\infty > 0$.}
	\Output{Discrete solution $u_k$, cost $J_k$.}
	$k \leftarrow 0$.\;
	\While{$\eta_k > \eta_\infty$}
	{
		Solve~\cref{eq:DiscreteCanoncialPDE} with $\mesh = \mesh_k$. \tcp*[l]{\textsc{solve}~~~~~~}
		Compute error estimates $\{\eta_T\}_{T\in\mesh_k}$. \tcp*[l]{\textsc{estimate}}
		Sample $\theta_k \sim \pi(\cdot|o_k)$. \tcp*[l]{\textsc{decide}}
		Mark all $T\in\mesh_k$ satisfying~\cref{eq:GreedyMeshRefinement} or~\cref{eq:DorflerMeshRefinement} with $\theta = \theta_k$. \tcp*[l]{\textsc{mark}}
		Form $\mesh_{k+1}$ by refining all marked elements in $\mesh_k$. \tcp*[l]{\textsc{refine}}
		$k \leftarrow k+1$.\;
	}
\end{algorithm2e}

\begin{algorithm2e}
\DontPrintSemicolon
	\caption{\label{alg:AM2R_2} (AM)$^2$R with a computational budget constraint.}
	\SetKwInOut{Input}{input}
	\SetKwInOut{Output}{output}
	\SetKw{Break}{break}
	\Input{Initial mesh $\mesh_0$, marking policy $\pi$, computational budget $J_\infty > 0$.}
	\Output{Discrete solution $u_k$, error estimate $\eta_k$.}
	$k \leftarrow 0$.\;
	\While{$J_k < J_\infty$}
	{
		Solve~\cref{eq:DiscreteCanoncialPDE} with $\mesh = \mesh_k$. \tcp*[l]{\textsc{solve}~~~~~~}
		Compute error estimates $\{\eta_T\}_{T\in\mesh_k}$. \tcp*[l]{\textsc{estimate}}
		Sample $\theta_k \sim \pi(\cdot|o_k)$. \tcp*[l]{\textsc{decide}}
		Mark all $T\in\mesh_k$ satisfying~\cref{eq:GreedyMeshRefinement} or~\cref{eq:DorflerMeshRefinement} with $\theta = \theta_k$. \tcp*[l]{\textsc{mark}}
		Form $\mesh_{k+1}$ by refining all marked elements in $\mesh_k$. \tcp*[l]{\textsc{refine}}
		$k \leftarrow k+1$.\;
	}
\end{algorithm2e}

\subsection[Defining an observation space for h-refinement]{Defining an observation space for $h$-refinement} \label{sub:observables}

When designing a marking policy, it is limiting to focus only on optimality over the problems seen during training.
In the context of PDE discretizations especially, it is more useful to attain a generalizable or robust policy that can provide competitive performance on problems outside the training set.
To allow such generalizability, the observation space $\mathcal{O}$ must be defined so that the policy can be applied to a large category of target problems.
This may preclude defining $\mathcal{O}$ with, e.g., geometric features of the domain such as values at interior angles or control points because, in that case, the trained policy could not be applied to domains that do not have an analogous feature set.

In this work, we choose to define $\mathcal{O}$ using only variables that appear in~\Cref{alg:AM2R_1,alg:AM2R_2}.
More specifically, we define the observation space in terms of the individual error estimates $\{\eta_T\}_{T\in\mesh_k}$ and the cost $J_k$.
As a first pass, one might consider including all possible lists of error estimates $\eta_T$ in the observation space, but this presents an immediate challenge due to the varying length of such lists over the course of the AFEM MDP.
Instead, we choose to observe statistics derived from the local error estimates.
In addition, we choose to observe the proximity to a target global error estimate or cumulative degree of freedom threshold.

The proximity observables are easy to define.
In~\Cref{alg:AM2R_1}, the loop ends when the target error is reached.
Therefore, in order to keep track of how far we are from the end of the refinement process when solving the efficiency problem~\cref{eq:err_thresh}, we include the relative distance to the target error,
\begin{equation}
\label{eq:EfficiencyBudget}
	b_k = {\eta_\infty}/{\eta_k},
	\end{equation}
in the observation space.
Alternatively, in~\Cref{alg:AM2R_2}, the loop ends when the computational budget is exhausted.
Therefore, when solving the accuracy problem~\cref{eq:dof_thresh}, we include the relative budget,
\begin{equation}
\label{eq:AccuracyBudget}
	b_k = {J_k}/{J_\infty},
	\end{equation}
in the observation space.

The statistics of $\eta_T$ that we choose to observe are more complicated to define and the remainder of this subsection is devoted to motivating our choices.
We begin by defining the \textit{empirical mean} of any element-indexed set $\{x_T\in\R~\colon~T\in\mesh_k\}$, written
\begin{equation}
	\mean_k[x_T] = \frac{1}{|\mesh_k|} \sum_{T\in\mesh_k} x_T,
\end{equation}
where $|\mesh_k|$ denotes the number of elements in the mesh $\mesh_k$.
The corresponding \textit{empirical variance} is defined as
\begin{equation}
\label{eq:VarianceDefinition}
	\variance_k[x_T]
	=
	\mean_k\big[(x_T - \mean_k[x_T])^2\big],
		\end{equation}
and, in turn, the \textit{empirical standard deviation} is defined
$\sd_k[x_T] = \sqrt{\variance_k[x_T]}$.
Finally, we define the \textit{root mean square} of $\{x_T\}$ to be
\begin{equation}
	\rms_k[x_T]
 	=
 	\sqrt{\mean_k[x_T^2]}
 	.
\end{equation}
Note that one may rewrite $\variance_k[x_T] = \mean_k[x_T^2] - \mean_k[x_T]^2$ and thus see that
\begin{equation}
\label{eq:BoundOnSD}
	\sd_k[x_T] \leq \rms_k[x_T]
	,
\end{equation}
with equality if and only if $\mean_k[x_T] = 0$.

The main challenge of defining an appropriate statistic of $\eta_T$ is ensuring that the statistic is properly normalized with respect to the number of elements and degrees of freedom in the discretization.
To illustrate this challenge, we consider the error in~\cref{eq:FEMPoisson}, namely $u-u_k$.
A straightforward computation shows that
\begin{equation}
\label{eq:Normalization1}
	\mean_k\big[\|\nabla(u-u_k)\|_{L^2(T)}^2\big]
	=
	\frac{1}{|\mesh_k|}\|\nabla(u-u_k)\|_{L^2(\Omega)}^2
		.
\end{equation}
If $u$ is sufficiently smooth and mild assumptions on the problem context are satisfied~\cite{lin2014lower}, then for a uniform $p$-order finite element discretization undergoing uniform $h$-refinements, there exist constants $C_0$ and $C_1$ depending on $u$ but independent of $\mathcal{T}_k$ such that
\begin{equation}
\label{eq:Normalization2}
	C_0 \ndofs(\mesh_k)^{-p/d} \leq \|\nabla(u-u_k)\|_{L^2(\Omega)} \leq C_1 \ndofs(\mesh_k)^{-p/d}.
\end{equation}
Together,~\cref{eq:Normalization1,eq:Normalization2} deliver uniform upper and lower bounds on the root mean square of a normalized distribution of local errors, i.e., for any mesh $\mathcal{T}_k$,
\begin{equation}
\label{eq:UniformBounds}
	C_0 \leq \rms_k[ \widetilde{e}_T] \leq C_1
		,
\end{equation}
where $\widetilde{e}_T = |\mesh_k|^{1/2} \ndofs(\mesh_k)^{p/d} \|\nabla(u-u_k)\|_{L^2(T)}$.
Typically, we do not have access to the true local errors.
However, one may derive similar uniform bounds on the error estimates given the assumption $\eta_k \approx \|\nabla(u-u_k)\|_{L^2(\Omega)}/\|\nabla u_k\|_{L^2(\Omega)}$.
This leads us to consider the following normalized local error estimates:
\begin{equation}
\label{eq:NormalizedErrorEstimates}
	\widetilde{\eta}_T
	=
		|\mesh_k|^{1/2} \ndofs(\mesh_k)^{p/d} \eta_T
	.
\end{equation}

It is instructive to reflect on~\cref{eq:UniformBounds} and see that, if the error estimate converges optimally (cf.~\cref{eq:Normalization2}), then the root mean square of $\widetilde{\eta}_T$ remains bounded.
Under the same assumption, the standard deviation of $\widetilde{\eta}_T$ is bounded due to~\cref{eq:BoundOnSD}.
This observation is summarized in~\Cref{prop:VarianceOfEta}.

\begin{proposition}
\label{prop:VarianceOfEta}
	If there exists a constant $C$ such that
	\begin{equation}
	\label{eq:UpperBoundAssumption}
		\eta_k
		\leq
		C \ndofs(\mesh_k)^{-p/d} 
	\end{equation}
	then for all $T$ in $\mathcal{T}_k$
	\begin{equation}
	\label{eq:BoundedVariance}
		\sd_k[\widetilde{\eta}_T]
		\leq
		\rms_k[\widetilde{\eta}_T]
		\leq
		C
		.
	\end{equation}
\end{proposition}
\begin{proof}
	The first inequality in~\cref{eq:BoundedVariance} is an immediate consequence of~\cref{eq:BoundOnSD}.
	The second inequality follows from the straightforward identity $|\mesh_k|^{1/2}\rms_k[\eta_T] = \eta_k$, assumption~\cref{eq:UpperBoundAssumption}, and definition~\cref{eq:NormalizedErrorEstimates}.
\end{proof}

If the global error estimate does not converge optimally, then neither the standard deviation nor the root mean square of $\widetilde{\eta}_T$ is guaranteed to be bounded, as we now demonstrate by example.
In~\Cref{fig:error-dist-href}, the empirical distribution of $\widetilde{\eta}_T$ is plotted for four discretized model Poisson problems undergoing $h$-refinement.
In the middle row of plots, the discretizations are enriched through uniform $h$-refinement.
When the solution is infinitely smooth, $u \in C^\infty(\overline{\Omega})$, we witness that the distribution of $\widetilde{\eta}_T$ converges after only a few refinements.
However, when the solution has only finite regularity---as in the canonical singular solution on the L-shaped domain, where $u \in H^{s}(\Omega)$, $1\leq s < 3/2$, $u \not\in H^{3/2}(\Omega)$---only the median of $\widetilde{\eta}_T$ appears to converge while the mean and variance diverge exponentially.
In contrast, the mean and variance remain bounded under AMR for both regularity scenarios, as evidenced by the bottom row of~\Cref{fig:error-dist-href}, which employs $h$-refinement via~\Cref{alg:BasicAMR} using the greedy marking strategy~\cref{eq:GreedyMeshRefinement} and $\theta = 0.5$.

A heuristic interpretation of the diverging local error distribution in the uniform refinement, L-shaped domain case from~\Cref{fig:error-dist-href} is found through the concept of error equidistribution \cite{demkowicz2002fully}, which the standard deviation of $\widetilde{\eta}_T$ allows us to quantify.
In some sense, for an ``ideal'' mesh, all elements will have precisely the same local error.
This is equivalent to having zero empirical variance in the local error distribution.
On the other hand, when the local errors vary wildly, the variance of $\widetilde{\eta}_T$ will be accordingly large.
Because uniform refinement is suboptimal when $u$ is singular \cite{dorfler1996convergent}, the local errors become less equally distributed after every refinement.
In other words, this suboptimal refinement process \textit{causes} the variance of $\widetilde{\eta}_T$ to grow with $k$.

\begin{figure}
\centering
\begin{tabular}{lcc}
{\rotatebox[origin=l]{90}{~~~Initial mesh}}
				& \includegraphics[width=.25\linewidth]{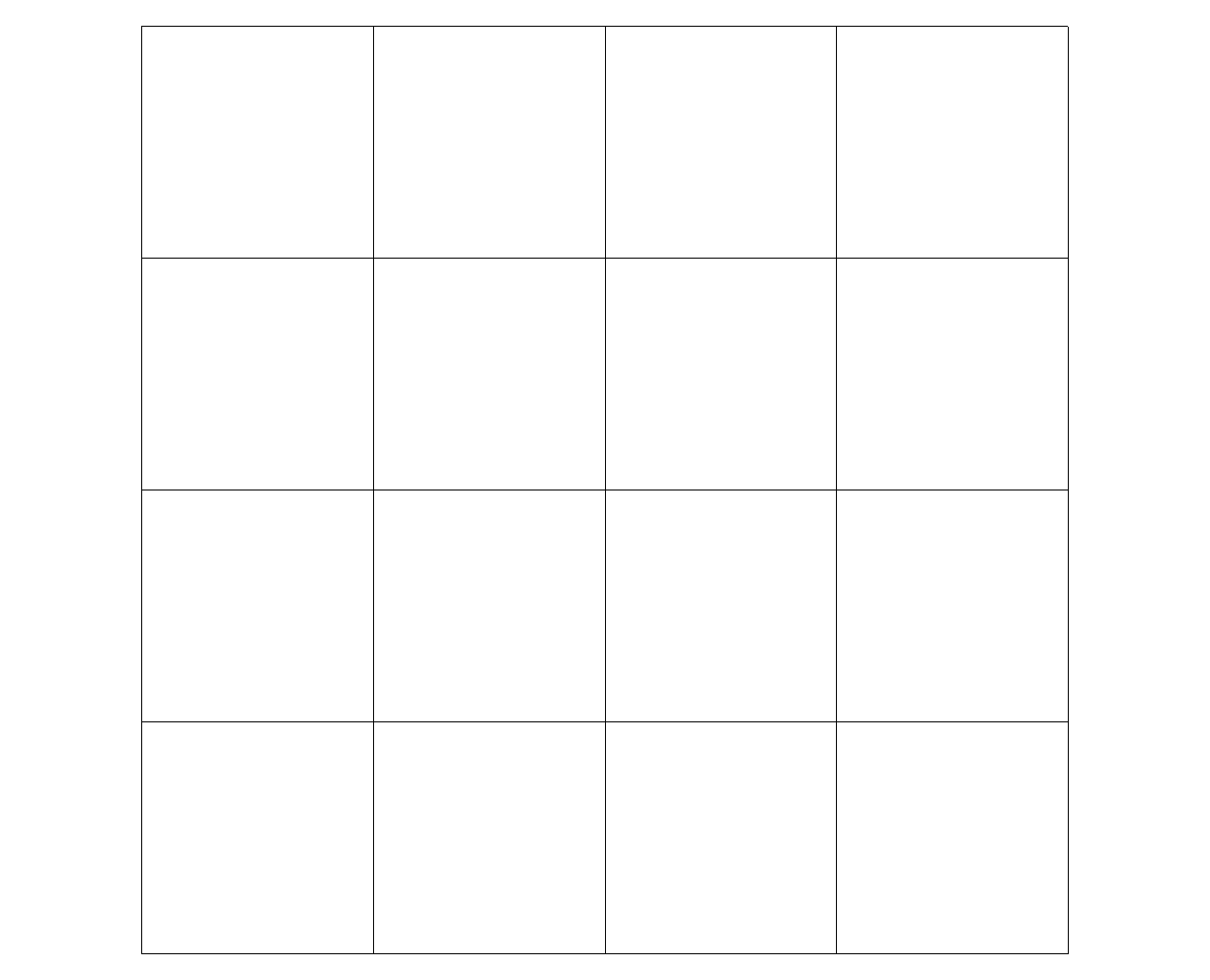} 
	& \includegraphics[width=.25\linewidth]{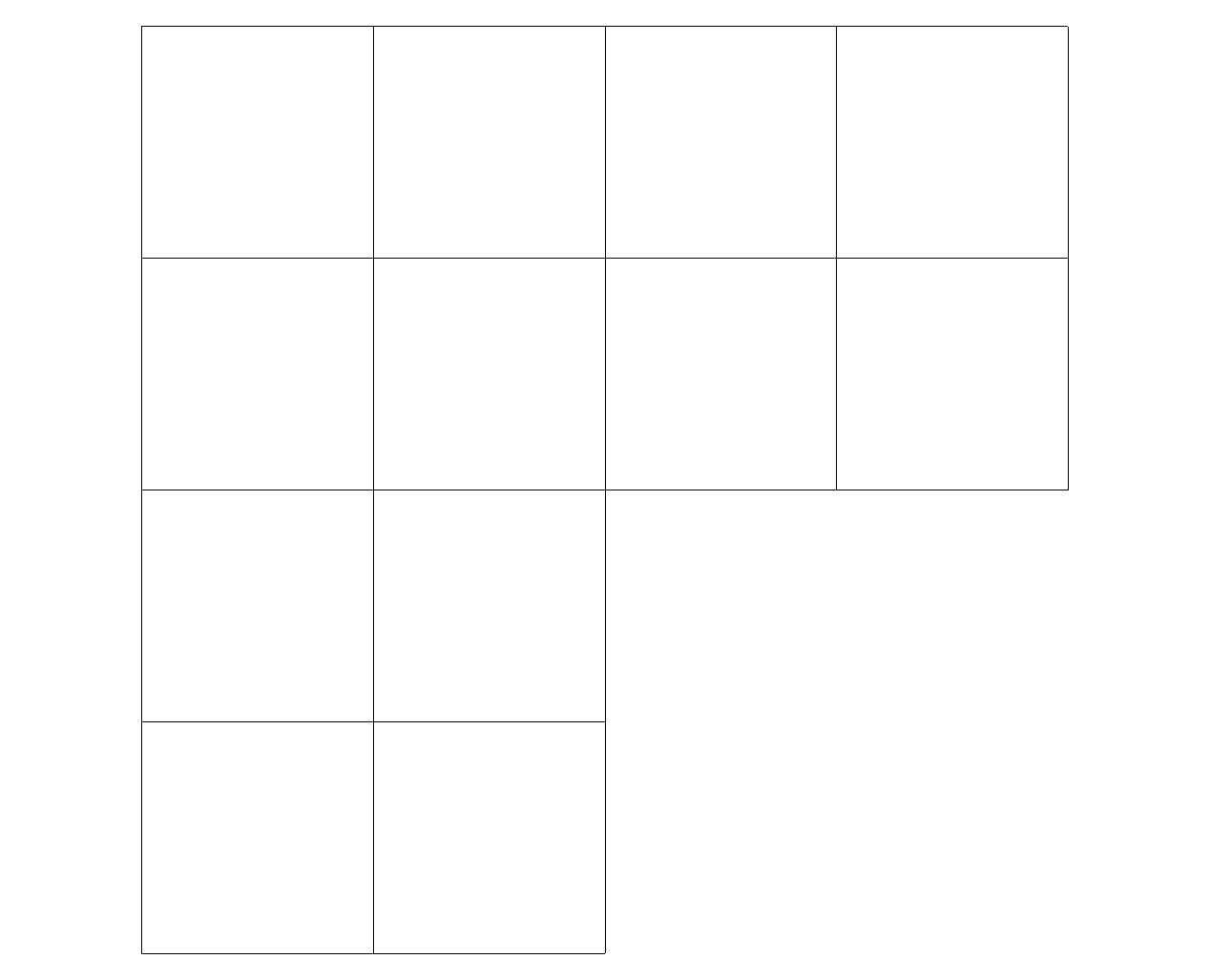} \\
{\rotatebox[origin=l]{90}{~~~Uniform $h$-ref}}
	& \includegraphics[width=.32\linewidth]{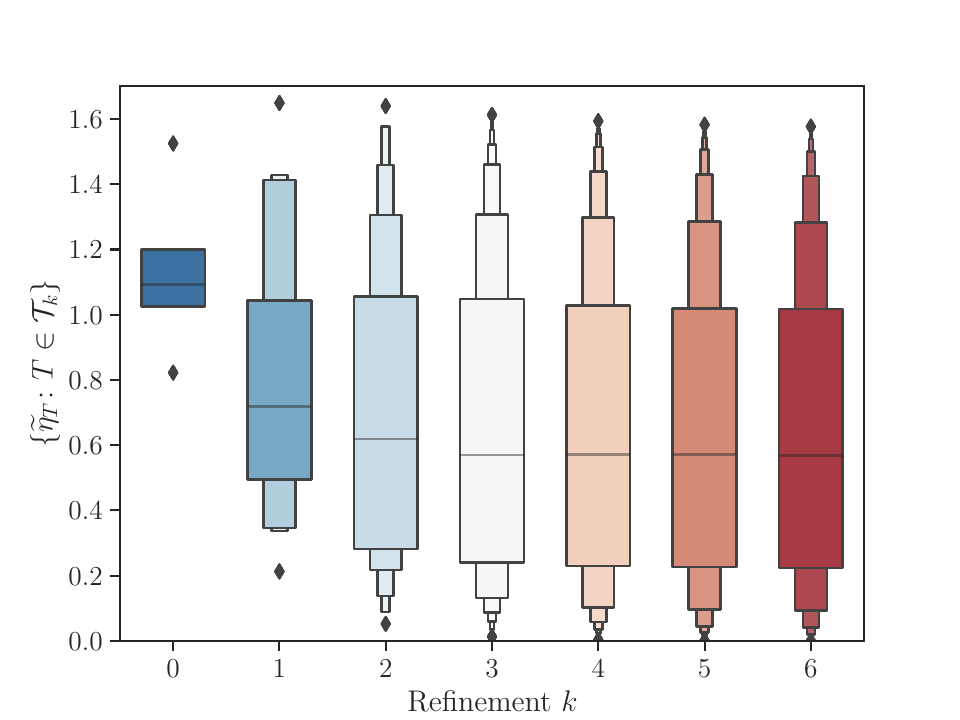}
	& \includegraphics[width=.32\linewidth]{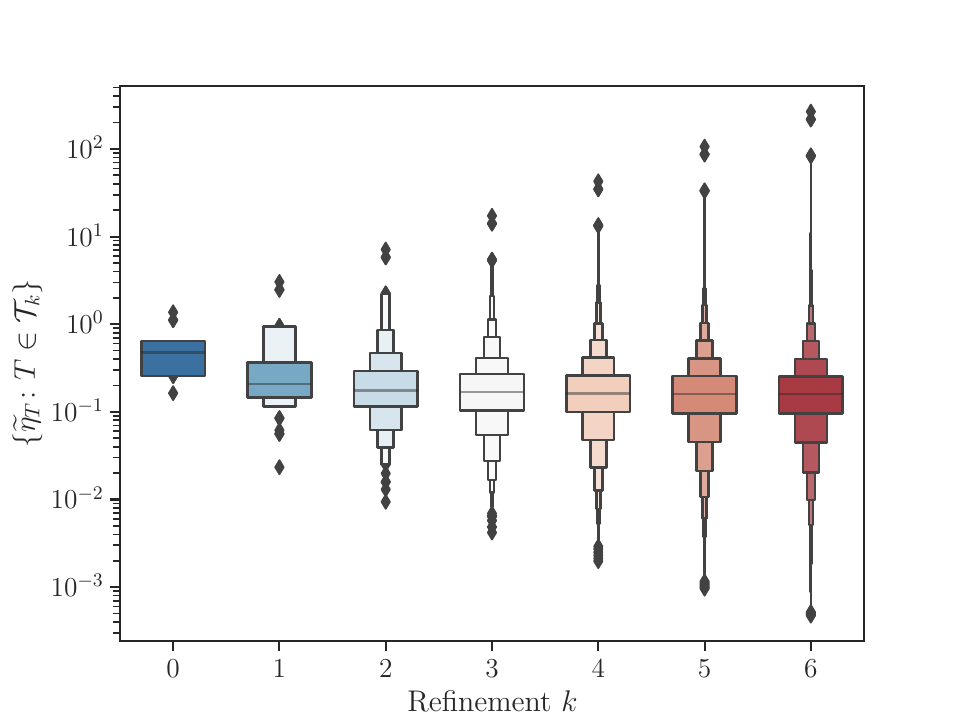} \\
{\rotatebox[origin=l]{90}{~~~Adaptive $h$-ref}}
	& \includegraphics[width=.32\linewidth]{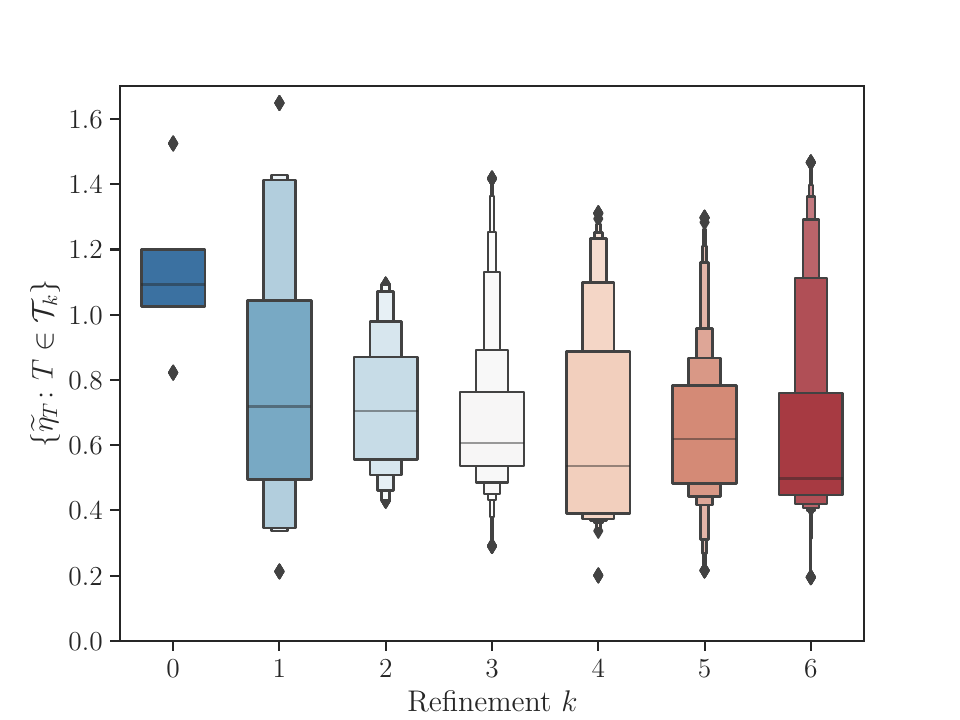} 
	& \includegraphics[width=.32\linewidth]{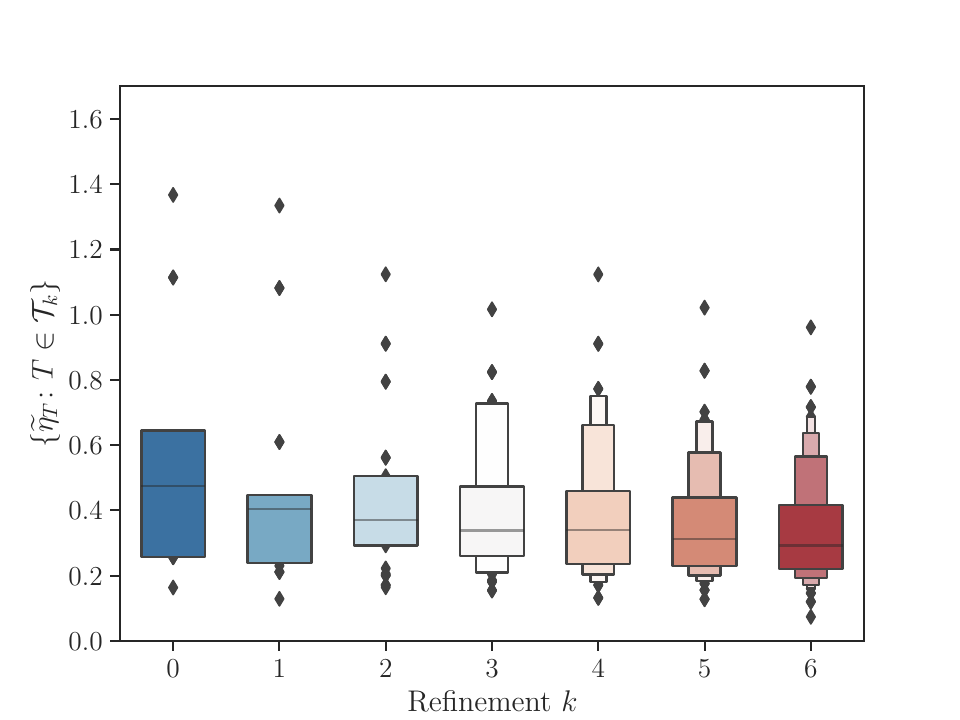} \\[1mm]
	& $u=\sin(\pi x)\sin(\pi y)$ & $u=r^{2/3}\sin(2\theta/3)$
	\end{tabular}
	\caption{\label{fig:Distribution} Empirical distributions of the normalized local error estimates $\widetilde{\eta}_T$ \cref{eq:NormalizedErrorEstimates} under uniform $h$-refinement (middle row) and standard adaptive $h$-refinement using \Cref{alg:BasicAMR} with $\theta = 0.5$ (bottom row) for two Poisson problems (cf.~\cref{eq:FEMPoisson}) with $p=1$.
	\textit{Left column:} The difference in distributions between uniform and adaptive $h$-refinement for the infinitely smooth solution $u=\sin(\pi x)\sin(\pi y)$ over the unit square is noticeable but not significant. 
	\textit{Right column:} For the canonical L-shaped domain problem, uniform $h$-refinement results in exponentially divergent mean and variance, while AMR keeps both controlled.
	This observation guides the design of our RL training regimen for the \textsc{decide} step.
}
\label{fig:error-dist-href}
\end{figure}

We can now formulate our first observation space (an alternative observation space for $hp$-refinement is proposed in~\Cref{sec:extension_to_hp-refinement}).
As motivated previously, there is an intuitive reason to observe the proximity variable $b_k$ corresponding to either~\cref{eq:EfficiencyBudget} or~\cref{eq:AccuracyBudget}, depending on if the efficiency problem or the accuracy problem is being solved, respectively.
Likewise, it is convenient to observe the root mean square of $\widetilde{\eta}_T$, due to its connection to the convergence constants $C_0$ and $C_1$ in~\cref{eq:UniformBounds}, and the standard deviation of $\widetilde{\eta}_T$, due to its connection to the error equidistribution.
Therefore, we choose to observe some combination of $b_k$, $\rms_k[\widetilde{\eta}_T]$, and $\sd_k[\widetilde{\eta}_T]$.

In our experiments, we tried multiple combinations of these variables but settled on the following formulation: $\mathcal{O}_h = [0,1] \times [0,\infty) \times [0,\infty)$, where each $o_k \in \mathcal{O}_h$ is defined by
\begin{equation}
\label{eq:hobsspace}
	o_k
	=
	\left(~b_k,~~\log_2(1 + \rms_k[\widetilde{\eta}_T]),~~\log_2(1 + \sd_k[\widetilde{\eta}_T])~\right)
	.
\end{equation}
The logarithms in the second and third components of $o_k$ exist for numerical stability.
Recalling the fact that $\rms_k[\widetilde{\eta}_T]$ and $\sd_k[\widetilde{\eta}_T]$ may diverge exponentially (cf.~\Cref{fig:error-dist-href}), we found it more numerically stable to observe logarithms of $\rms_k[\widetilde{\eta}_T]$ and $\sd_k[\widetilde{\eta}_T]$ rather than their direct values.
We admit our definition of $\mathcal{O}_h$ is ad hoc and encourage future research to explore the benefits of other combinations of these or other variables.

\section{Putting it all together} \label{sec:putting_it_all_together}

In the previous section, we characterized marked element AMR as an MDP in which the value of the refinement parameter $\theta$ can be chosen by querying a marking policy $\pi(\theta|o)$ that depends on the current refinement state, distinguished by an observable $o\in\mathcal{O}$.
We then motivated a specific observation space~\cref{eq:hobsspace} intended for both the efficiency problem~\cref{eq:err_thresh} and the accuracy problem~\cref{eq:dof_thresh}.

It remains to formulate a methodology to solve the corresponding optimization problems.
The first step is to define a statistical model for the policy $\pi(\theta|o)$.
Experience has shown that projecting a Gaussian model whose mean $\mu = \mu(o)$ and standard deviation $\sigma = \sigma(o)$ are parameterized by a feed-forward neural network works well on this type of problem \cite{sutton2018reinforcement}.
In other words, the policy $\pi(\theta|o)$ is sampled by projecting normally distributed samples $\tilde{\theta}\sim \tilde{\pi}(\tilde{\theta}|o)$ onto the interval $[0,1]$:
\begin{equation}
	\theta = \max\{0,\min\{1,\tilde{\theta}\}\}
	.
\end{equation}
The family of Gaussian probability density functions is written
\begin{equation}
	\tilde{\pi}(\tilde{\theta}|o)
	=
	\frac{1}{\sigma(o)\sqrt{2\pi}}
	\exp
	\bigg\{
		-\frac{1}{2}\bigg(\frac{\tilde{\theta} - \mu(o)}{\sigma(o)}\bigg)^2
	\bigg\}
	,
\end{equation}
where $(\mu(o),\ln(\sigma(o))) = z_{L}(o)$ and
\begin{equation}
	z_{\ell+1}(o)
	=
	W_L\phi(z_{\ell}(o)) + b_\ell
	,
	\quad
	1\leq \ell \leq L
	,
\end{equation}
starting with $z_1(o) = W_1 o + b_1$.
Here, $W_\ell\in\mathbb{R}^{n_\ell\times n_{\ell-1}}$ is the weight matrix, $b\in\mathbb{R}^{n_\ell}$ is the bias vector in the $\ell$th layer ($n_0 = \mathrm{dim}(\mathcal{O})$, and $n_{L} = 2$), and $\phi\colon\mathbb{R}\to\mathbb{R}$ is a nonlinear activation function applied elementwise to its argument.
The execution of the resulting (AM)$^2$R process is described in the flowchart in~\Cref{fig:flowchart}.

\begin{figure}
\centering
	\includegraphics[width=\textwidth]{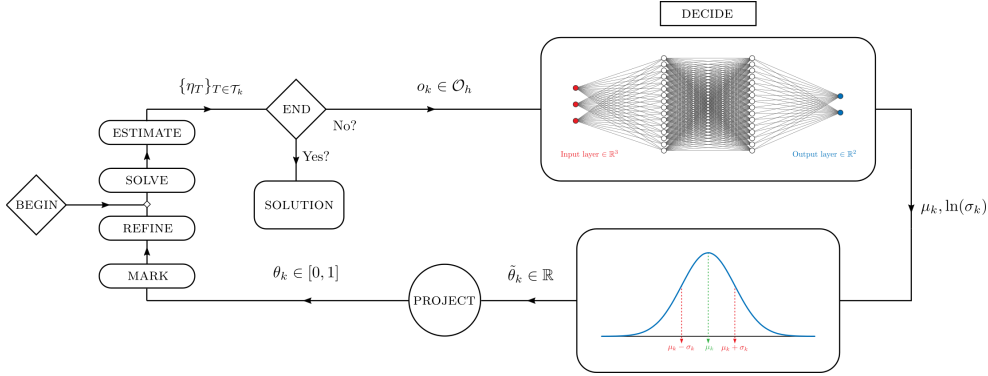}
		\caption{Flowchart describing the \rev{pipeline for training a marking policy. When deployed, the policy always sets~$\tilde\theta_k$ to $\mu_k$, making the trained policy a deterministic function of the observation space.}}
	\label{fig:flowchart}
\end{figure}

With the substitutions above, a trained policy is found by optimizing~\cref{eq:err_thresh} or~\cref{eq:dof_thresh} over all weight matrices $W_\ell$ and bias vectors $b_\ell$.
To solve the corresponding finite-dimensional optimization problem, we employ a particular type of reinforcement learning algorithm called \textit{proximal policy optimization} (PPO)~\cite{schulman2017proximal}.
PPO methods fall under the broader class of \textit{policy gradient methods}.
Supporting technical details on policy gradient methods can be found in~\cite{schulman2015trust,silver2014deterministic,sutton1999policy} and the textbook of Sutton and Barto~\cite{sutton2018reinforcement}.
We also refer the interested reader to our open-source software implementation \cite{Code}.

\rev{Finally, we note that while training a policy involves random draws from the Gaussian density functions, deploying the policy does not.  When deployed, the policy selects $\tilde\theta_k$ to be $\mu_k$, making it a deterministic function of the observation space.  A traditional fixed-parameter marking rule can thus be thought of as a constant-valued policy with the constant chosen heuristically, rather than by training on representative problems.}

\begin{remark}
\label{rem:SoftwareImplementation}
Our software implementation is based on the open-source C++ finite element software library \mfem~\cite{mfem,mfem-web} and the open-source Python-based reinforcement learning toolbox \texttt{RLlib}~\cite{rllib}.
To interface between \mfem~and \texttt{RLlib}, we rely on the \mfem~Python wrapper \texttt{PyMFEM}~\cite{pymfem-web}.
\texttt{PyMFEM} allows us to interact with the rigorously tested $hp$-refinement functionalities of \mfem~and create Python-based versions of \Cref{alg:AM2R_1,alg:AM2R_2}.
\texttt{RLlib} allows us to design and sample from the corresponding refinement policies as well as provide numerous state-of-the-art training algorithms that can be used to solve~\cref{eq:err_thresh,eq:dof_thresh}.
These training routines are made especially efficient through parallelism provided by the open-source workload manager~\texttt{Ray}~\cite{ray}, which \texttt{RLlib} natively employs.
\end{remark}

\begin{remark}
The following configuration settings for \texttt{RLlib} that control the policy training regimen \texttt{ppo.PPOTrainer(...)} are common among our experiments.
We use batch mode \texttt{truncate\_episodes}, sgd minibatch size = 100, rollout fragment length = 50, number of workers = 10, train batch size = 500, $\gamma=1.0$, learning ratio = $10^{-4}$, and seed = 4000.
The neural network that we train (i.e., the ``model'' in \texttt{RLlib} terminology) has two hidden layers, each of size 128, with the Swish activation function~\cite{RZL2017}.
\end{remark}

\section[Extension to hp refinement]{Extension to $hp$-refinement} \label{sec:extension_to_hp-refinement}

Like (AM)$^2$R (\Cref{sub:episodic_environments}), traditional $hp$-AMR obeys a generalization of the SEMR sequence~\cref{SEMR} with an implicit  ``decision'' step \cite{mitchell2014comparison}.
In particular, after marking the set of elements to be refined, the algorithm must \emph{decide} whether to $h$- or $p$-refine each marked element \cite{guo1986hp,demkowicz1989toward,oden1989toward,rachowicz1989toward,babuvska1994p}.
One of the most popular ways to make this decision is to partition the set of marked elements based on a local regularity estimate~\cite{gui1985h-pt1,melenk2001residual,houston2003sobolev,houston2005note} or \emph{a priori} knowledge of the local solution regularity \cite{ainsworth1997aspects}.
Another popular strategy relies on brute force computations on an enriched mesh~\cite{demkowicz2002fully}.
These and other strategies are compared in the review article \cite{mitchell2014comparison}.

In general, the philosophy behind an efficient $hp$-AFEM algorithm is that the solution should be $h$-refined in regions with low regularity and $p$-refined where it has high regularity \cite{babuvska1994p}.
Estimating the local solution regularity often requires multiple local estimators, which we will not describe further for the sake of time and space (see, e.g., \cite{mitchell2014comparison}).
Instead, we devise an $hp$-refinement strategy that requires only one error estimator and encourage follow-on work that considers using, e.g., multiple local estimators to make similar $hp$-refinement decisions.
Rather than aiming to provide a complete view into reinforcement learning for $hp$-refinement, the purpose of our investigation is to demonstrate that sophisticated refinement policies can be learned with our framework.
In particular, we will show that employing a multidimensional action space $\mathcal{A}$ is both feasible and practical.

\subsection[hp action space]{$hp$ action space} \label{sub:action_space}

In this work, we rely on \emph{a priori} knowledge to partition the set of marked elements into disjoint subsets for $h$- and $p$-refinement, respectively.
Our marking rule is inspired by the largely successful ``flagging'' strategy proposed in \cite{ainsworth1997aspects}.
In the ``flagging'' strategy, the user flags specific geometric features in the mesh where they anticipate the regularity to be low, and, in turn, a marked element is $h$-refined if and only if its closure intersects a flagged feature and all other marked elements are $p$-refined.
The comparison in \cite{mitchell2014comparison} demonstrates that flagging can outperform much more complicated strategies in benchmark problems with singular solutions.
However, it is widely acknowledged that flagging has limited utility because it involves direct (sometimes ad hoc) user interaction.

In our generalization, we aim to target elliptic PDEs with singular solutions and, in doing so, assume that the relative size of the local error estimate is correlated to the physical distance from singularities.
Based on this correlation, we can induce $h$-refinement near singularities by marking a subset of elements with the largest local error estimates for $h$-refinement.
We then mark a disjoint subset of elements with the next largest local error estimates for $p$-refinement.

More specifically, let $\mathcal{A} = [0,1] \times [0,1]$.
For $(\theta, \rho) \in \mathcal{A}$, we $h$-refine all elements $T \in \mathcal{T}$ satisfying
\begin{subequations}
\begin{equation}
\label{eq:hp-hMaxRule}
  \theta\cdot\max_{S\in\mathcal{T}}\{\eta_S\} < \eta_T
  ,
\end{equation}
and we $p$-refine all elements $T \in \mathcal{T}$ satisfying
\begin{equation}
\label{eq:hp-pMaxRule}
  \rho \theta \cdot\max_{S\in\mathcal{T}}\{\eta_S\}
  <
  \eta_T
  \leq
  \theta\cdot\max_{S\in\mathcal{T}}\{\eta_S\}
  .
\end{equation}
\end{subequations}
One may readily note \rev{that conditions \cref{eq:hp-hMaxRule} and \cref{eq:hp-pMaxRule} are exclusive, meaning any given element is marked for $h$, $p$, or no refinement, but never for both $h$ and $p$ refinement.}
Further, observe that $\theta = 1$ induces only $p$-refinement and $\rho = 1$ induces only $h$-refinement.
Alternatively, $\theta = 0$ induces uniform $h$-refinement and $\rho = 0$ induces uniform $hp$-refinement (with the split between $h$ and $p$ controlled by $\theta$).
Thus, our marking rule provides a wide (but not exhaustive) range of possible refinement actions.
This flexibility is appealing but leads to a difficult parameter specification problem that has a longstanding precedent for $hp$ marking rules.
For instance, one of the oldest $hp$ strategies in the literature \cite{gui1985h-pt3} also uses a parameter to partition the marked elements into $h$- and $p$-refinement subsets.

\subsection[hp observation space]{$hp$ observation space} \label{sub:observation_space}

Unlike optimal $h$-refinement, optimal $hp$-refinement leads to an exponential convergence rate in error reduction~\cite{babuvska1994p} and typically causes $p$ to vary across elements in the mesh.
As a result, the normalization of ${\eta}_T$ in~\cref{eq:NormalizedErrorEstimates} is not suitable for $hp$-refinement because it depends explicitly on the polynomial order $p$.
Our remedy is to construct an alternative distribution variable based on the local convergence rate, which takes the place of the exponent $p/d$ in~\cref{eq:NormalizedErrorEstimates}.
In particular, we define
\begin{equation}
\label{eq:DefinitionOfZeta}
	\zeta_T = -\frac{\ln(|\mesh_k|^{1/2} \eta_T)}{\ln( \ndofs(\mesh_k) )}
	,
\end{equation}
or, equivalently,
\begin{equation}
\label{eq:ProofStep2}
	\ndofs(\mesh_k)^{-\zeta_T} = |\mesh_k|^{1/2}\, \eta_T
	.
\end{equation}

It is straightforward to show that $\variance_k[\zeta_T] = 0$ if and only if $\variance_k[\eta_T] = 0$.
Therefore, the variance of $\zeta_T$ also provides a way to quantify error equidistribution.
An interesting second property is that the expected value of $\zeta_T$ is related to the global convergence rate, as evidenced by~\cref{prop:ExpectedValueOfZeta}.

\begin{proposition}
\label{prop:ExpectedValueOfZeta}
						If there exist constants $C,\beta>0$ such that
	\begin{equation}
	\label{eq:LowerBoundAssumption}
		\eta_k
		\leq
		C \ndofs(\mesh_k)^{-\beta}
		,
	\end{equation}
	then
	\begin{equation}
		\liminf_{k\to\infty}~\mean_k[\zeta_T] \geq \beta
		.
	\end{equation}
\end{proposition}
\begin{proof}
	As with~\cref{eq:Normalization1}, it is straightforward to show that
	\begin{equation}
		|\mesh_k| \mean_k[\eta_T^2]
		=
		\eta_k^2
		.
	\end{equation}
	Therefore, by~\cref{eq:ProofStep2}, we have that $\mean_k[\ndofs(\mesh_k)^{-2\zeta_T}] = \eta_k^2$.
	We now use this identity and Jensen's inequality to derive an upper bound on $\ln \eta_k^2$:
	\begin{align*}
						\ln \eta_k^2
						&=
		\ln \mean_k[\ndofs(\mesh_k)^{-2\zeta_T}]
						\geq
		\mean_k[\ln (\ndofs(\mesh_k)^{-2\zeta_T}) ]
		=
		-2\mean_k[\zeta_T]\ln (\ndofs(\mesh_k))
		.
	\end{align*}
		Thus, by~\cref{eq:LowerBoundAssumption}, we have
	\begin{equation}
	\label{eq:ProofStep3}
		-2\mean_k[\zeta_T]\ln (\ndofs(\mesh_k))
		\leq
		\ln \eta_k^2
		\leq
		2\ln C - 2\beta \ln (\ndofs(\mesh_k))
		.
	\end{equation}
	The proof is completed by dividing the left and right sides of~\cref{eq:ProofStep3} by $-2\ln (\ndofs(\mesh_k))$ and considering the limit as $\ndofs(\mesh_k) \to \infty$.~
\end{proof}

\Cref{fig:error-dist-zeta} depicts the distribution of $\zeta_T$ for the same model problems considered in~\Cref{fig:error-dist-href}.
In the top row, we see that variance of $\zeta_T$ remains bounded for both the smooth and singular solutions undergoing uniform $h$-refinement.
Moreover, both distributions appear to be converging logarithmically to a fixed distribution.
In the bottom row, we see that adaptive $h$-refinement decreases the variance for both types of solutions.
In our experiments, we also tried observing different combinations of statistics of $\zeta_T$ and settled on the following formulation due to its simplicity: $\mathcal{O}_{hp} = [0,1] \times [0,\infty) \times [0,\infty)$, where each $o_k \in \mathcal{O}$ is defined by
\begin{equation}
	o_k
	=
	\left(~b_k,~\mean_k[\zeta_T],~\sd_k[\zeta_T]~\right)
	.
	\label{eq:hpobs}
\end{equation}

\begin{figure}
\centering
\begin{tabular}{lcc}
	{\rotatebox[origin=l]{90}{~~~Uniform $h$-ref}}
	& \includegraphics[width=.32\linewidth]{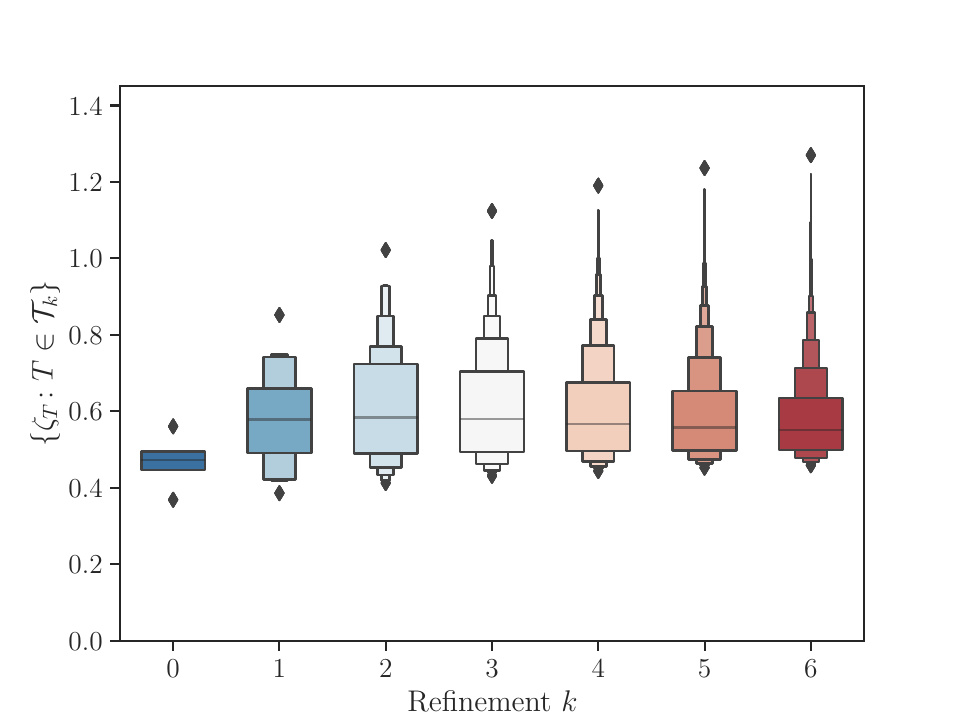}
	& \includegraphics[width=.32\linewidth]{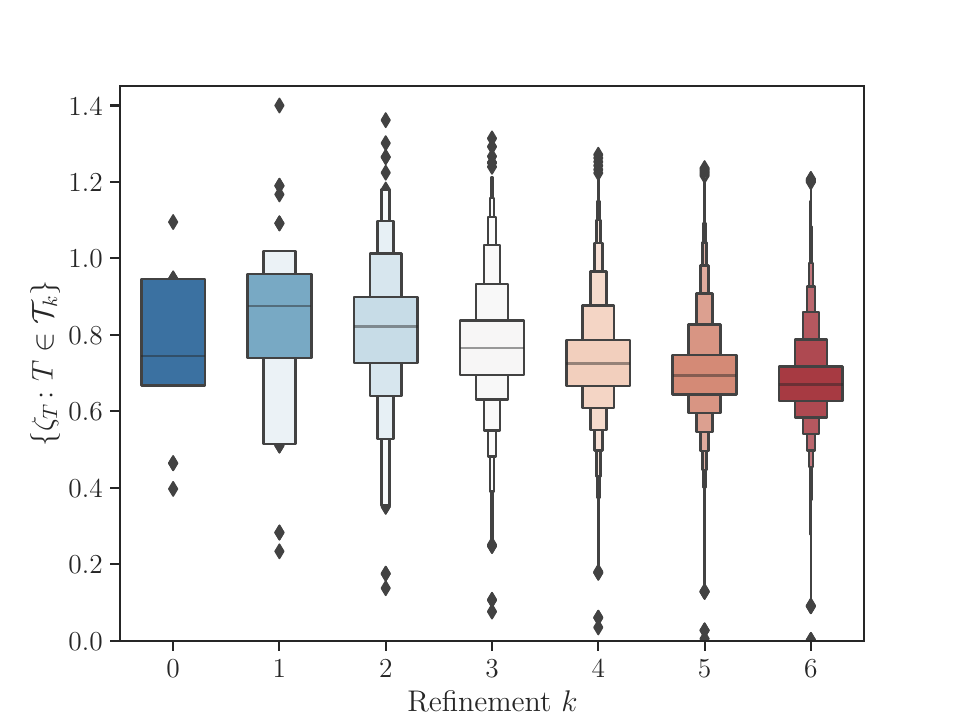} \\
{\rotatebox[origin=l]{90}{~~~Adaptive $h$-ref}}
	& \includegraphics[width=.32\linewidth]{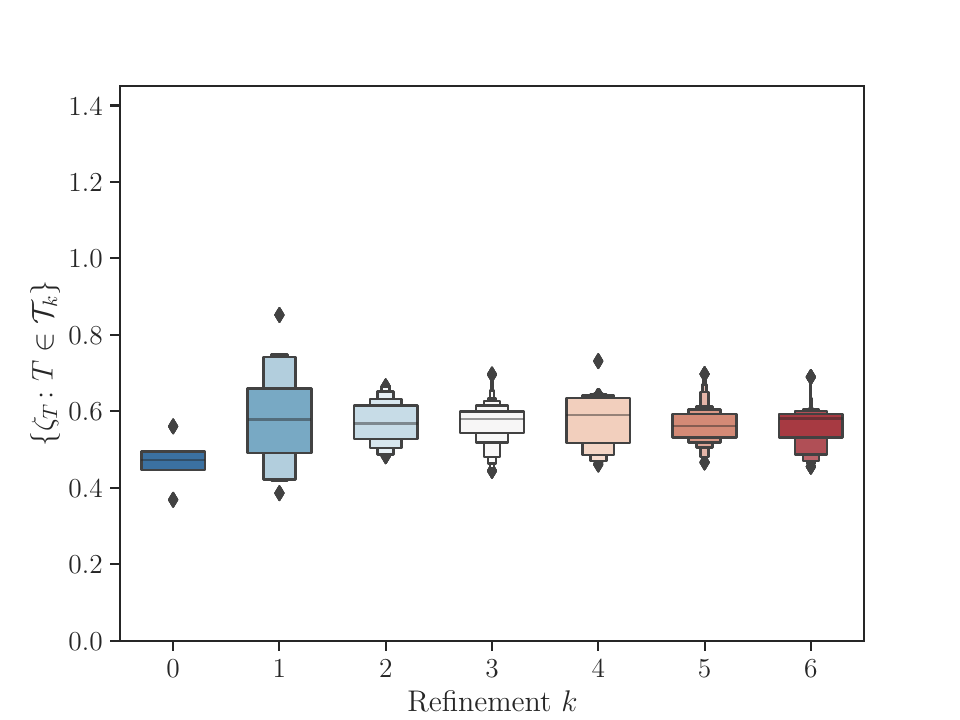} 
	& \includegraphics[width=.32\linewidth]{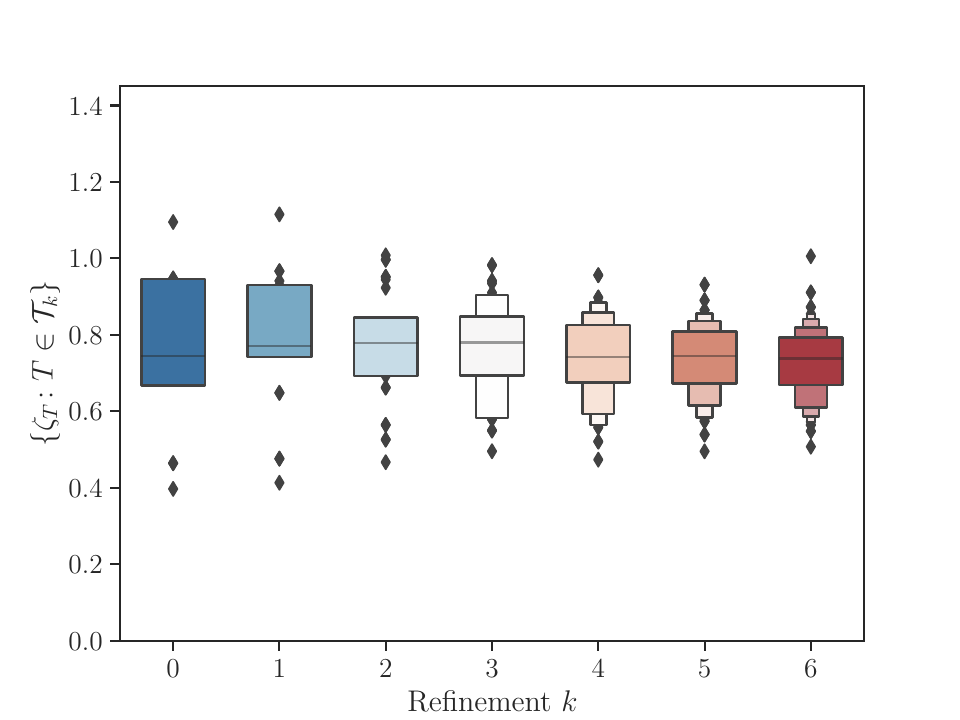} \\[1mm]
	& $u=\sin(\pi x)\sin(\pi y)$ & $u=r^{2/3}\sin(2\theta/3)$
	\end{tabular}
	\caption{
	Empirical distributions of $\zeta_T$ \cref{eq:DefinitionOfZeta}, the normalized local error estimate we will use for $hp$-refinement.
	The experimental setup is identical to that of~\Cref{fig:error-dist-href}, replacing $\widetilde{\eta}_T$ with $\zeta_T$.
	Unlike the prior case, the mean and variance of $\zeta_T$ do not diverge under uniform refinement for the singular solution case (top right).  
	We observe that the variance of $\zeta_T$ is bounded more tightly by AMR than by uniform refinement by comparing the scale of the axes in the right column.
}
\label{fig:error-dist-zeta}
\end{figure}

\subsection[hp SEMDR algorithm]{$hp$ SEDMR algorithm} \label{sub:algorithm}

In~\Cref{alg:hpAM2R_2} we state the $hp$-(AM)$^2$R algorithm for the accuracy problem.
The algorithm for the efficiency problem is similar; cf.~\Cref{alg:AM2R_1}. 
\begin{algorithm2e}
\DontPrintSemicolon
	\caption{\label{alg:hpAM2R_2} $hp$-(AM)$^2$R with a computational budget constraint.}
	\SetKwInOut{Input}{input}
	\SetKwInOut{Output}{output}
	\SetKw{Break}{break}
	\Input{Initial mesh $\mesh_0$, marking policy $\pi$, computational budget $J_\infty > 0$.}
	\Output{Discrete solution $u_k$, error estimate $\eta_k$.}
	$k \leftarrow 0$.\;
	\While{$J_k < J_\infty$}
	{
		Solve~\cref{eq:DiscreteCanoncialPDE} with $\mesh = \mesh_k$. \tcp*[l]{\textsc{solve}~~~~~~}
		Compute error estimates $\{\eta_T\}_{T\in\mesh_0}$. \tcp*[l]{\textsc{estimate}}
		Sample $(\theta_k,\rho_k) \sim \pi(\cdot|o_k)$. \tcp*[l]{\textsc{decide}}
		Mark all $T\in\mesh_k$ satisfying~\cref{eq:hp-hMaxRule} for $h$-refinement. \tcp*[l]{\textsc{mark}}
		Mark all $T\in\mesh_k$ satisfying~\cref{eq:hp-pMaxRule} for $p$-refinement. \tcp*[l]{\textsc{mark}}
		Form $\mesh_{k+1}$ by refining all marked elements in $\mesh_k$. \tcp*[l]{\textsc{refine}}
		$k \leftarrow k+1$.\;
	}
\end{algorithm2e}

\section{Numerical results} \label{sec:applications}

We present a collection of numerical experiments to demonstrate the feasibility and potential benefits of employing a trained (AM)$^2$R policy.
The following examples begin with simple $h$-refinement validation cases, followed by extensions to more general $hp$-refinement on 3D meshes.
In all experiments, we used the Zienkiewicz--Zhu error estimator \cite{zienkiewicz1992superconvergent1,zienkiewicz1992superconvergent2} to compute the local error estimates $\eta_T$; cf.~\Cref{ex:ZZ}.
Moreover, all experiments relied on the greedy marking rule~\cref{eq:GreedyMeshRefinement}.
Similar results can be obtained with the D\"orfler rule~\cref{eq:DorflerMeshRefinement}.

\subsection[Validation example, h-refinement]{Validation example, $h$-refinement} \label{ssub:ex1a}
We begin with the well-known L-shaped domain problem from \Cref{fig:motivation} and allow only $h$-refinement.
Here, we seek approximations of solutions to Laplace's equation:
\begin{equation}
\label{eq:LaplaceEqn}
	\Delta u = 0
	\quad \text{in } \Omega,
	\qquad
	u = g \quad \text{on } \partial\Omega,
\end{equation}
where the exact solution is known, allowing the specification of the appropriate Dirichlet boundary condition $g$.
Placing a re-entrant corner at the origin of an infinite L-shaped domain and assigning zero Dirichlet boundary conditions on the incident edges provides the following nontrivial exact solution to Poisson's equation in polar coordinates: $r^\alpha\sin(\alpha\theta)$, where $\alpha=2/3$.
The boundary condition $g$ for the (bounded) L-shaped domain in \cref{eq:LaplaceEqn} is the trace of this function and, therefore, $u = r^\alpha\sin(\alpha\theta)$ in $\Omega$.
Optimal $h$-refinement converge rates for this problem are attained by grading refinement depth to be larger for elements closer to the singularity at the re-entrant corner; this strategy is observed in both refinement patterns shown in \Cref{fig:motivation}.

\begin{figure}\centering
        \includegraphics[width=.35\linewidth]{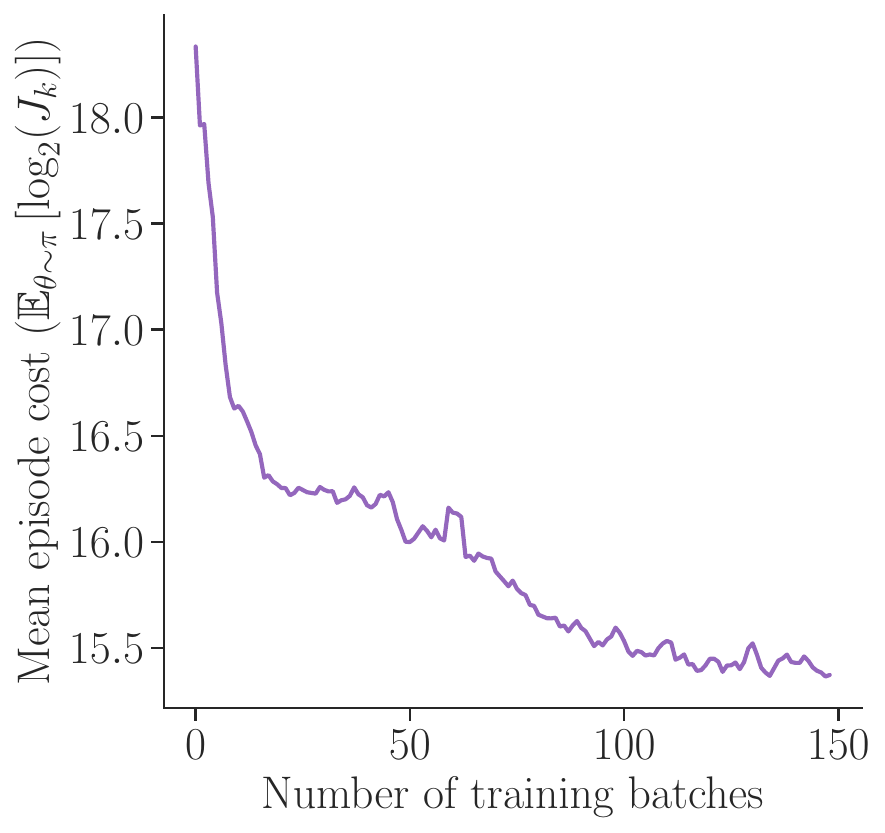}
		\qquad
        \includegraphics[width=.35\linewidth]{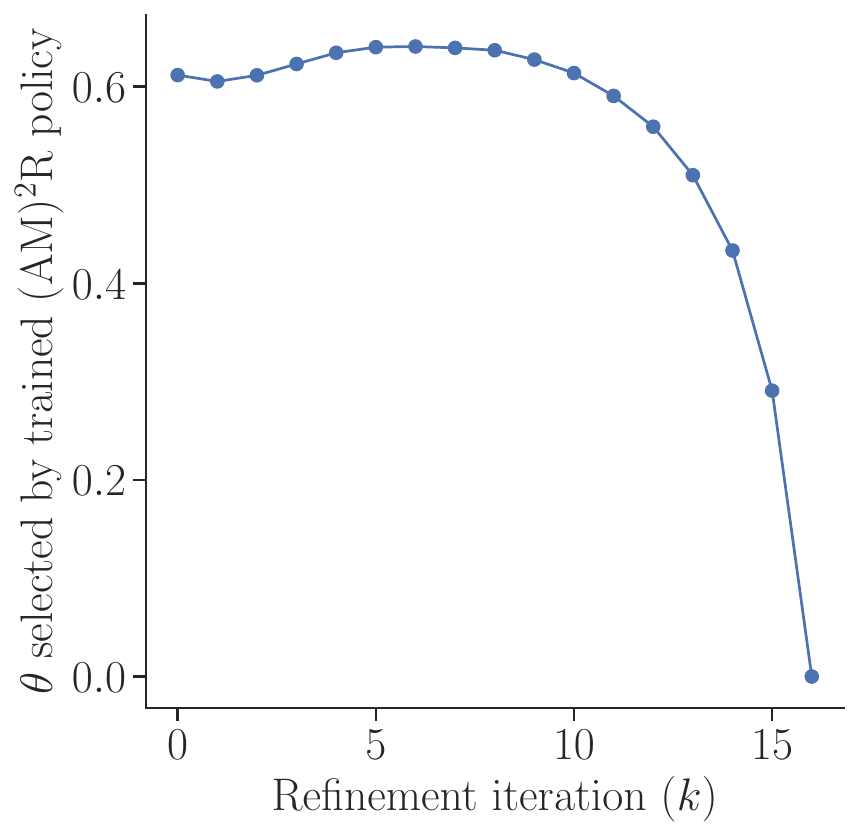}
    \caption{\protect \textit{Left:} Result of training an $h$-refinement policy for the efficiency problem \cref{eq:err_thresh} with uniform polynomial order $p=2$.  The mean episode cost decreases sharply and levels out, as desired. \textit{Right:} When the trained policy is deployed, it dynamically changes the selected $\theta$ value throughout the AFEM workflow, ending with uniform refinement $(\theta=0)$ in the final iteration of \Cref{alg:AM2R_1}.  This sequence of $\theta$ choices is consistent with an intuitively optimal procedure for the classical L-shaped domain problem, as described in the text.}
\label{fig:ex1a-ab}
\end{figure}

To arrive at an $h$-refinement marking policy for \cref{eq:LaplaceEqn}, we solve the efficiency problem \cref{eq:err_thresh} with uniform polynomial order $p=2$, error threshold $\eta_\infty=10^{-4}$, and observation space \cref{eq:hobsspace}.
Once trained, we deploy the policy as described in \Cref{alg:AM2R_1} (see also \Cref{fig:flowchart}) using the same threshold from training.
The results shown in \Crefrange{fig:ex1a-ab}{fig:ex1a-ef} verify that the training and deployment pipelines function as expected in this simple setting and validate its development for more complicated problems.
\Cref{fig:ex1a-ab} (left) shows how the mean episode cost decreases and levels out over a series of 150 training batches, thus indicating that the training has converged to a locally optimal policy.
When we deploy the trained policy, \Cref{fig:ex1a-ab} (right) shows how the action---i.e., the selected $\theta$ value---changes dynamically during the refinement sequence.

We can provide an intuitive interpretation for the sequence of $\theta$ values based on the problem formulation.
The objective of minimizing the cumulative dof count, while still delivering a fixed target error, is best served by initially refining only a small fraction of elements with the very highest error; the value $\theta\approx 0.6$ is determined from training to be a balanced choice in this regard.
Eventually, the error becomes more equidistributed across elements, and the objective is better served by refining many elements at once, resulting in the sharp decrease in $\theta$ values in the latter steps and culminating in uniform refinement at the final step ($\theta = 0$).
This dynamic behavior can be explained by the fact that the cost of refining an individual element grows with the number of remaining refinement steps.
This, in turn, results in a preference to withhold refinements for as long as possible.
A learned policy of transitioning from fairly limited refinements to more uniform refinements will also be observed in the $hp$ experiments described later.

\begin{figure}\centering
        \includegraphics[width=.38\linewidth]{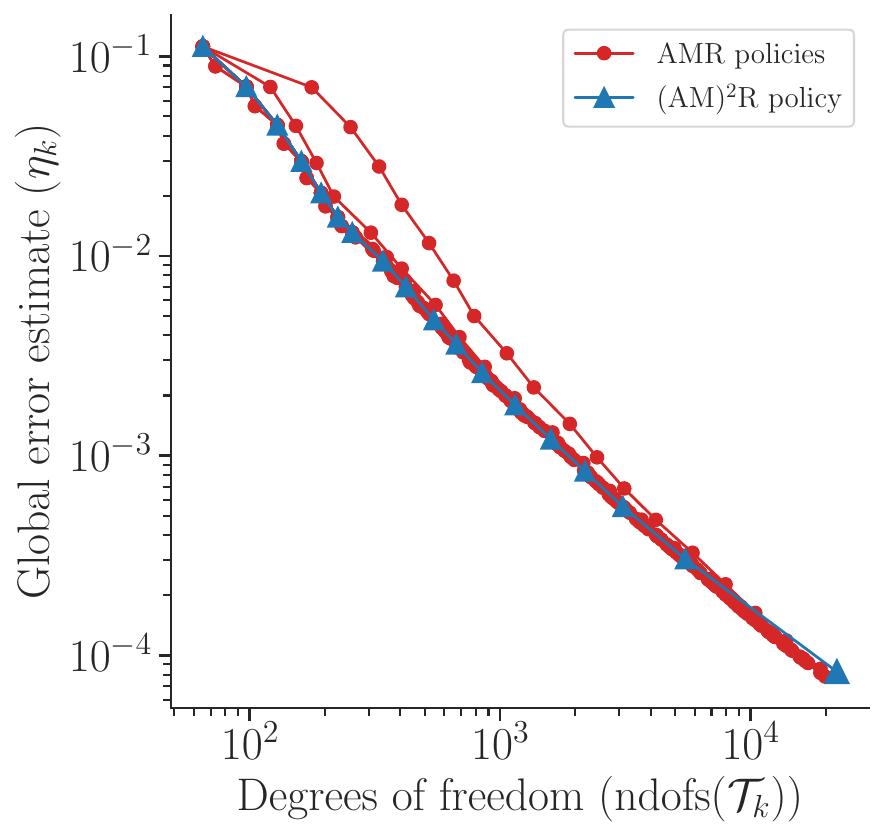}
		\qquad
        \includegraphics[width=.38\linewidth]{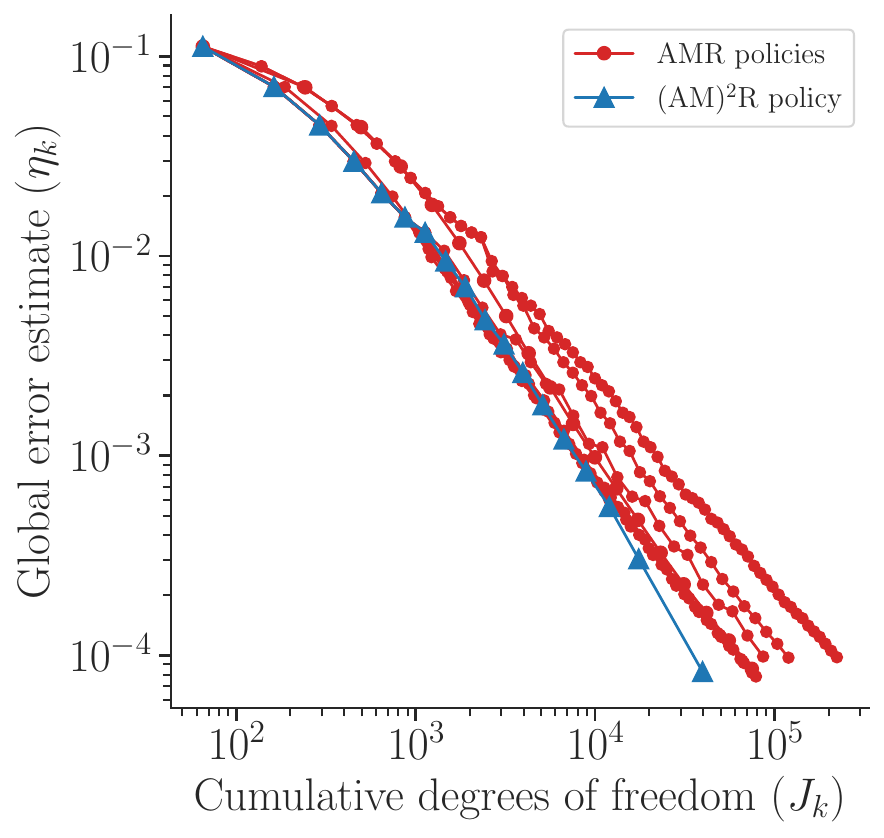}
    \caption{Comparison of the trained (AM$)^2$R policy (blue triangle series) to fixed theta policies (red circle series) for $h$-refinement. \textit{Left:} The (AM$)^2$R policy reduces error as a function of dofs at the same rate as the AMR policies, but the dynamic change in $\theta$ allows it to take larger steps near the end of the AFEM process, thus improving the overall efficiency.  \textit{Right:} The (AM$)^2$R policy achieves the same order of error with significantly fewer \textit{cumulative} dofs than any of the AMR policies, which is the desired goal of the efficiency problem \cref{eq:err_thresh}.  }
\label{fig:ex1a-cd}
\end{figure}

In \Cref{fig:ex1a-cd} (left), we plot the global error estimate $\eta_k$ as a function of dofs for 9 distinct AMR policies with $\theta\in\{0.1,\ldots,0.9\}$ fixed (red dot series) and compare to the RL-trained policy (blue triangle series).
Each point in a series indicates a refinement step ($k$) in an AFEM MDP with the associated policy.
Observe that the (AM$)^2$R policy has a path through these axes that is similar to those of many of the fixed $\theta$ policies, reflecting the fact that it is driving down error at the same rate.
However, in the final steps of the (AM$)^2$R policy, notice that many more dofs are added per step, in accordance with the decrease in $\theta$ value seen in \Cref{fig:ex1a-ab}.
In particular, \rev{the final step of the trained policy is visibly much longer than the typical step sizes of the AMR policies.  By looking at the raw data, we observe that the final step} goes from $5497$ to $22,177$ dofs while driving error down by a factor of $\approx 3.7$ (from $3.0\times 10^{-4}$ to $8.3\times 10^{-5}$), which is a substantially larger step than any of the fixed $\theta$ policies.

A related story is shown in the right plot of \Cref{fig:ex1a-cd}.
Here, the global error estimate of each policy is plotted as a function of \textit{cumulative} dofs at each step, i.e.,~$J_k$.
The (AM$)^2$R policy was trained to minimize $J_k$ and indeed it attains the $10^{-4}$ error threshold with $18\%-61\%$ as many cumulative dofs as any of the traditional (fixed $\theta$) marking policies.

\begin{figure}\centering
        \includegraphics[width=.35\linewidth]{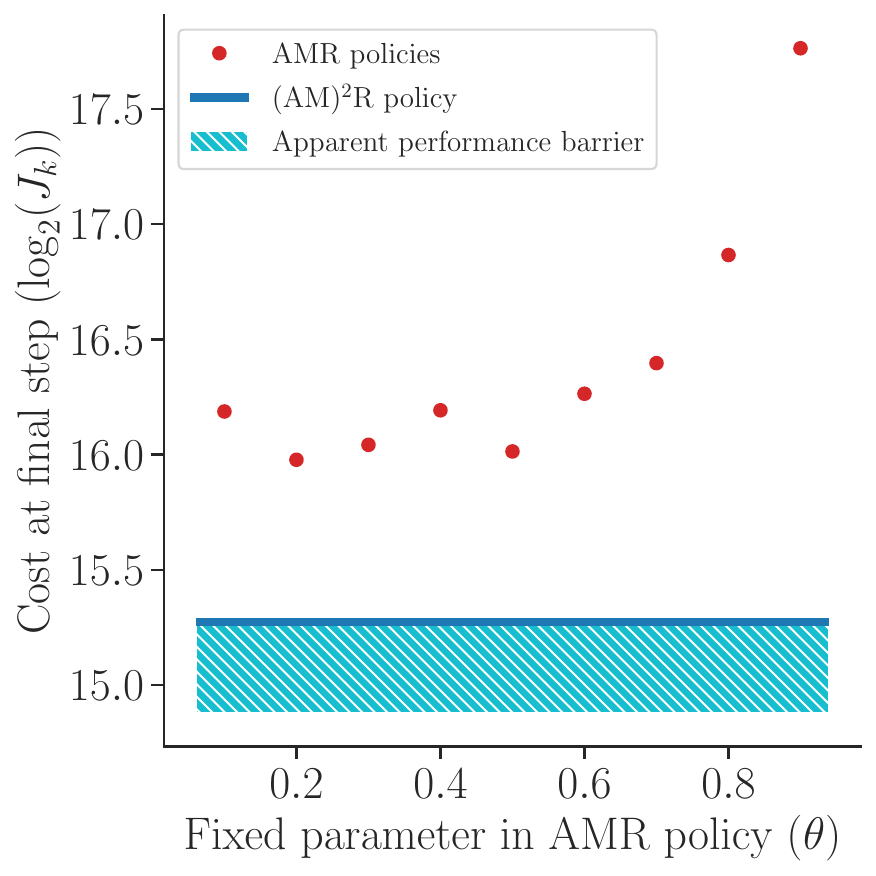}
        \qquad
        \includegraphics[width=.35\linewidth]{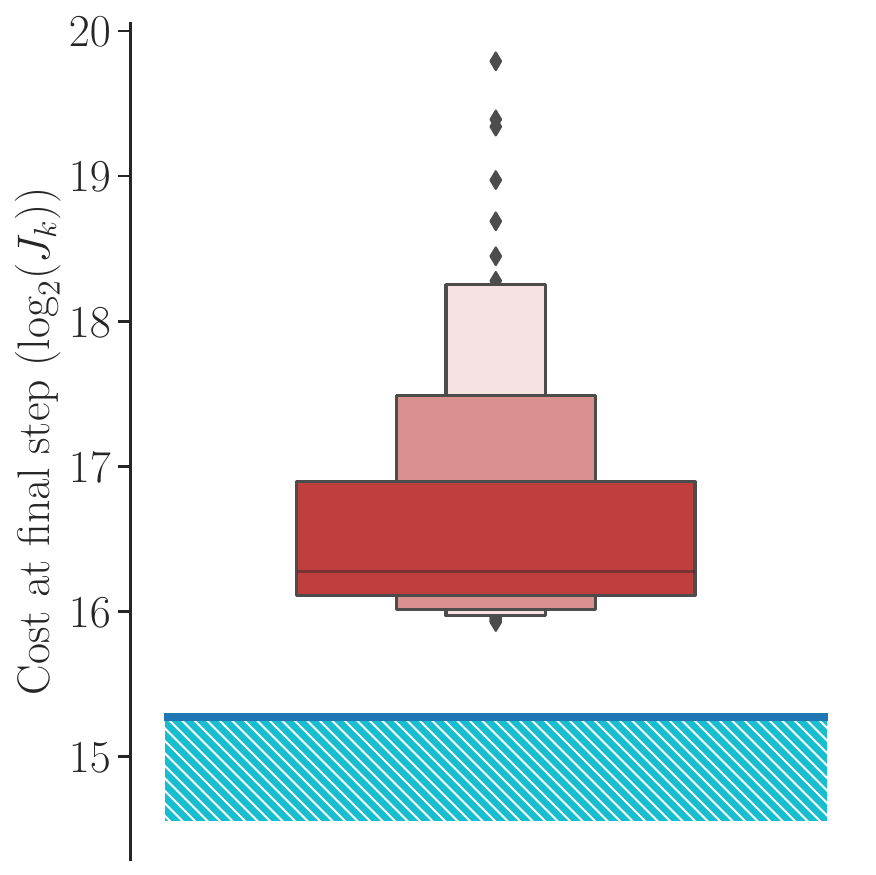}
    \caption{Cost at final step---in this case $\log_2$ of the dof count for the final mesh in which the desired error threshold is achieved---for both the (AM$)^2$R and AMR policies. \textit{Left:} The (AM$)^2$R policy has a noticeably lower cost at final step than a coarse sweep of AMR policies; an apparent barrier to further improvement in this setting is indicated. \textit{Right:} A finer sweep of AMR policies, shown as the red letterbox plot, \rev{reveals that the cost of the \textit{median} fixed parameter policy is roughly one unit higher than the cost of (AM$)^2$R policy. Hence, the (AM$)^2$R policy employs half the cumulative degrees of freedom as the median-performing fixed-parameter policy, meaning it has roughly twice the efficiency by the cumulative degree of freedom metric.}}
\label{fig:ex1a-ef}
\end{figure}

In \Cref{fig:ex1a-ef}, we present two additional views of the cumulative dof count data.
In the left plot, the red dots show the cost at the final step (i.e.,~$\log_2$ of the final cumulative dof count) for each of the traditional AMR policies.
If $\theta$ is fixed throughout an AFEM process---as is the case in every production-level AFEM code that we are aware of---these data suggest that $\theta\approx 0.2$ or $\theta\approx 0.5$ is an optimal choice for minimizing cumulative dof count in this particular setting.
The dark blue line indicates the final cost of the RL policy.  
\rev{Notably, the blue line is well below the cost of any of the traditional AMR policies.  In particular, the median cost of the fixed parameter policies ($\approx$ 16.3, as seen from the letterbox plot on the right) is a full unit above the cost of the RL policy.
Due to the $\log_2$ scaling of the vertical axis, we conclude that the (AM$)^2$R policy trained here has roughly twice the efficiency as the median fixed parameter policy, which could be taken as a proxy for a median performant single parameter choice in a traditional setting.}
We shade the region below the RL policy line to indicate that additional RL training is unlikely to discover a lower final cost, based on our numerical experiments.

Finally, in the right plot of \Cref{fig:ex1a-ef}, we show a ``letter-value plot'' of a larger set of AMR policies, for which we tried every $\theta\in\{0.1,\ldots,0.99\}$.
No improvement is found over the coarser sampling of $\theta$ values shown in the left plot, and, moreover, some choices of $\theta$ (particularly those very close to 1.0) are observed to perform much worse.
With this experiment, we have \rev{demonstrated that dynamically selecting marking parameters \textit{can} improve the efficiency of an AFEM pipeline over a fixed parameter selection; we now move on to experiments that generalize beyond training and deploying on the same mesh.}

\subsection[Robustness to domain geometry, hp-refinement]{Robustness to domain geometry, $hp$-refinement} \label{ssub:ex2c}

We now move on to $hp$-AFEM and more general domain geometries.  
As described in~\Cref{alg:hpAM2R_2}, the action space for the $hp$ \textsc{decide} step is a tuple $(\theta,\rho)\in[0,1]^2$.
Recall that if $\eta_{\max}$ denotes the maximum error estimate for an element at any $k$th step of the MDP, elements $T\in\mesh_k$ with error estimates $\eta_T \in (\rho\theta \eta_{\max},~\theta \eta_{\max}]$ will be marked for $p$-refinement while elements with error estimates $\eta_T \in (\theta \eta_{\max}, \eta_{\max}]$ will be marked for $h$-refinement.
We consider a setting where the optimization goal is best served not only by per-step changes to $\theta$ and $\rho$ but also by allowing the \textit{pace} of such changes to respond to the computed global error distributions.

We approximate solutions to Laplace's equation over a family of domain geometries consisting of the unit disk with a radial section of angle $\omega\in(0, 2\pi)$ removed.
Example domains are shown in~\Cref{fig:hp-train}.
As in the L-shaped domain case, on the straight edges we assign zero Dirichlet boundary conditions and have an exact solution to \cref{eq:LaplaceEqn} in polar coordinates given by $u=r^\alpha\sin(\alpha\theta)$, where $\alpha = \pi/(2\pi -\omega)$~(see, e.g.,~\cite{mitchell2013collection}).
Boundary conditions for the curved portion of the domain are determined from this solution.
Note that the gradient of the solution is singular at the origin if and only if $\omega<\pi$, i.e., if and only if the domain is nonconvex; the singularity is stronger the closer $\omega$ is to 0.

\begin{figure}
\centering
\setlength\tabcolsep{0mm}
\begin{tabular}{ccccc}
{\rotatebox[origin=l]{90}{Initial mesh}}
	\includegraphics[width=0.19\textwidth]{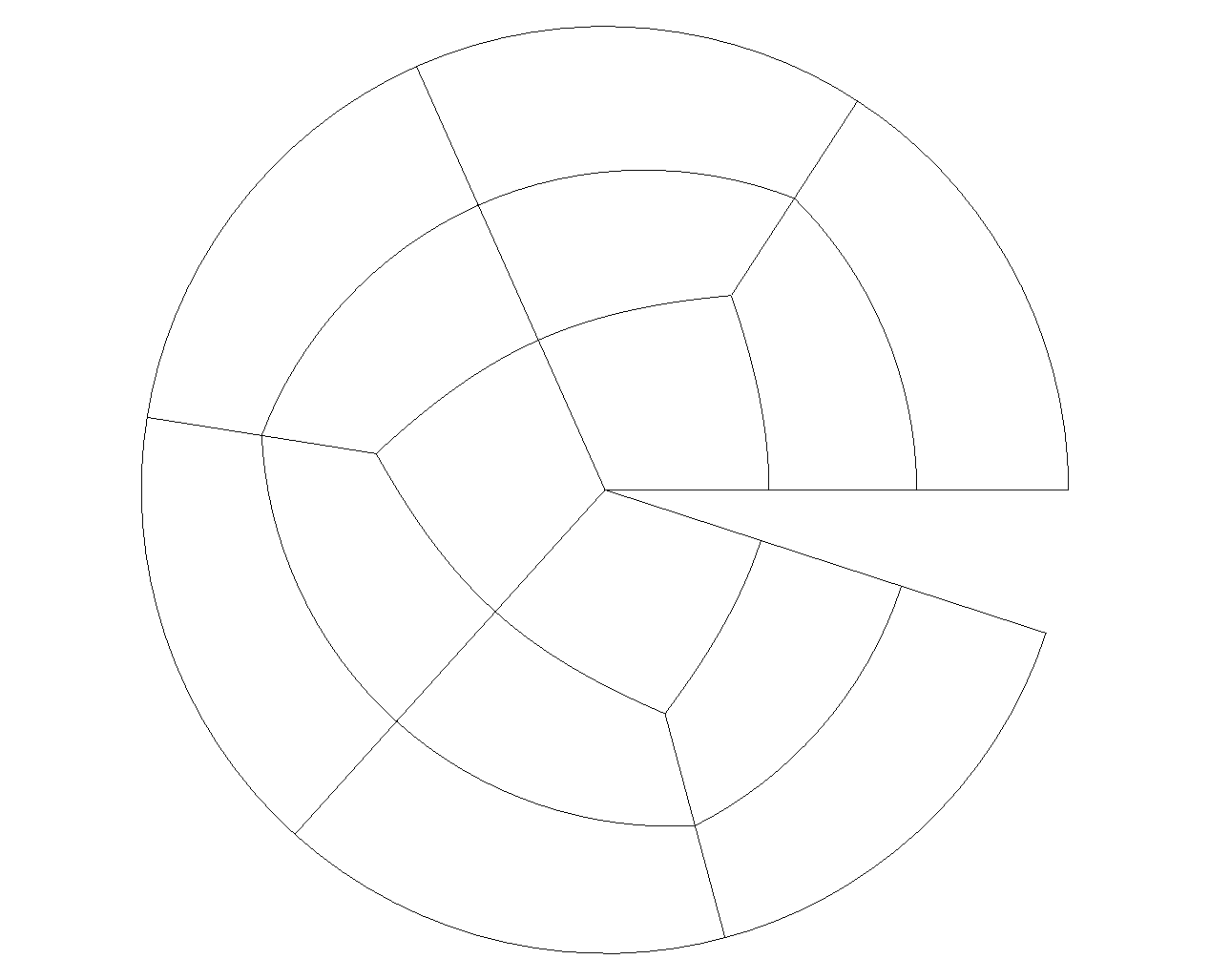} &
	\includegraphics[width=0.19\textwidth]{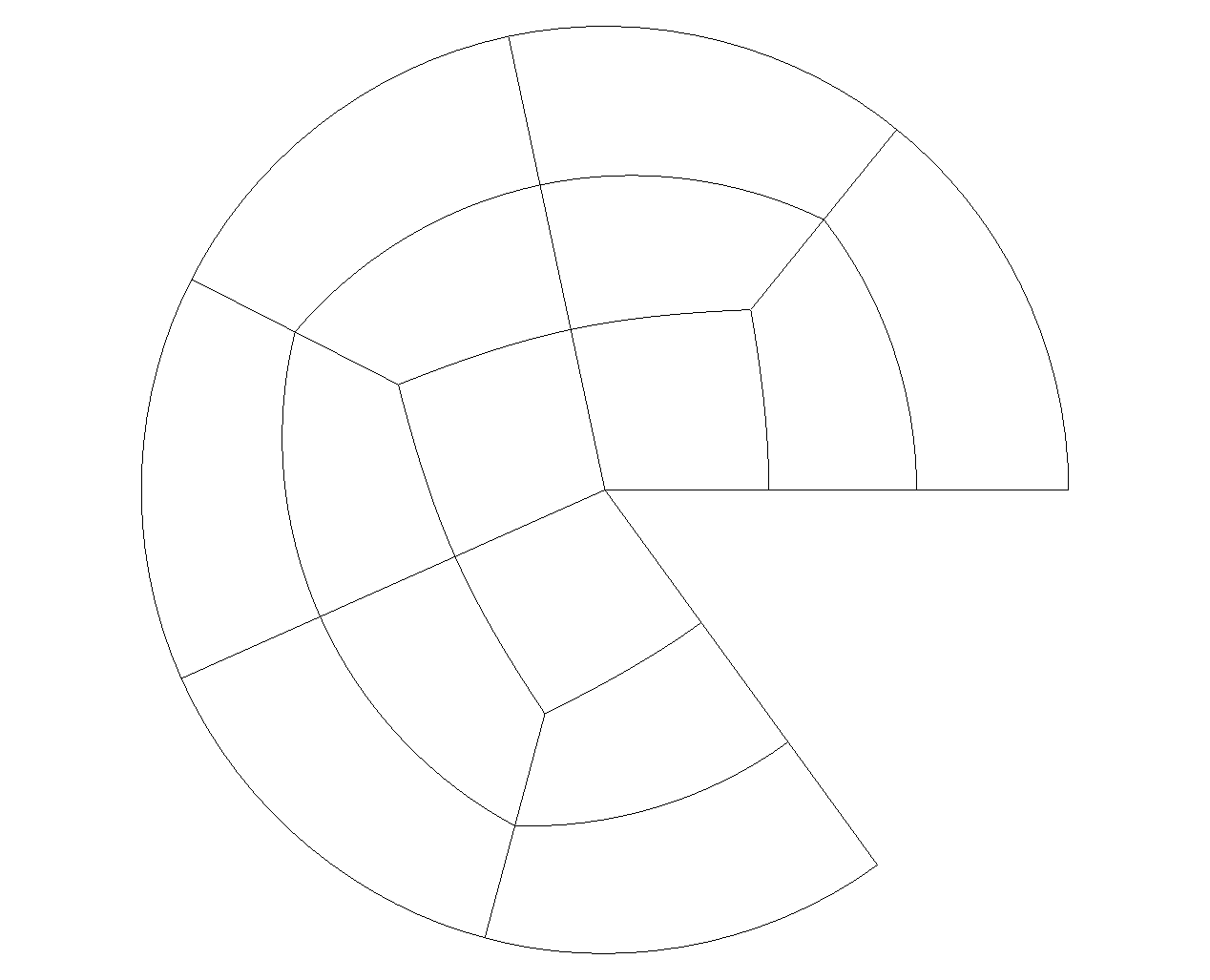} &
	\includegraphics[width=0.19\textwidth]{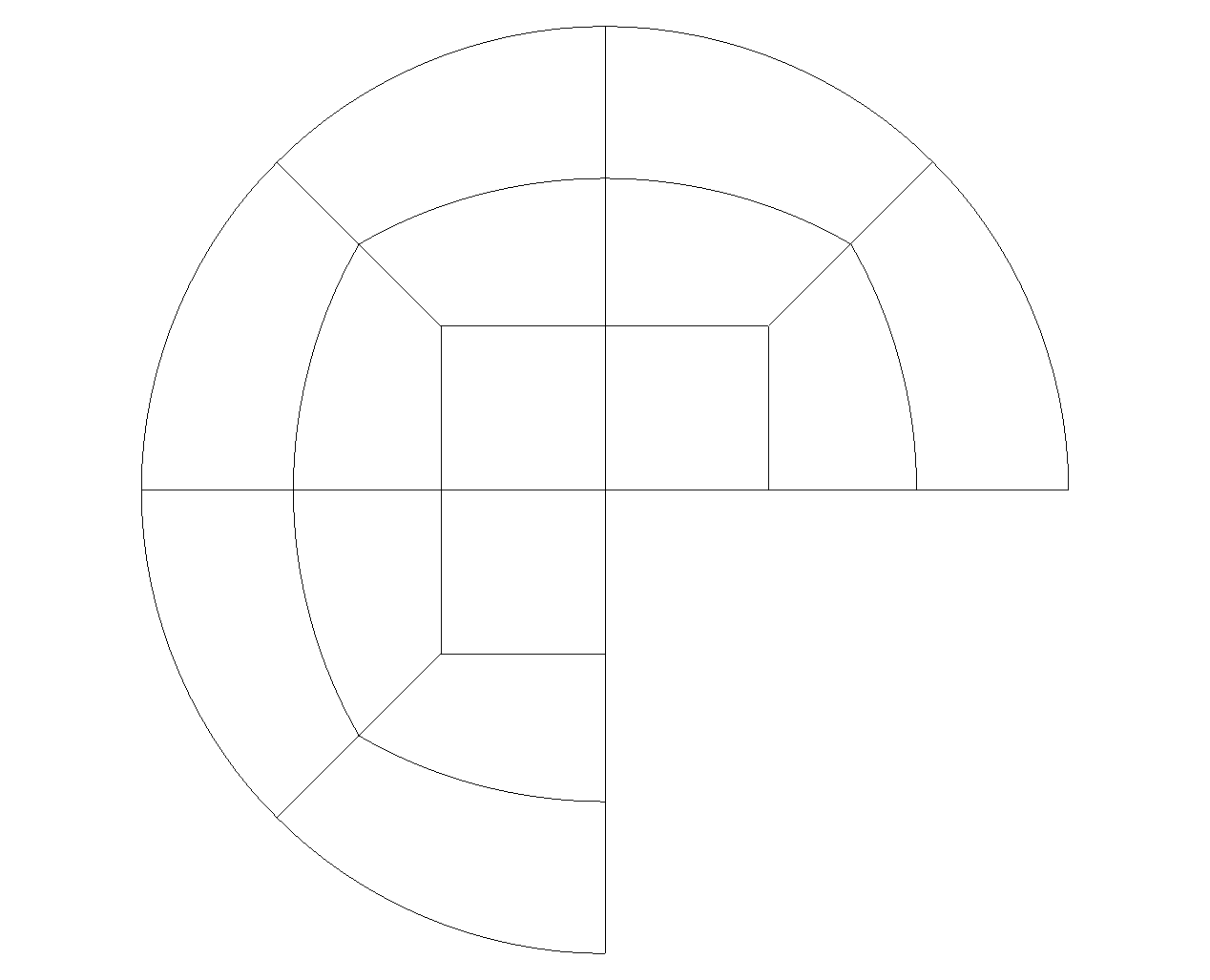} &
	\includegraphics[width=0.19\textwidth]{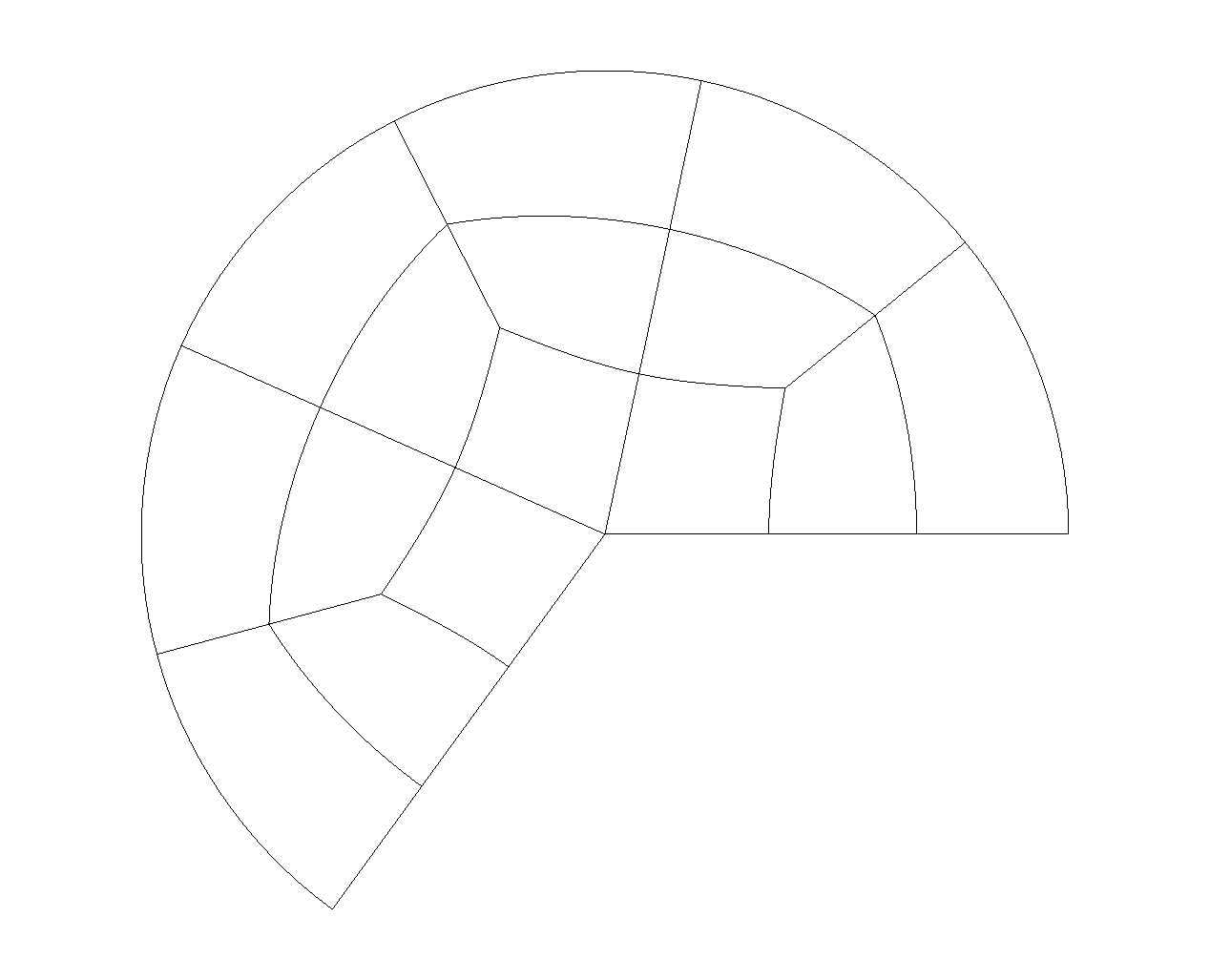} &
	\includegraphics[width=0.19\textwidth]{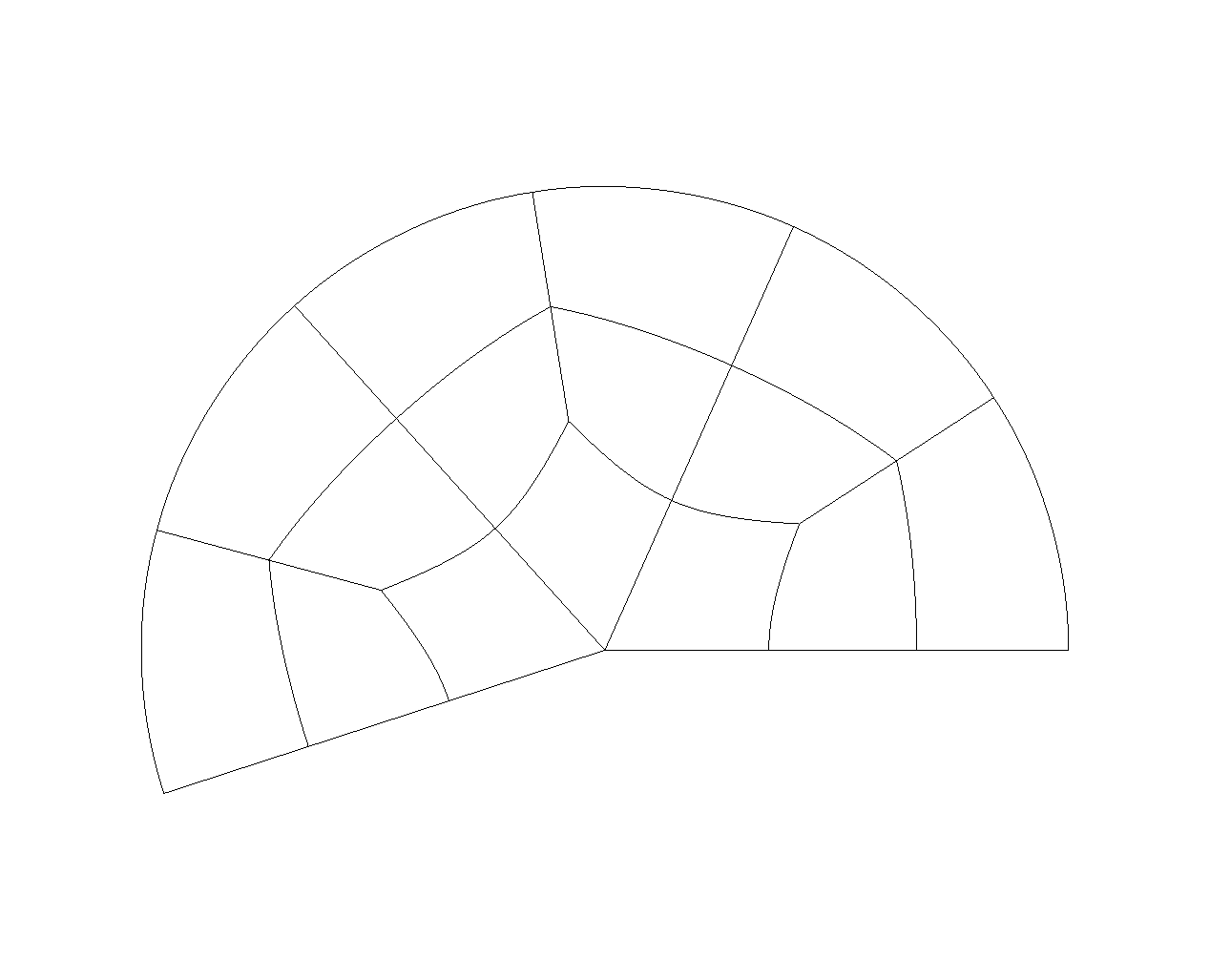} 
	\\
{\rotatebox[origin=l]{90}{~~Final mesh}}
	\includegraphics[width=0.19\textwidth]{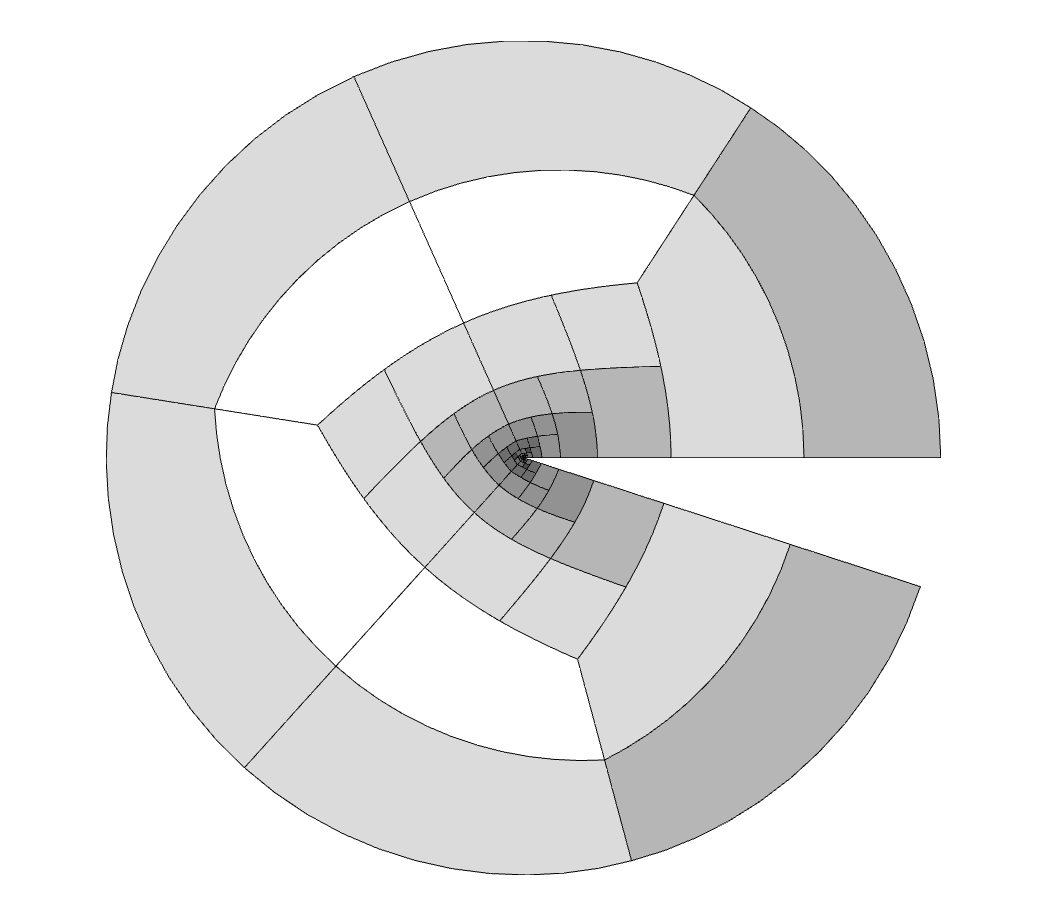} &
	\includegraphics[width=0.19\textwidth]{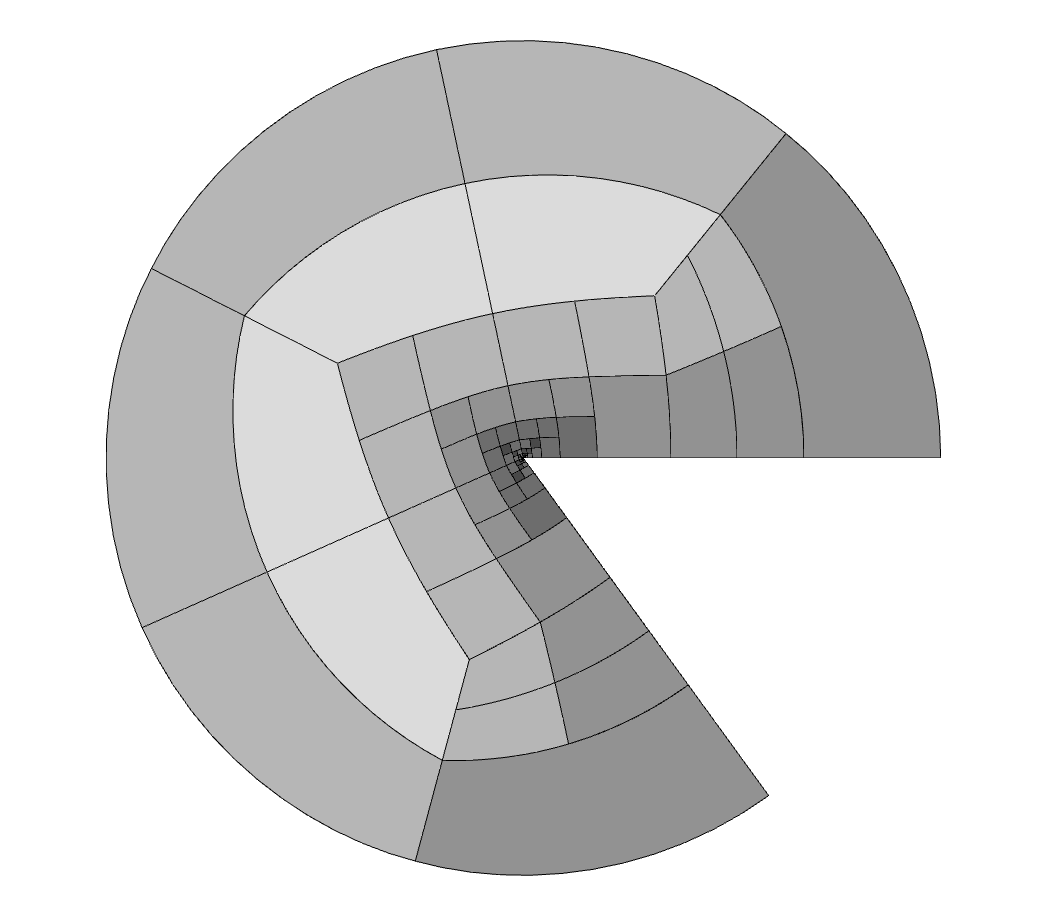} &
	\includegraphics[width=0.19\textwidth]{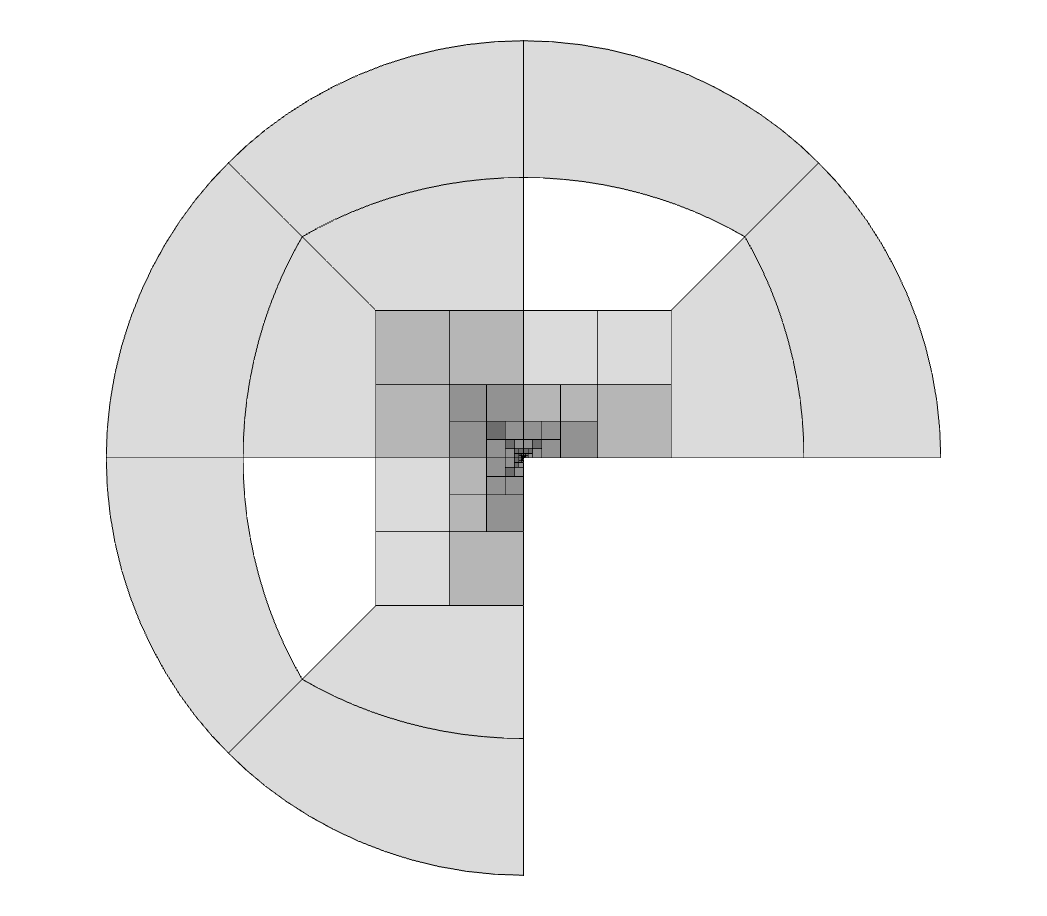} &
	\includegraphics[width=0.19\textwidth]{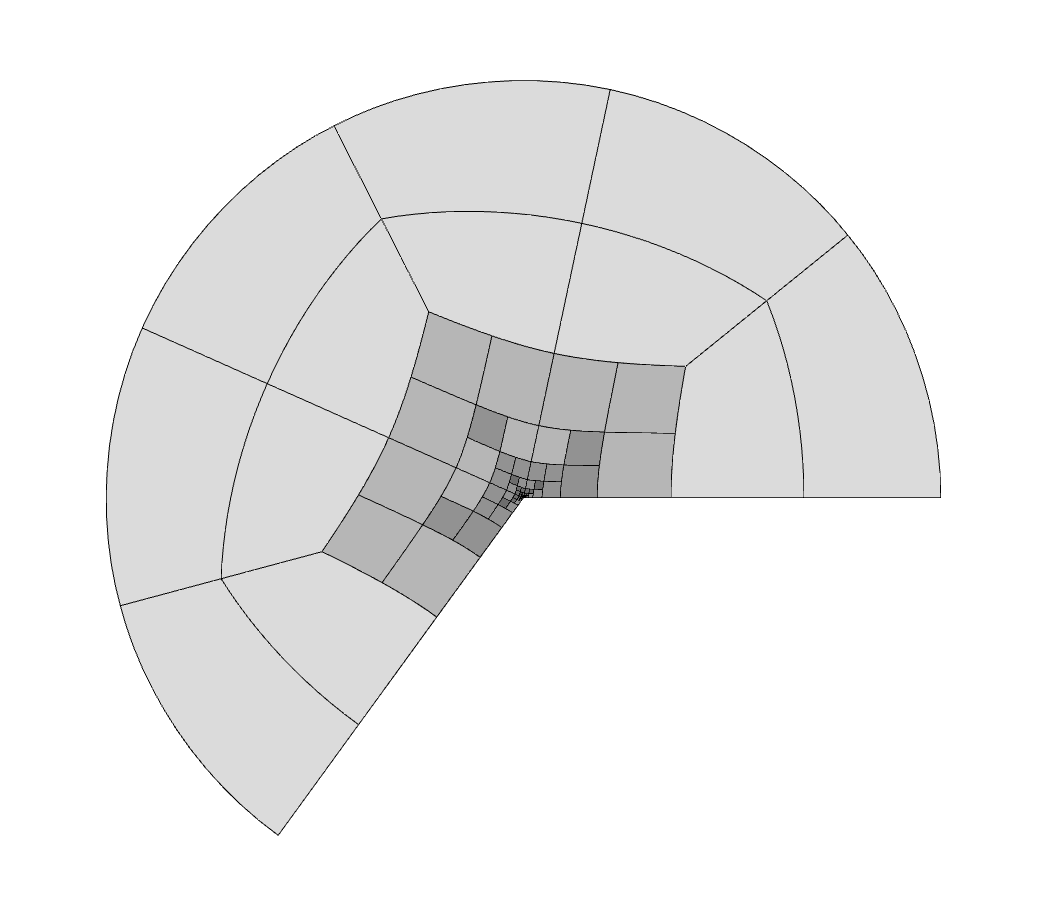} &
	\includegraphics[width=0.19\textwidth]{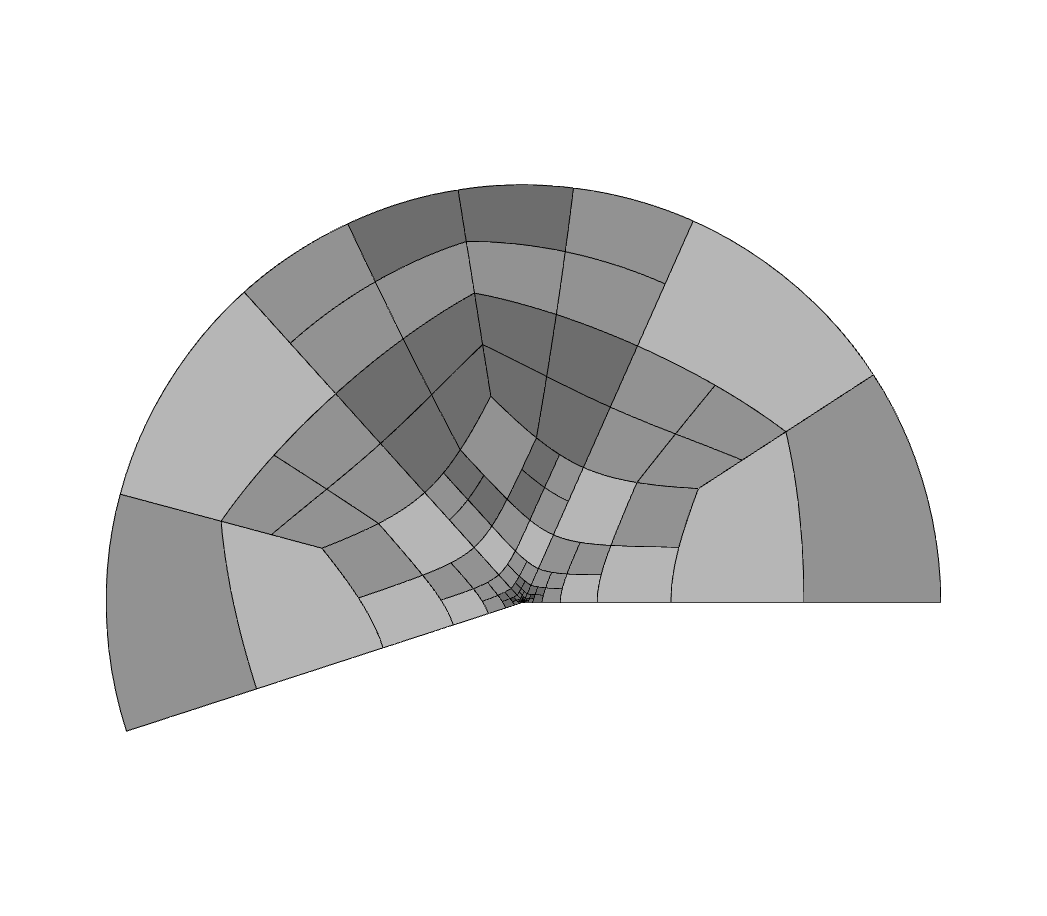} 
	\\
	{\rotatebox[origin=l]{90}{~~Observables}}
	\includegraphics[width=0.19\textwidth]{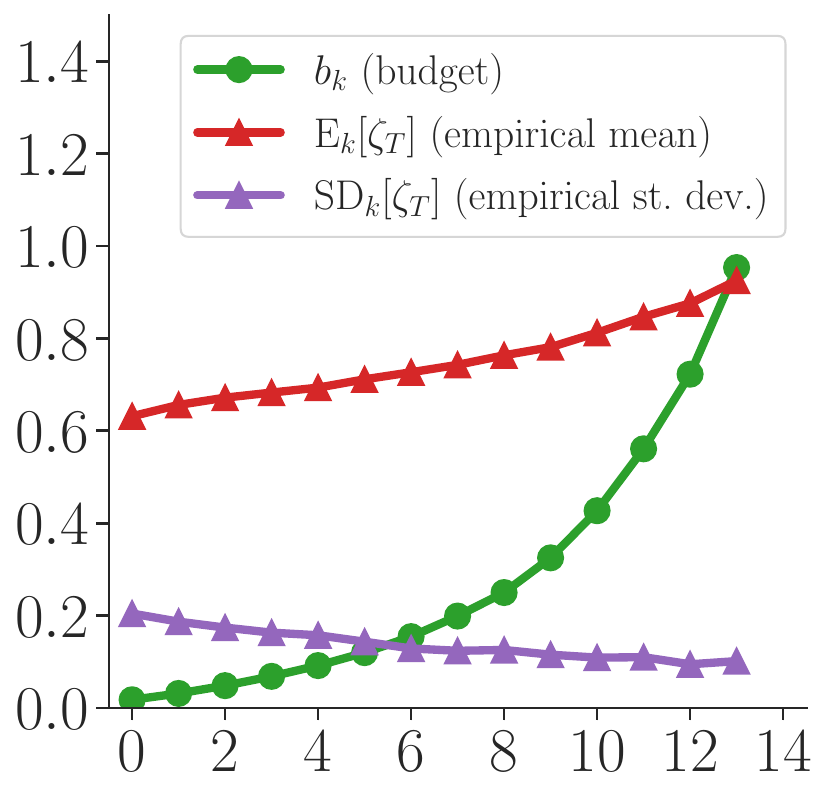} &
	\includegraphics[width=0.19\textwidth]{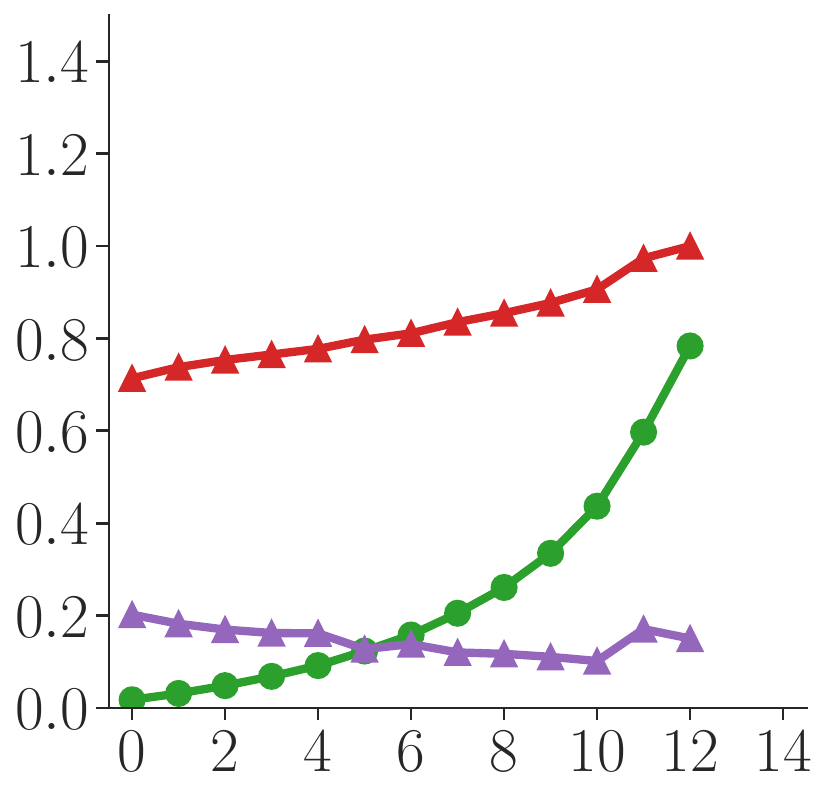} &
	\includegraphics[width=0.19\textwidth]{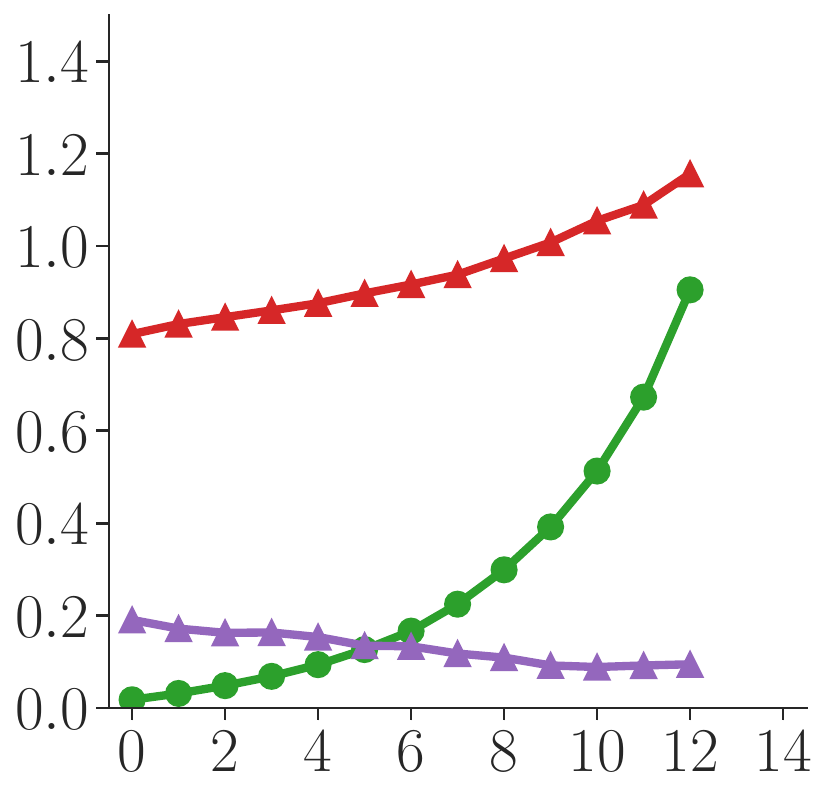} &
	\includegraphics[width=0.19\textwidth]{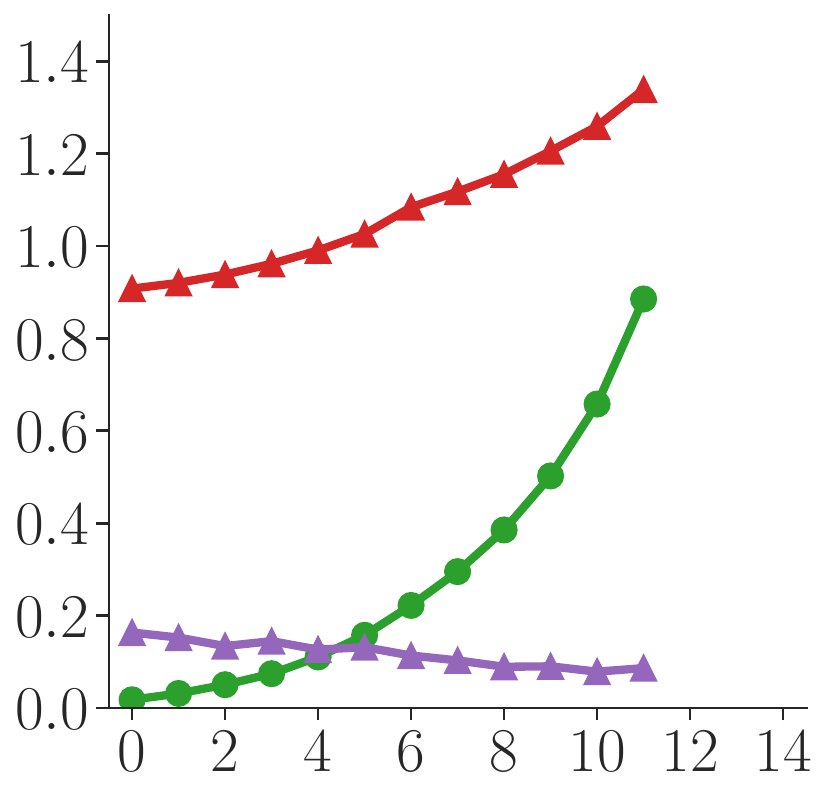} &
	\includegraphics[width=0.19\textwidth]{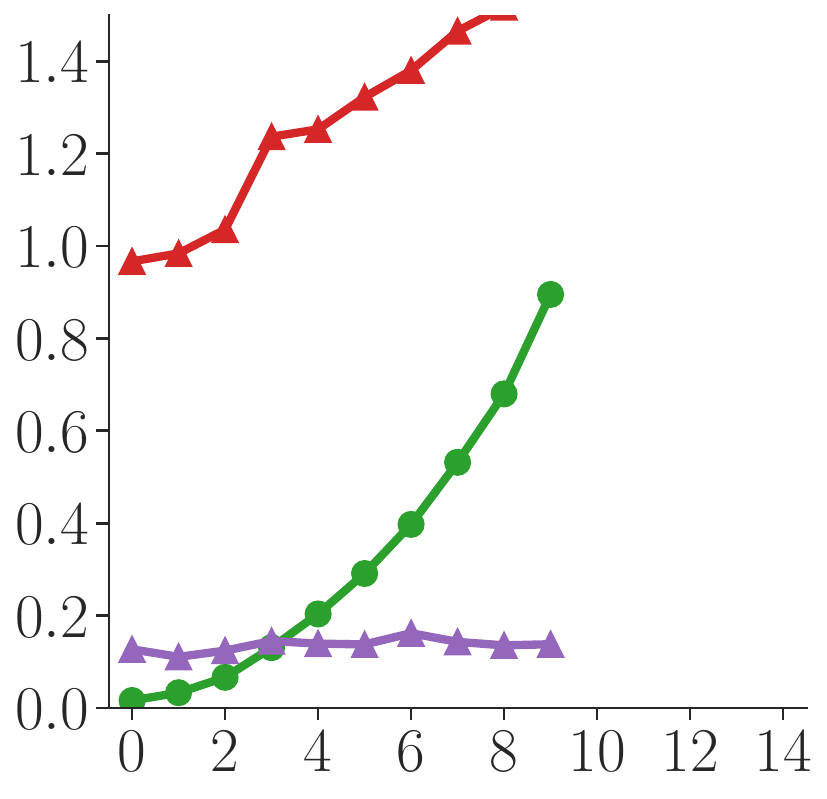} 
	\\
{\rotatebox[origin=l]{90}{~~Policy actions}}
	\includegraphics[width=0.19\textwidth]{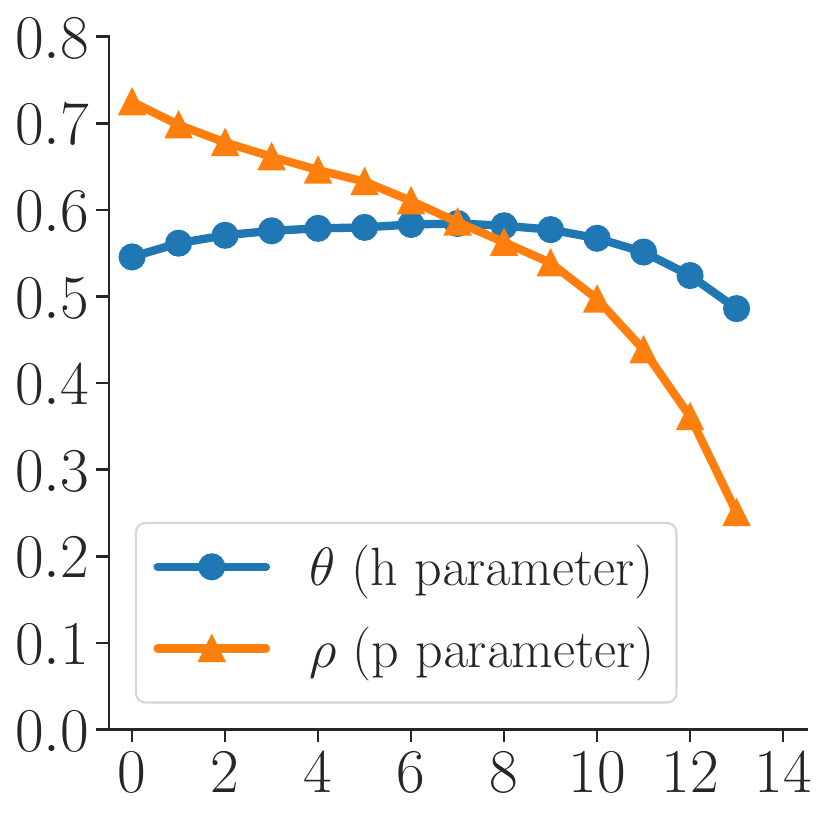} &
	\includegraphics[width=0.19\textwidth]{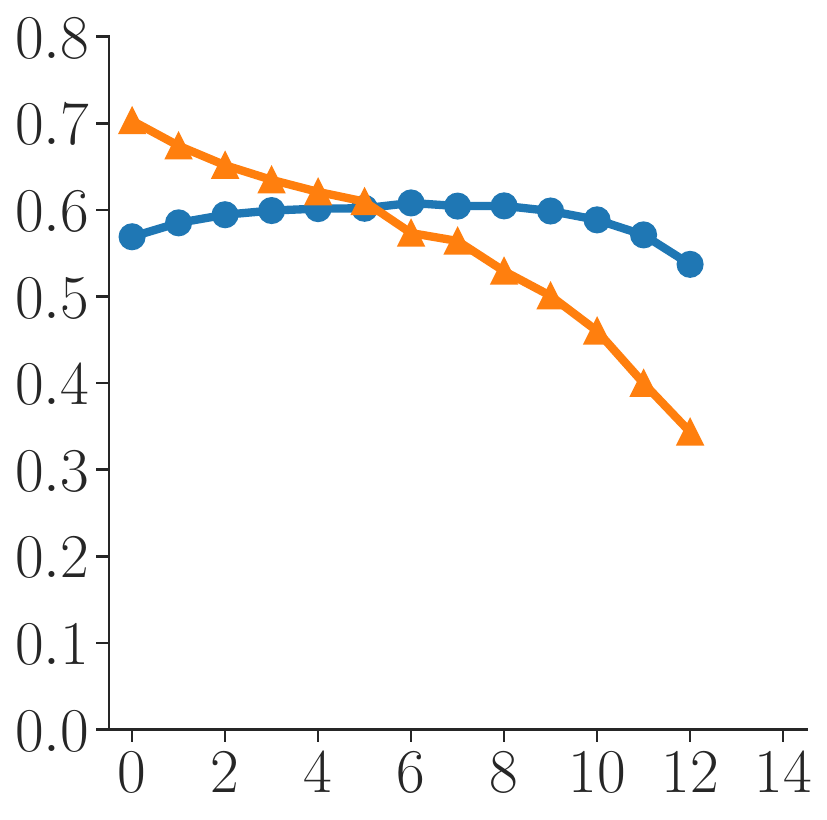} &
	\includegraphics[width=0.19\textwidth]{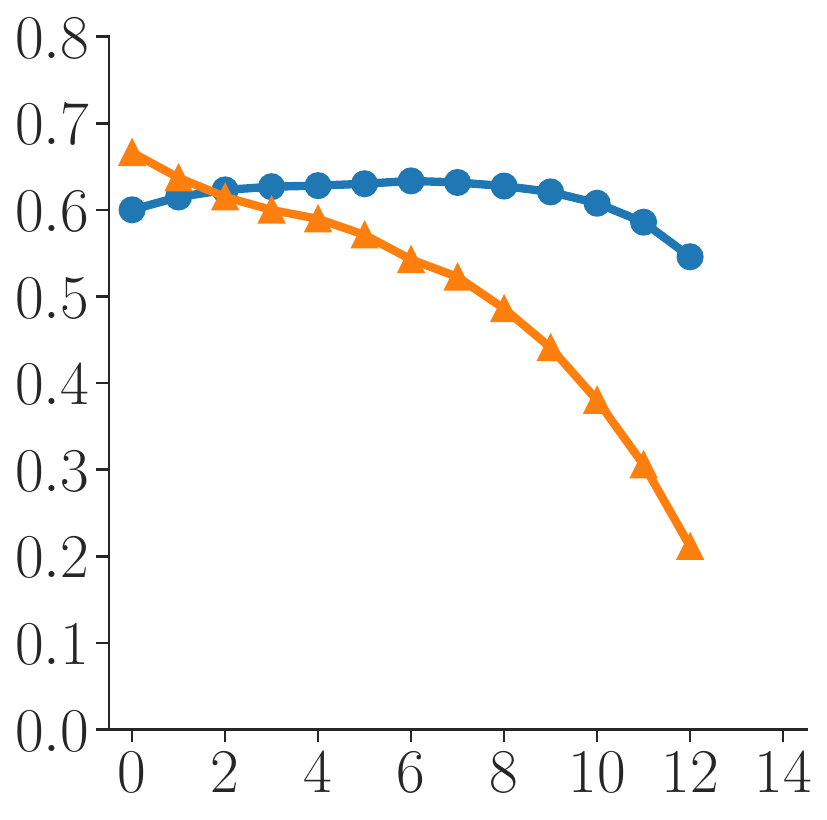} &
	\includegraphics[width=0.19\textwidth]{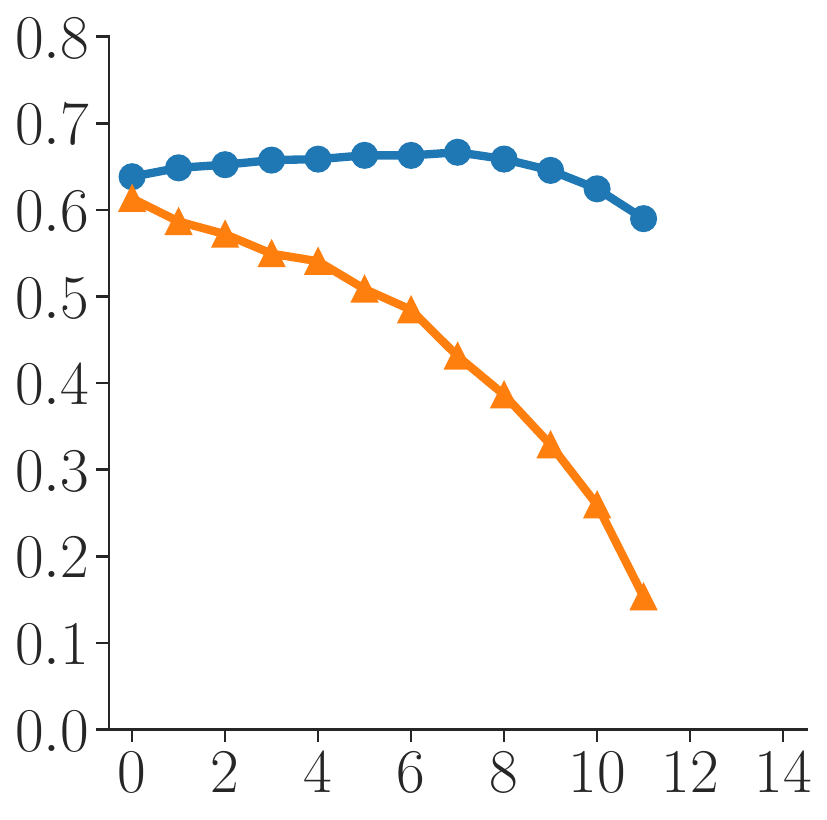} &
	\includegraphics[width=0.19\textwidth]{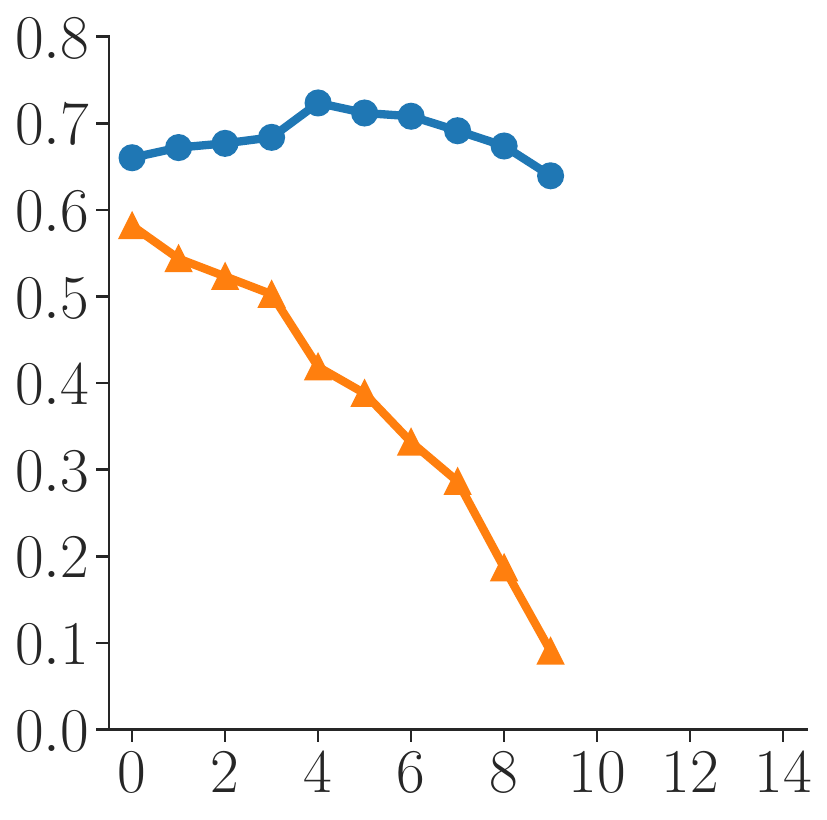} 
	\\
	$\omega = 0.1\pi$ &
	$\omega = 0.3\pi$ &
	$\omega = 0.5\pi$ &
	$\omega = 0.7\pi$ &
	$\omega = 0.9\pi$ 
\end{tabular}
	\caption{Results of a trained $hp$-(AM$)^2$R policy deployed on five different meshes.  \textit{First row:} Initial mesh passed to AFEM MDP. \textit{Second row:} Mesh at conclusion of the deployed (AM$)^2$R policy. Element shade indicates $p$ order (white = 1; black=8).  \rev{\textit{Third row:} Record of the quantities from (\ref{eq:hpobs}) observed by the policy at each step. \textit{Fourth row:} Sequences of actions taken by the policy.  At each step, elements whose error estimate is within $\theta$ of the maximum error estimate are marked for $h$-refinement, while elements whose error estimate lies between $\theta\rho$ and $\theta$ of the maximum error estimate are marked for $p$-refinement.  Visibly, the actions selected by the policy change in response to the different observed quantities on each of the meshes.}
	}
	\label{fig:hp-train}
\end{figure} 

We first train our marking policy on domains with $\omega$ drawn uniformly from $[0.1\pi,0.9\pi]$, representing a range of domains with re-entrant corners and hence solutions with a range of singularities.
The angle drawn is \textit{not} observed by the policy, as we are attempting to learn a marking policy that does not require knowledge of the global geometry.
The training is applied on the accuracy problem \cref{eq:dof_thresh} with threshold $J_\infty=10^{4}$ and observation space \cref{eq:hpobs}.
Once trained, we deploy the policy as described in~\Cref{alg:hpAM2R_2} using the same threshold from training.

In \Cref{fig:hp-train}, we show the effect of the trained policy when deployed on meshes with five different $\omega$ values, spread evenly across the sampling domain including the extreme values ($\omega=0.1\pi$ and $0.9\pi$).
The top row shows the initial state of each mesh.  The \rev{second} row shows the final mesh when the cumulative dof threshold is reached; the shade of an element indicates its assigned order (i.e.,~$p$) on a linear scale of 1 (white) to 8 (black).  
\rev{The third row shows the three observed quantities---budget, empirical mean, and empirical standard deviation as stated in (\ref{eq:hpobs})---at each step of the deployment.}
The bottom row shows the actions of the trained policy during deployment; the blue circle series indicate the $\theta$ values at each step, while the orange triangle series show the $\rho$ values.
In each case, the policy decreases the $\rho$-parameter monotonically, thus increasing emphasis on $p$-refinement, as would be expected to drive down error with maximum efficiency. 
The smooth variation of the actions within a deployment and the moderated adjustments as $\omega$ varies suggest that the policy has been well trained.

\begin{figure}   \centering
    \includegraphics[width=.33\linewidth]{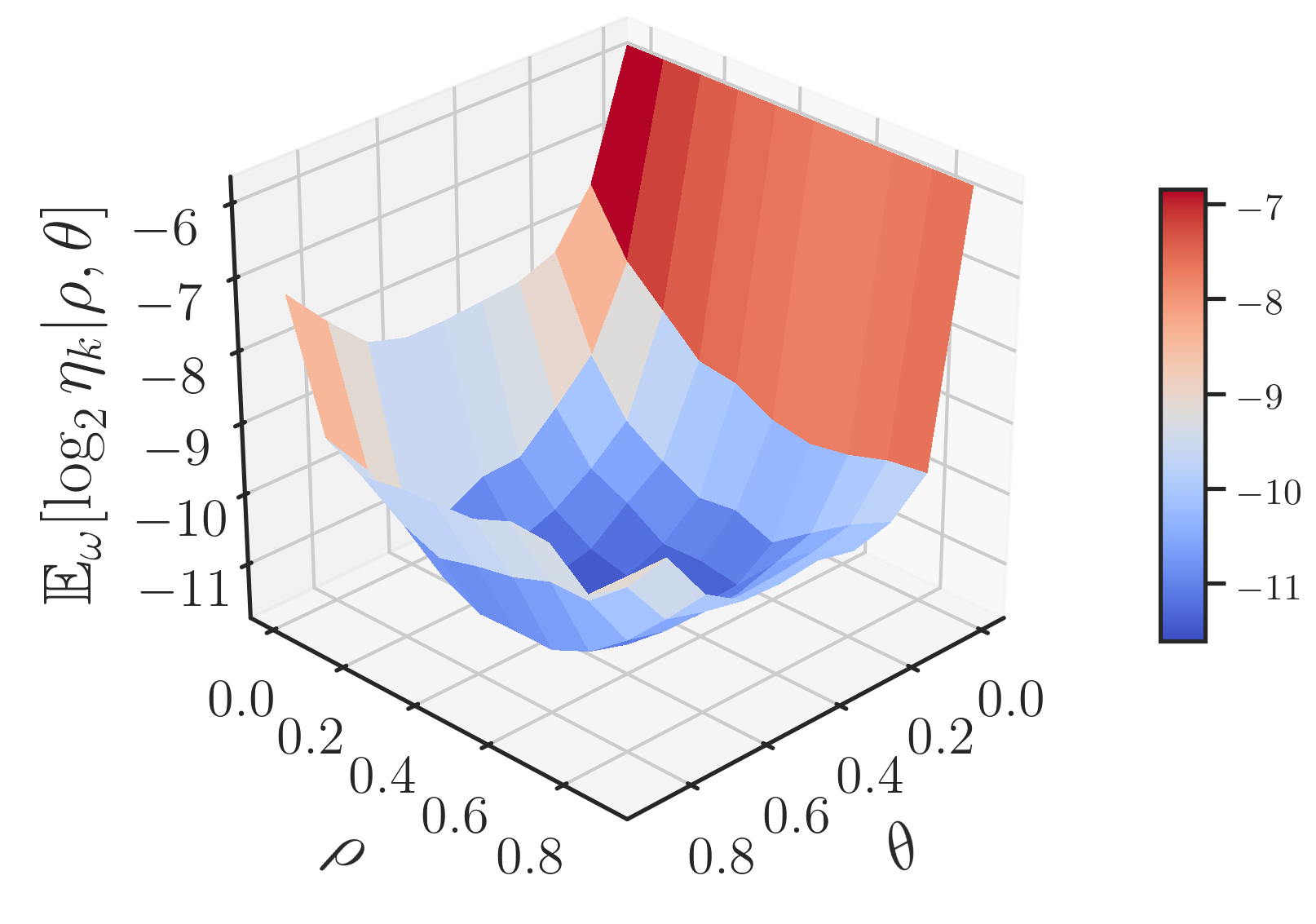}
	\qquad
   \includegraphics[width=.28\linewidth]{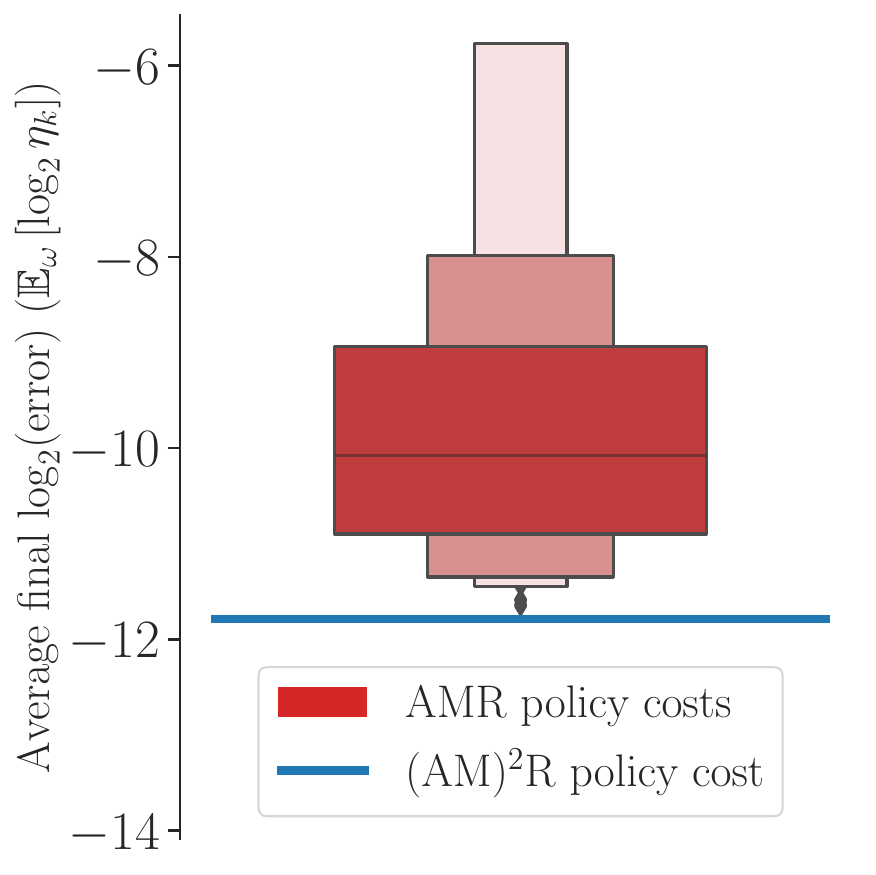}
    \caption{Comparison of the final cost averaged over deployments on 21 meshes with $\omega\in[0.1,0.9]$.    \textit{Left:} The landscape plot of this metric for 100 $(\theta,\rho)$ pairs reveals a convex shape for which $\theta=0.6$, $\rho=0.3$ gives the minimum.
    \textit{Right:} The letterbox plot compares the final average cost metric for the (AM$)^2$R policy (blue line) to those of the 100 AMR policies (red), indicating the ability of our RL methods to discover optimal $hp$-policies.
    Observe that the average final cost of the trained policy is marginally better than the best chosen fixed-parameter policy in these experiments.
    In contrast, the follow-on experiments in \Cref{ssub:ex_transfer} demonstrate the superior \emph{robustness} of the (AM$)^2$R policy.
    }
\label{fig:ex2c-letterbox}
\end{figure}

To compare the results of our deployed policy against a traditional $hp$-AMR policy, we carry out a parameter sweep with fixed choices of $\theta$ and $\rho$, emulating how a practitioner might reasonably choose to set these parameters in a typical computational setting.
We consider 100 distinct policies corresponding to $\theta,\rho\in\{0.0,\ldots,0.9\}$ and apply each policy to 21 meshes defined by $\omega\in \{(0.1 + 0.04k)\pi\}$ for $k\in\{0,\ldots, 20\}$.
We then select the policy that has the lowest \textit{average} final cost (i.e., $\mathbb{E}_\omega[\log_2\eta_\rev{k_\text{final}}]$) across all 21 meshes and consider that as the ``optimal'' traditional policy.
By this method, $\theta=0.6$ and $\rho=0.3$ are determined to be optimal.
As shown in~\Cref{fig:ex2c-letterbox}, the average final cost metric for the trained $hp$-(AM$)^2$R policy is slightly better than the best fixed parameter AMR policy.

For a more nuanced comparison between the two policies, we record the error estimate $\eta_k$ at the final mesh of the AFEM MDPs for each mesh shown in~\Cref{fig:hp-train}.  
The results are stated in the ``training'' row of~\Cref{tab:hp-results}.
To measure the improvement (or decline) in error estimate $\eta_k$ by switching from the AMR policy to the (AM$)^2$R policy, we define the \textit{improvement exponent} and \textit{improvement factor} metrics as follows:
\begin{multline}
\label{eq:impmetrics}
\text{improvement exponent}  := \log_2\left(\text{improvement factor} \right)\\ := \log_2\left(\frac{\text{final $\eta_k$, $hp$-AMR}}{\text{final $\eta_k$, $hp$-(AM$)^2$R}}\right).
\end{multline}
Thus, switching to (AM$)^2$R is favorable if the improvement factor is greater than 1 or, equivalently, if the improvement exponent is greater than 0.

\begin{table}
\small
\centering
\setlength\tabcolsep{2mm}
\begin{tabular}{ll|llll}
&                        & Final error    		& Final error          & Improvement & Improvement \\
&                        & estimate ($\eta_k$),  & estimate ($\eta_k$), & factor 	 & exponent  \\
& 						& $hp$-AMR   			& $hp$-(AM$)^2$R \\[1mm]
                        \hline
\text{}\\[-2mm]
\multirow{5}{*}{\rotatebox[origin=c]{90}{Training}} & Disk, $\omega=0.1 \pi$  &  \texttt{0.002613}  &  \texttt{0.001665}  &  \texttt{1.57}   & \texttt{ 0.65}  \\
& Disk, $\omega=0.3 \pi$  &  \texttt{0.001018}  &  \texttt{0.001106}  &  \texttt{0.92}   & \texttt{-0.12}  \\
& Disk, $\omega=0.5 \pi$  &  \texttt{0.000338}  &  \texttt{0.000280}  &  \texttt{1.21}   & \texttt{ 0.27}  \\
& Disk, $\omega=0.7 \pi$  &  \texttt{0.000110}  &  \texttt{0.000081}  &  \texttt{1.36}   & \texttt{ 0.44}  \\
& Disk, $\omega=0.9 \pi$  &  \texttt{0.000027}  &  \texttt{0.000022}  &  \texttt{1.21}   & \texttt{ 0.27}  \\[1mm]
\hline
\text{}\\[-2mm]
\multirow{7}{*}{\rotatebox[origin=c]{90}{Testing}}
& L-shape                 &  \texttt{0.000258}  &  \texttt{0.000190}  &  \texttt{1.36}   & \texttt{ 0.44}  \\
& Staircase               &  \texttt{0.002577}  &  \texttt{0.001609}  &  \texttt{1.60}   & \texttt{ 0.68}  \\
& Staircase-Tri           &  \texttt{0.005719}  &  \texttt{0.005601}  &  \texttt{1.02}   & \texttt{ 0.03}  \\
& Star                    &  \texttt{0.001137}  &  \texttt{0.001307}  &  \texttt{0.87}   & \texttt{-0.20}  \\
& Disk, $\omega=1.5 \pi$  &  \texttt{0.000015}  &  \texttt{0.000009}  &  \texttt{1.69}   & \texttt{ 0.76}  \\
& Fichera                 &  \texttt{0.010979}  &  \texttt{0.007431}  &  \texttt{1.47}   & \texttt{ 0.56}  \\
\end{tabular}
	\caption{Final error estimates for the optimal traditional $hp$-AMR policy ($\theta=0.6$, $\rho=0.3$) and an (AM$)^2$R policy on a variety of meshes.  The first five rows (disk domains with $\omega\in[0.1\pi,0.9\pi]$) were included in the training regimen for the (AM$)^2$R policy but the remaining rows were not, thus demonstrating the robustness or generalizability of the trained policy.  The (AM$)^2$R policy outperforms the traditional fixed-parameter policy in all but one instance from the training set ($\omega=0.3\pi$) and one instance from the test set (Star).}
	\label{tab:hp-results}
\end{table} 

For each $\omega$ value except $0.3\pi$, we see improvement factors over 1.2, meaning the final error estimate is reduced by a factor of at least 1.2 when switching to the (AM$)^2$R policy.
Since all other variables were held constant, such improvement is directly attributable to the ability of the policy to dynamically adjust the marking parameter values.
For $\omega=0.3\pi$, the AMR policy has a slightly better final error estimate, reflecting the fact that improved performance on average does not ensure improved performance in every instance. 
Still, for selecting a policy that performs well over a range of geometries, the (AM$)^2$R policy is certainly the better choice.

\subsection{Deploying a trained policy in new settings} \label{ssub:ex_transfer}

We next deploy the trained (AM$)^2$R policy, without modification, on different types of domains, PDE problems, and dimensions.
In \Cref{fig:hp-deploy}, we show five ``testing'' domains, none of which were used when training the (AM$)^2$R policy.
For the L-shape domain and $\omega=1.5\pi$ cases, we use the same Laplace problem defined in \Cref{ssub:ex1a,ssub:ex2c}; for the other domains we use Poisson's equation with zero Dirichlet boundary conditions, i.e.,
\begin{equation}
\label{eq:PoissonEqn}
	\Delta u = 1
	\quad \text{in } \Omega,
	\qquad
	u = 0 \quad \text{on } \partial\Omega.
\end{equation}
\Cref{fig:hp-deploy} recapitulates the conventions from~\Cref{fig:hp-train}.
The general trends of the policy actions on these domains are similar to those of the disk domains seen during training in \Cref{ssub:ex2c}, namely, $\rho$ trends towards zero while $\theta$ changes only slightly. 
Visibly, the rate at which both parameters change is dependent on the mesh, and, implicitly, the pace of exhausting the relative budget grows; cf.~\cref{eq:AccuracyBudget}.

\begin{figure}
\centering
\setlength\tabcolsep{0mm}
\begin{tabular}{ccccc} 
{\rotatebox[origin=l]{90}{Initial mesh}}
	\includegraphics[width=0.19\textwidth]{initial_lshape.png} &
	\includegraphics[width=0.19\textwidth]{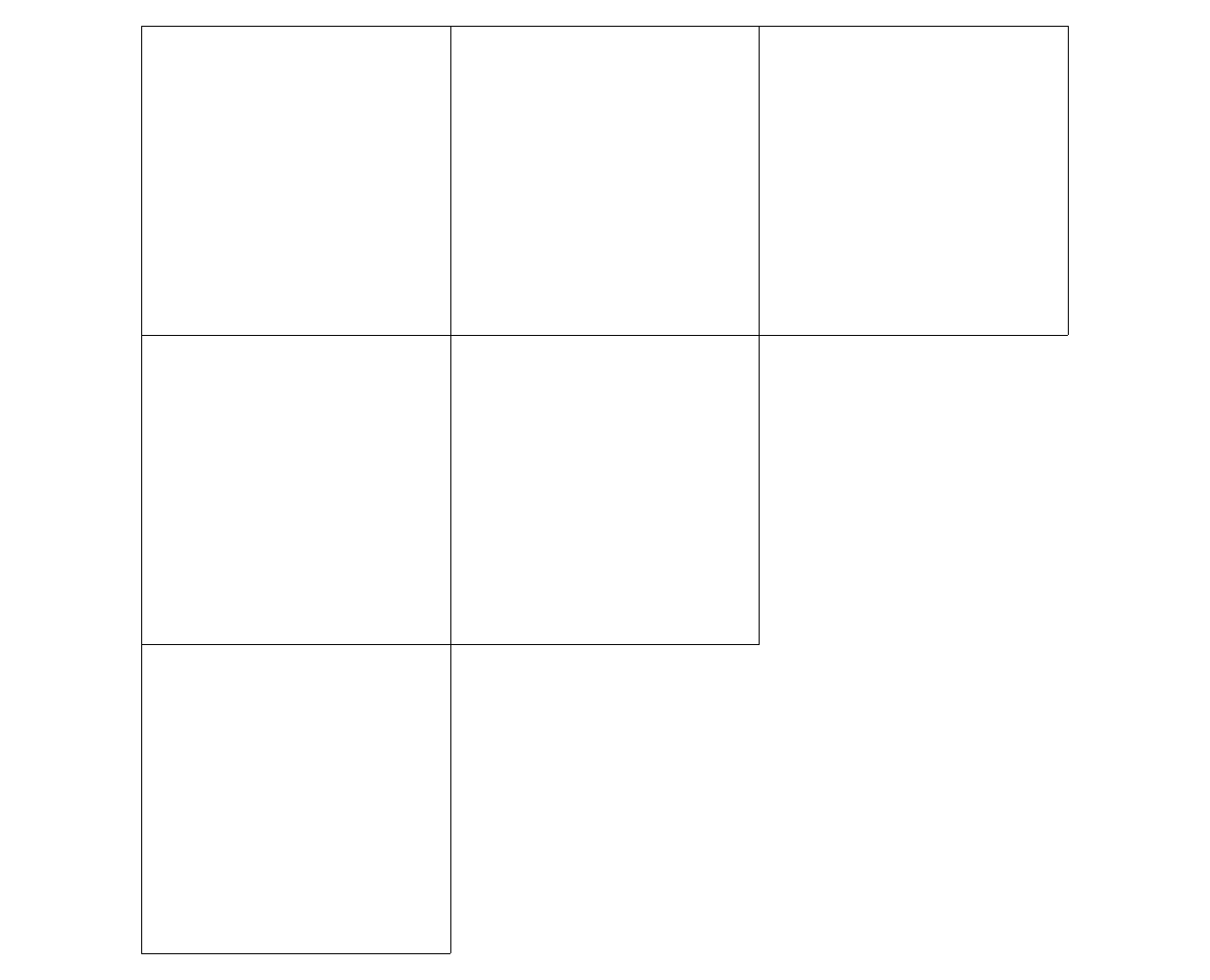} &
	\includegraphics[width=0.19\textwidth]{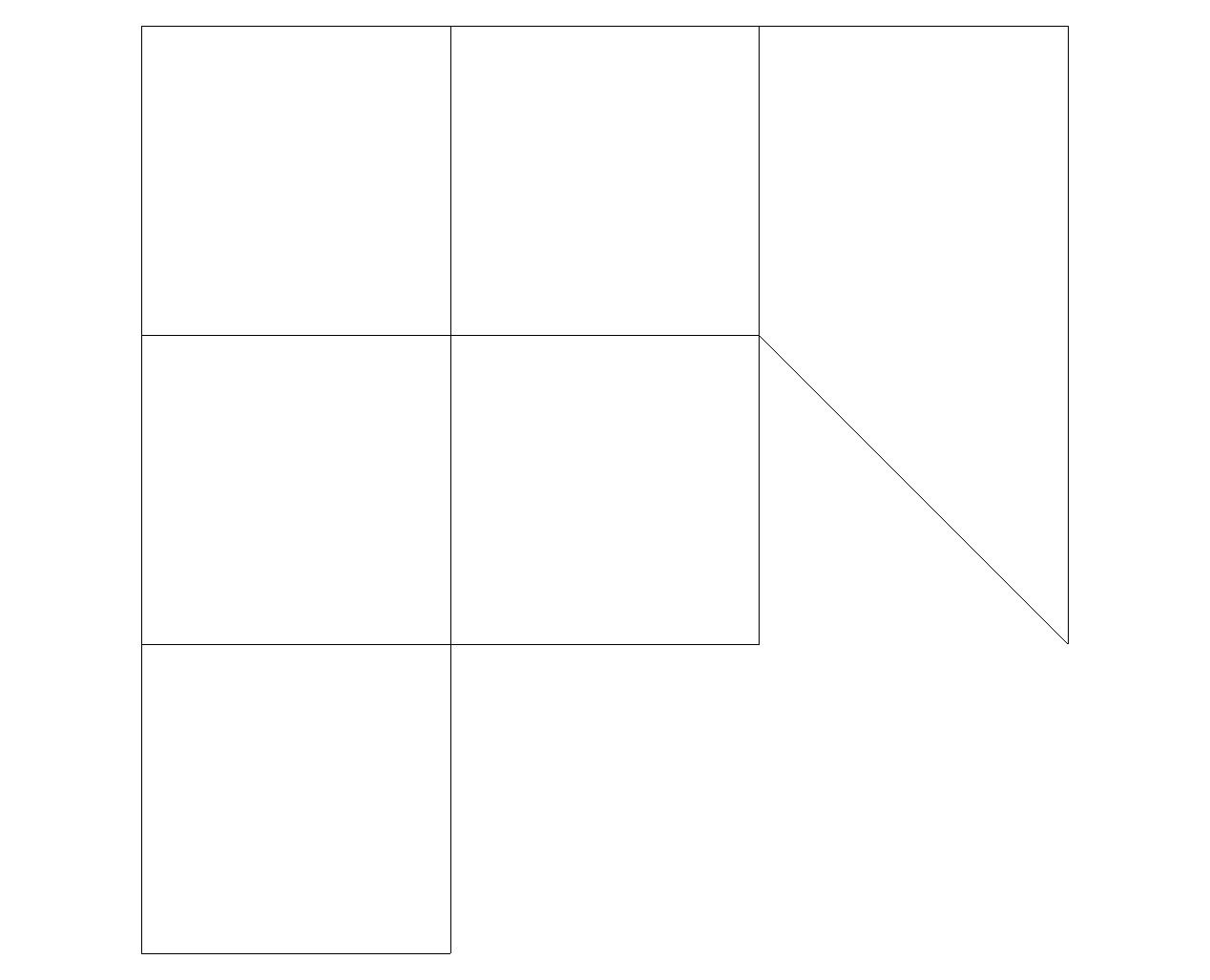} &
	\includegraphics[width=0.19\textwidth]{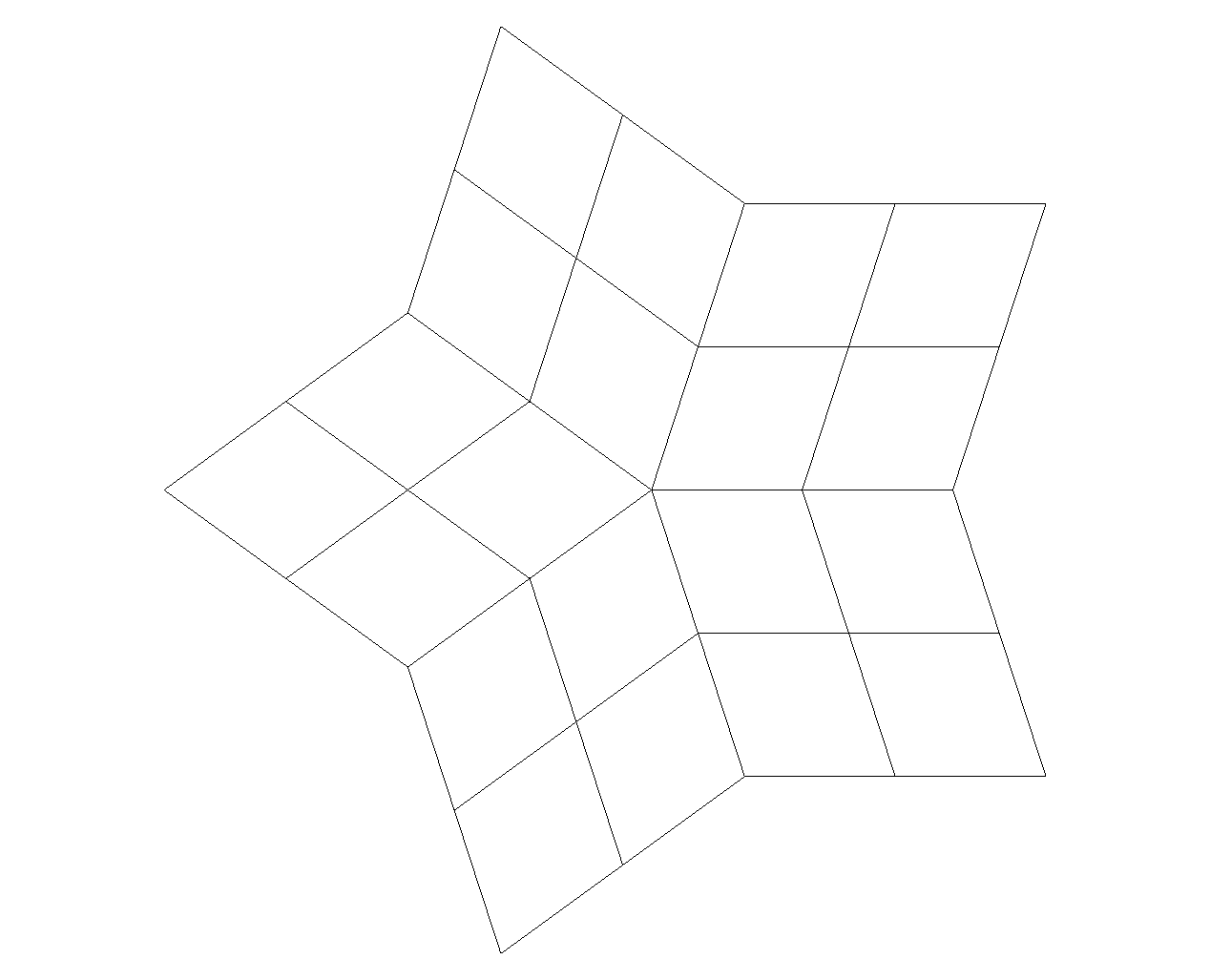} &
	\includegraphics[width=0.19\textwidth]{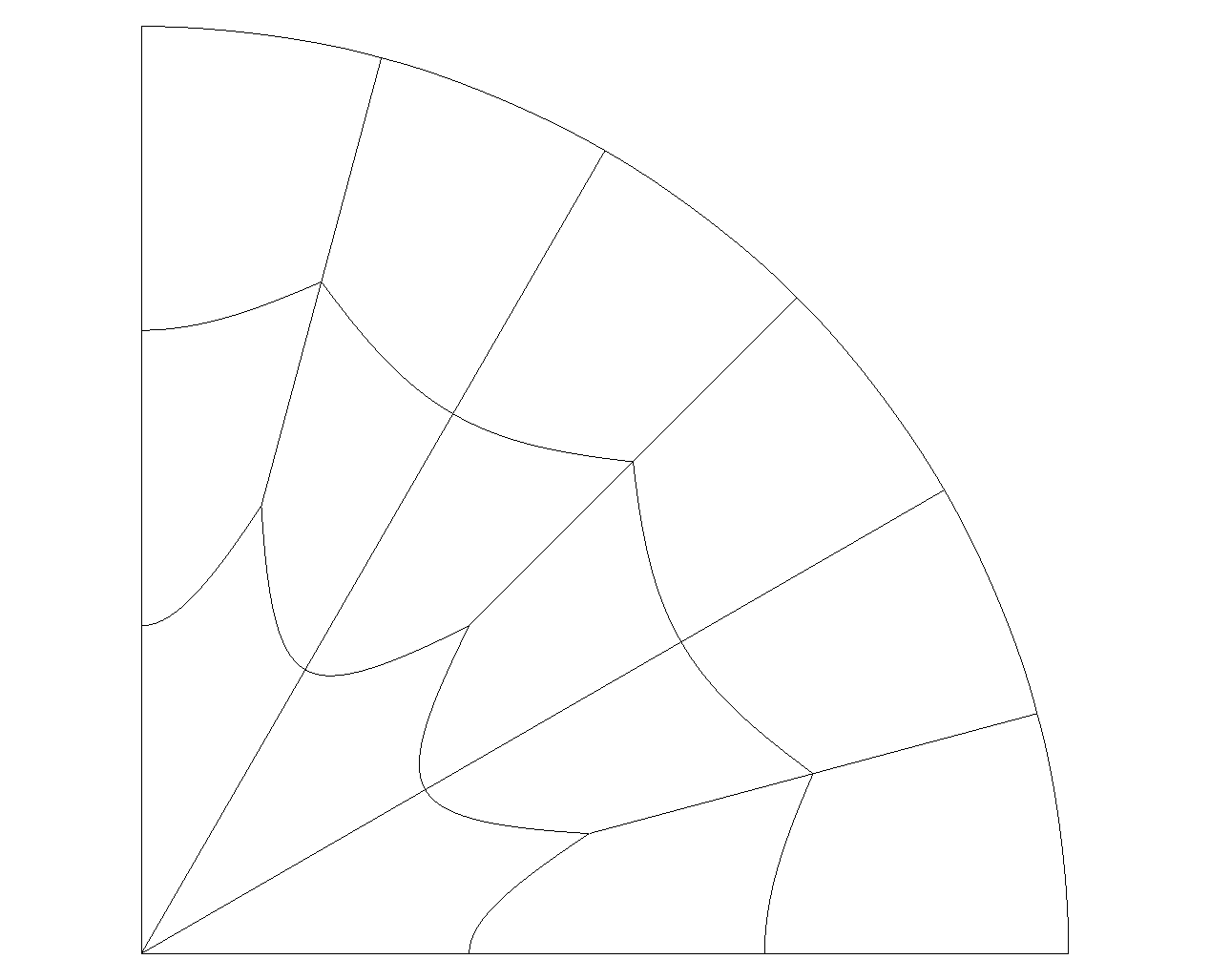} 
	\\
{\rotatebox[origin=l]{90}{~~Final mesh}}
	\includegraphics[width=0.19\textwidth]{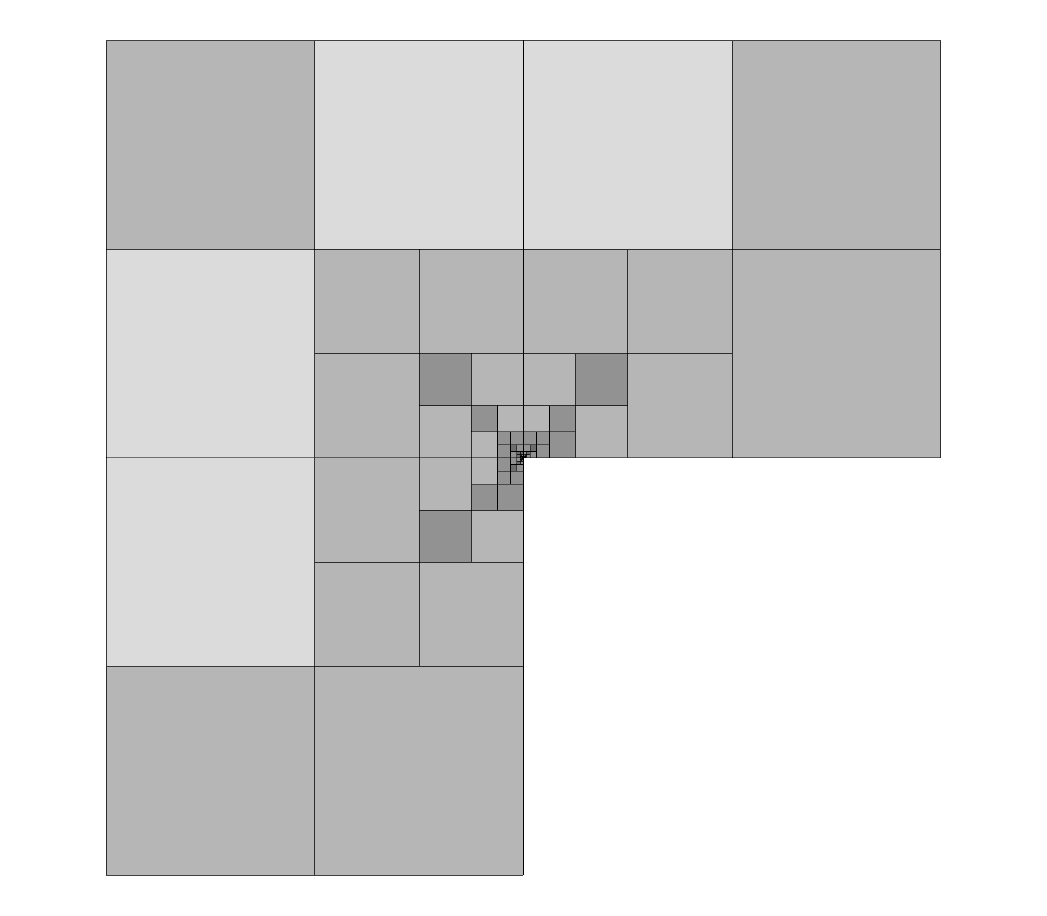} &
	\includegraphics[width=0.19\textwidth]{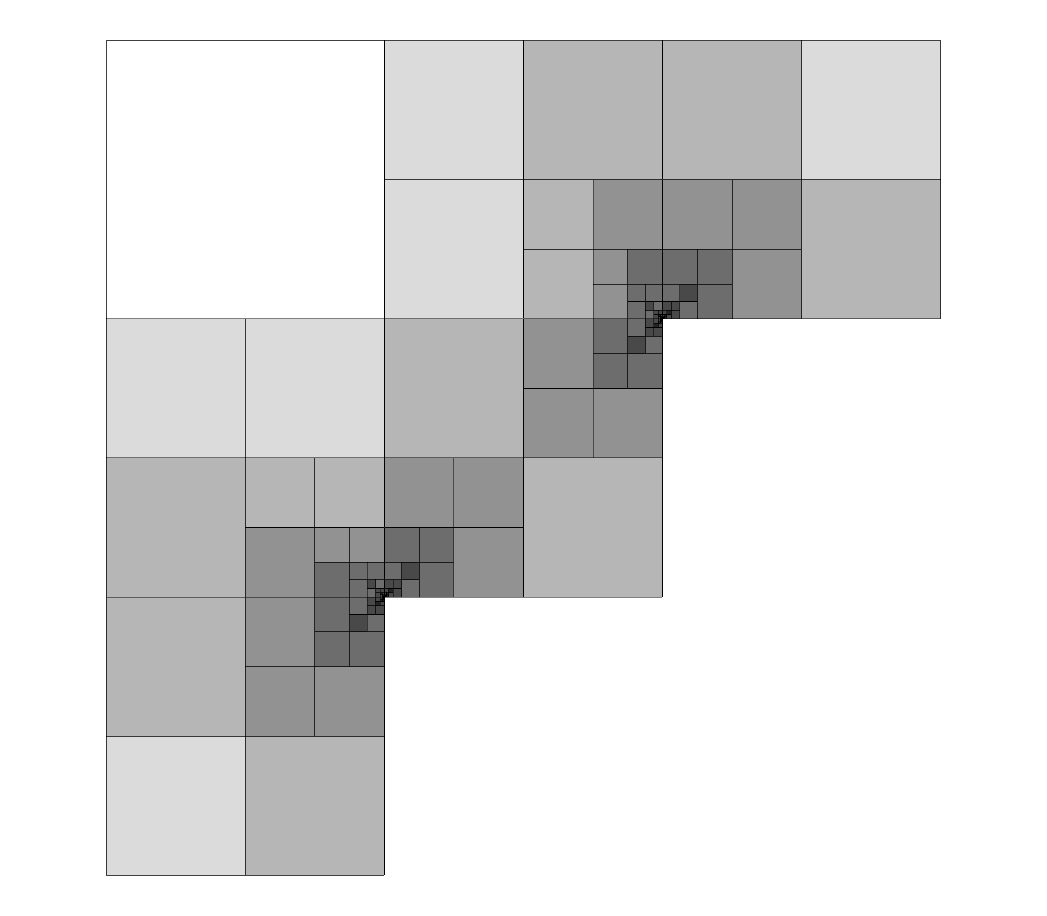} &
	\includegraphics[width=0.19\textwidth]{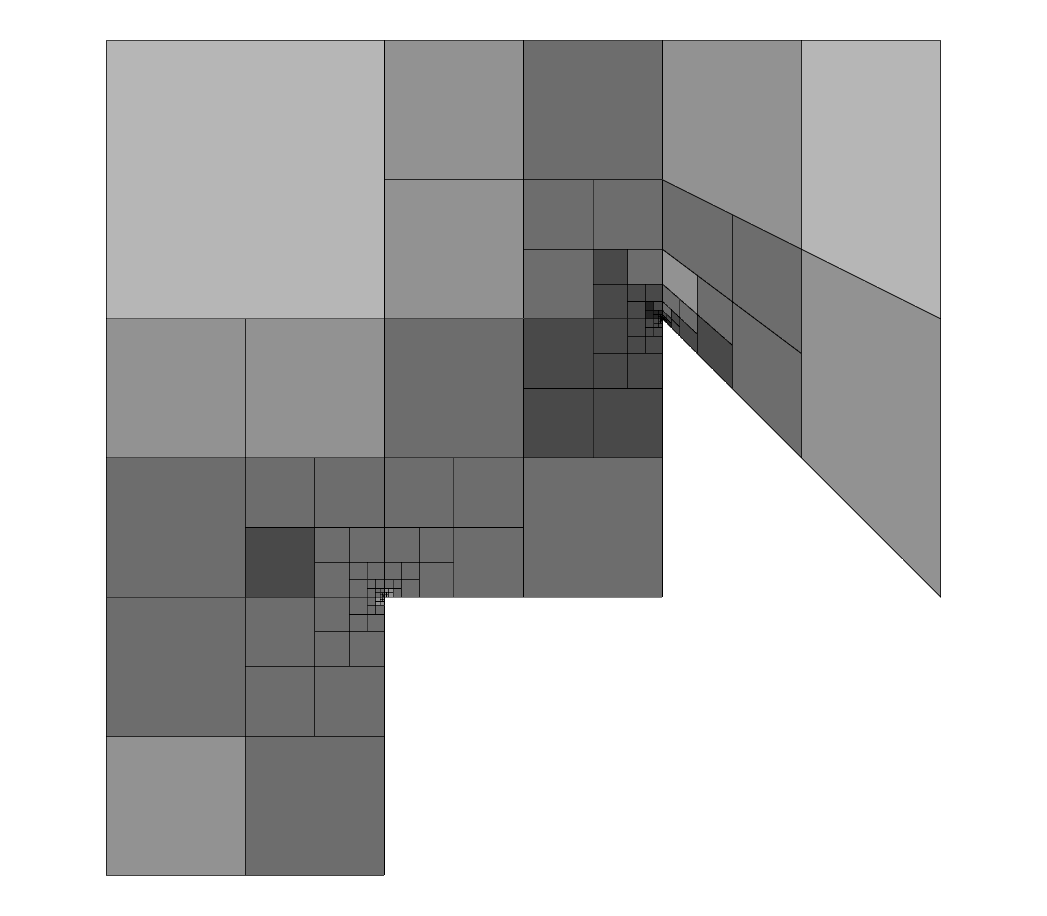} &
	\includegraphics[width=0.19\textwidth]{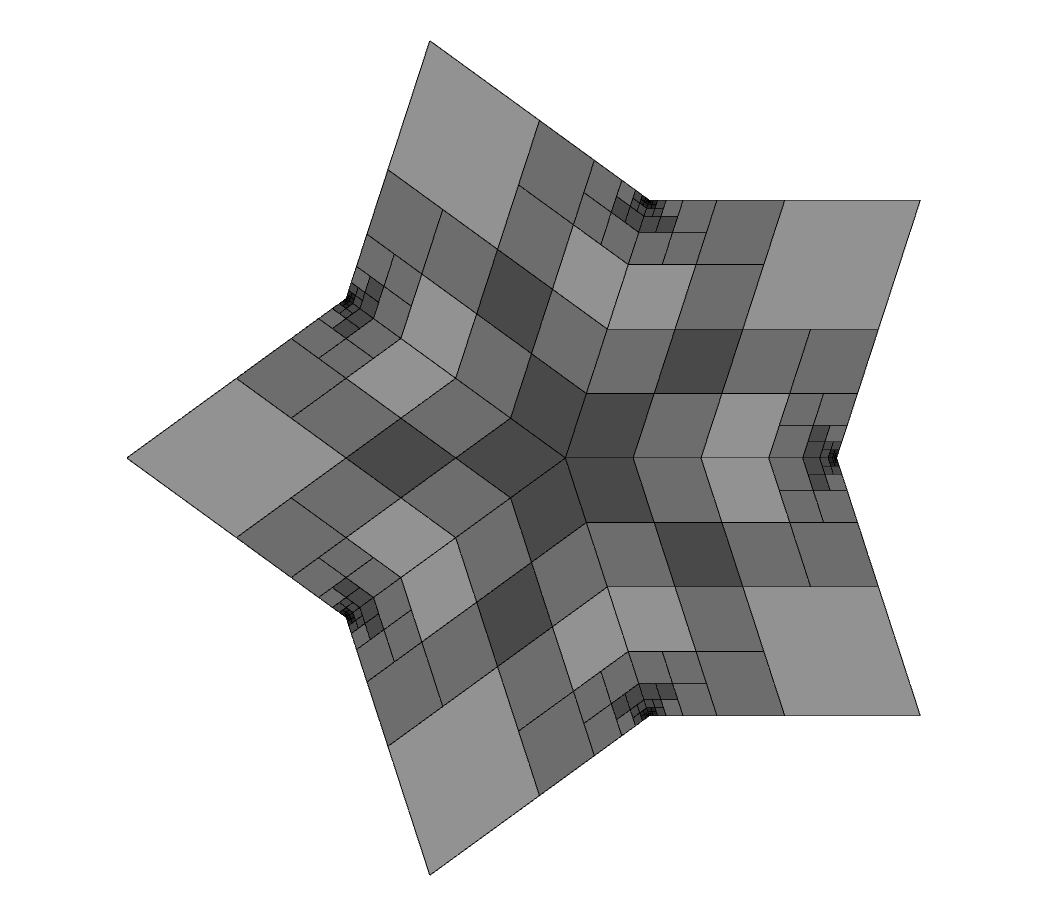} &
	\includegraphics[width=0.19\textwidth]{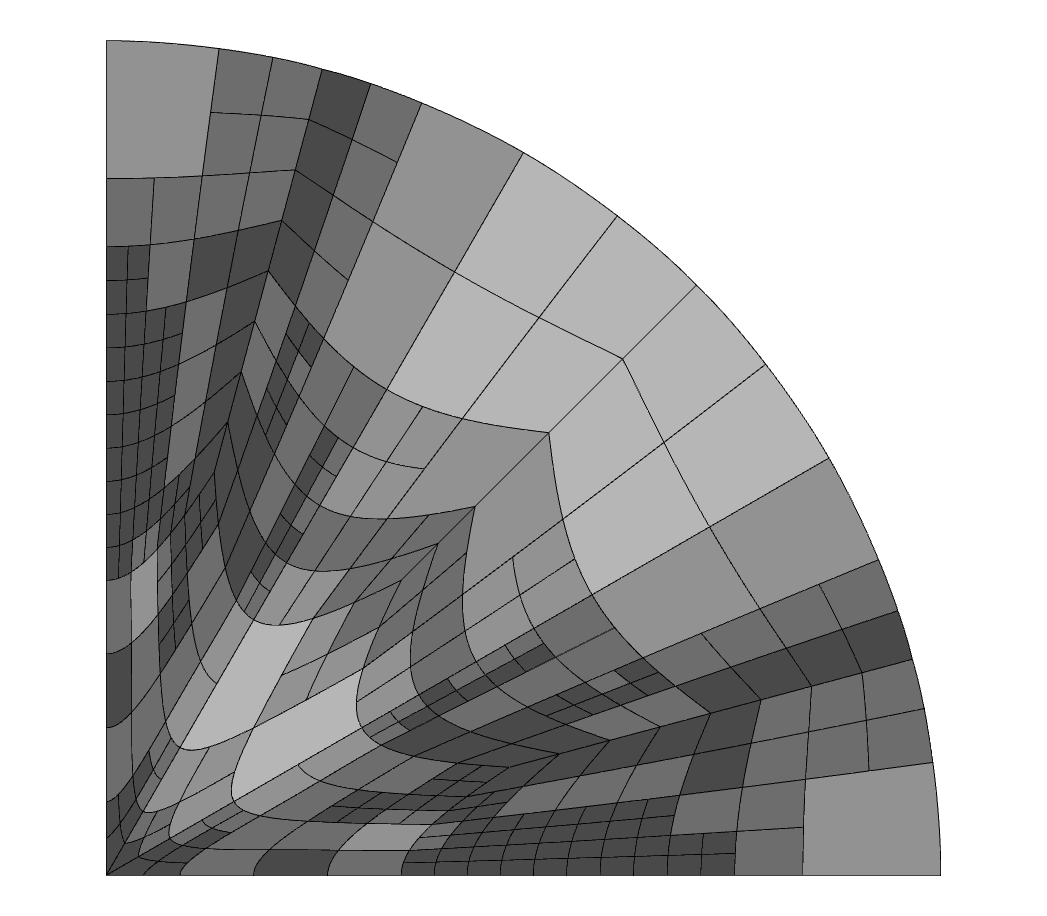}
	\\	
	{\rotatebox[origin=l]{90}{~~Observables}}
	\includegraphics[width=0.19\textwidth]{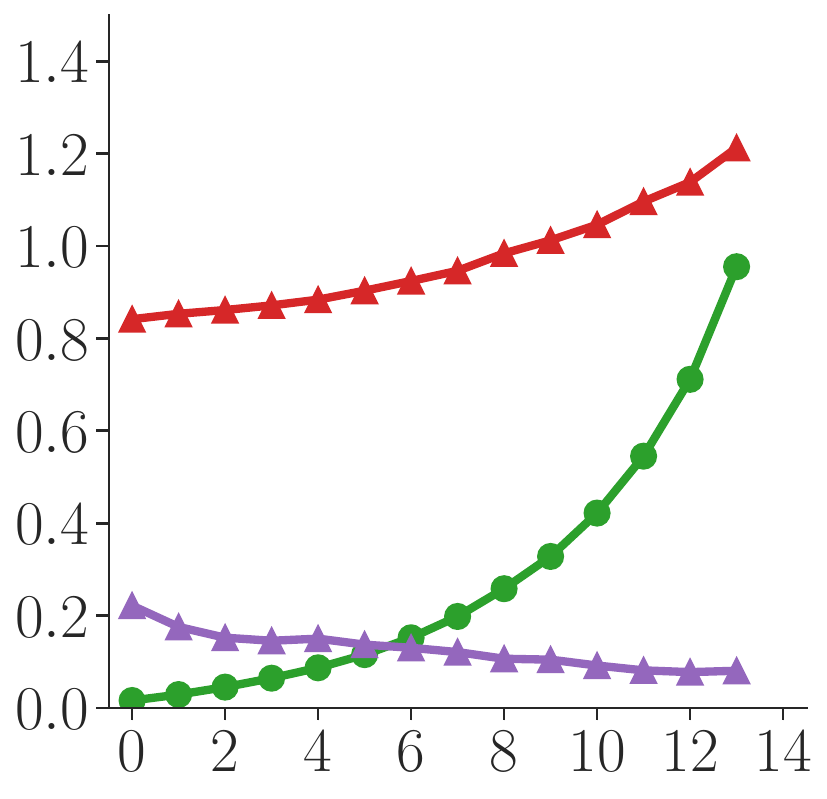} &
	\includegraphics[width=0.19\textwidth]{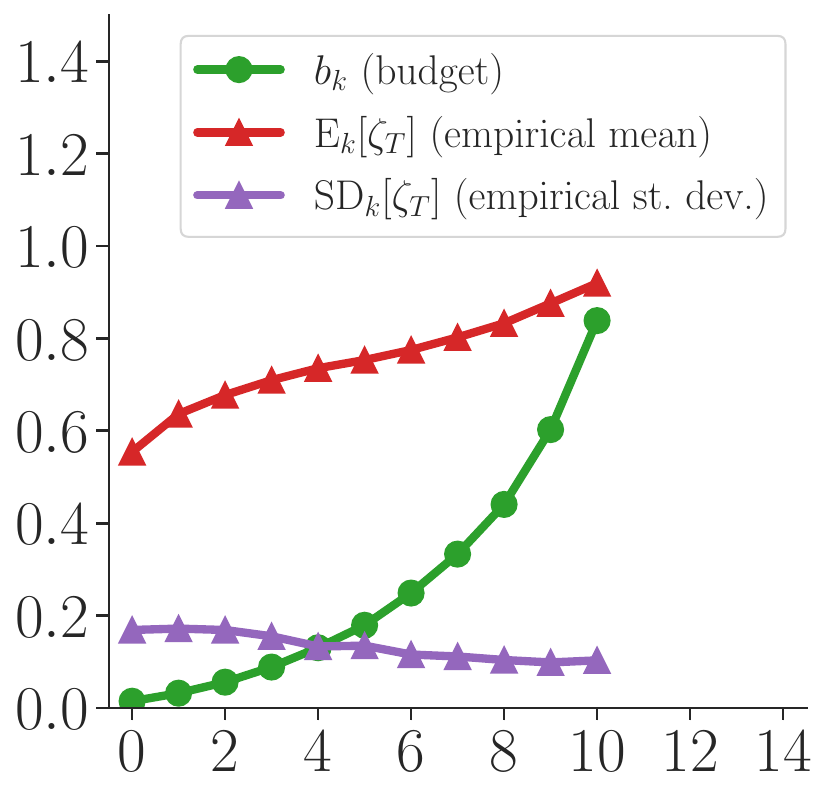} &
	\includegraphics[width=0.19\textwidth]{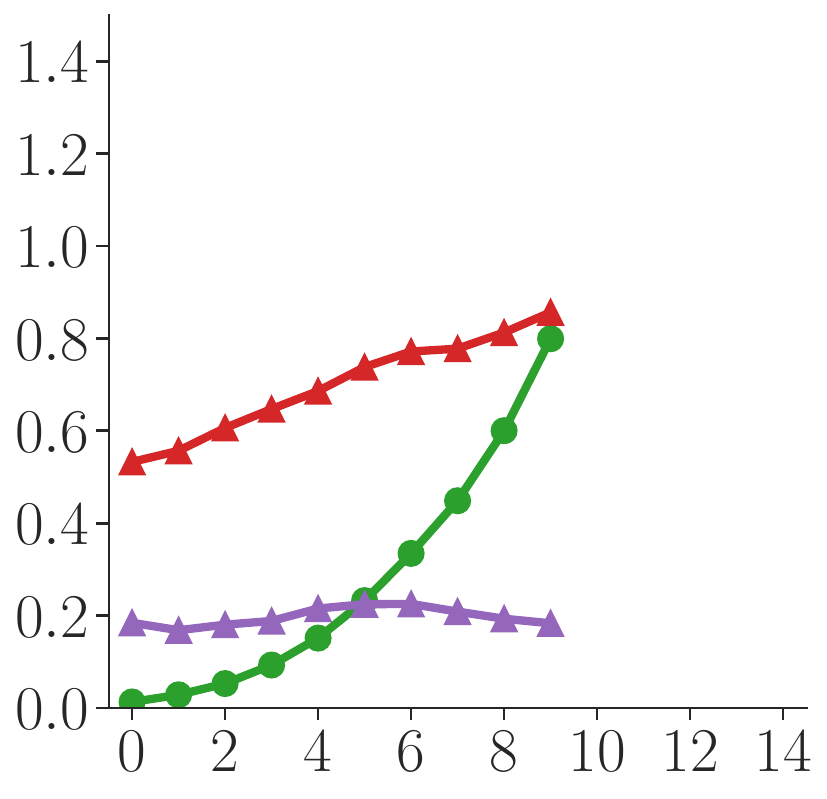} & 
	\includegraphics[width=0.19\textwidth]{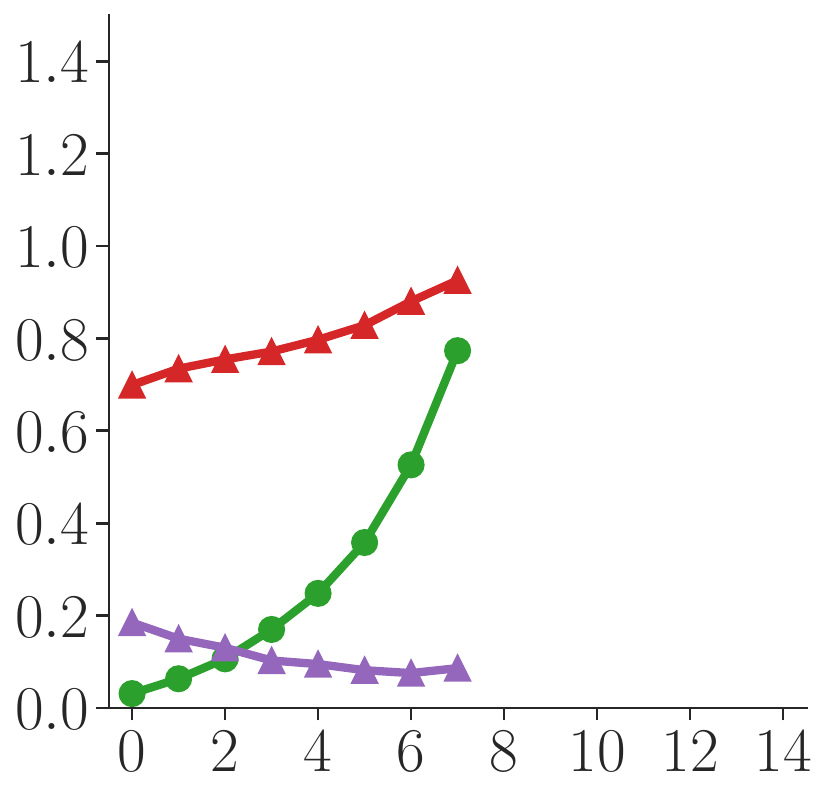} &
	\includegraphics[width=0.19\textwidth]{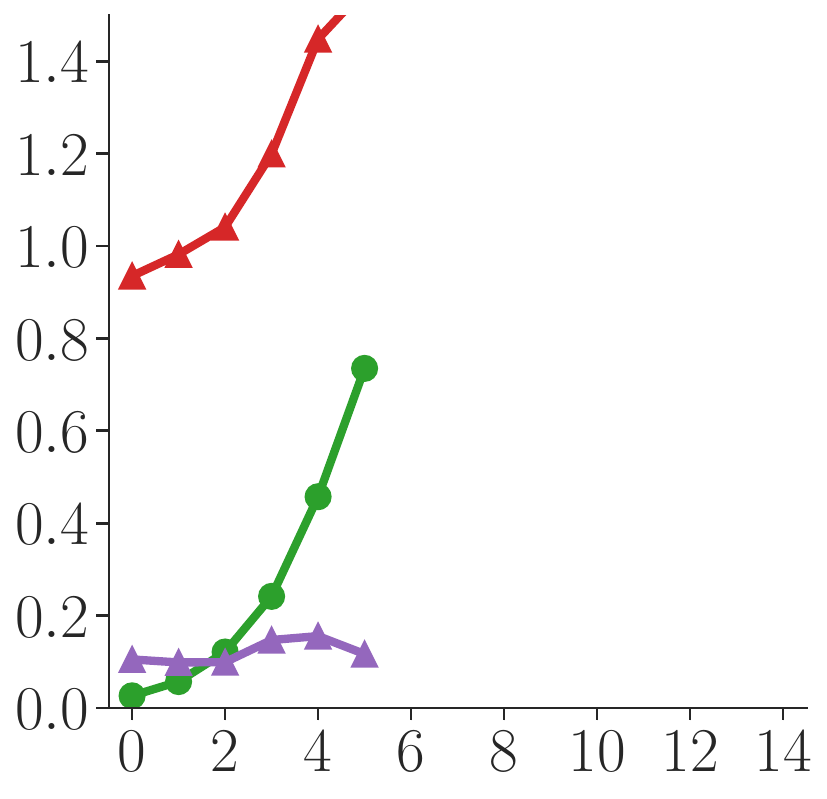} 
	\\
{\rotatebox[origin=l]{90}{~~Policy actions}}
	\includegraphics[width=0.19\textwidth]{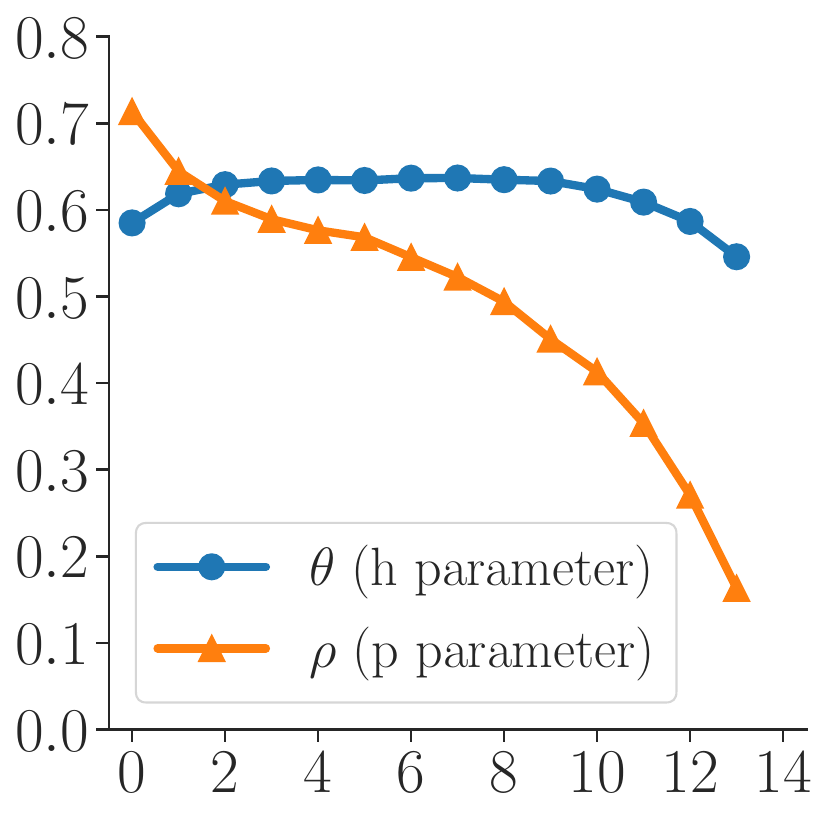} &
	\includegraphics[width=0.19\textwidth]{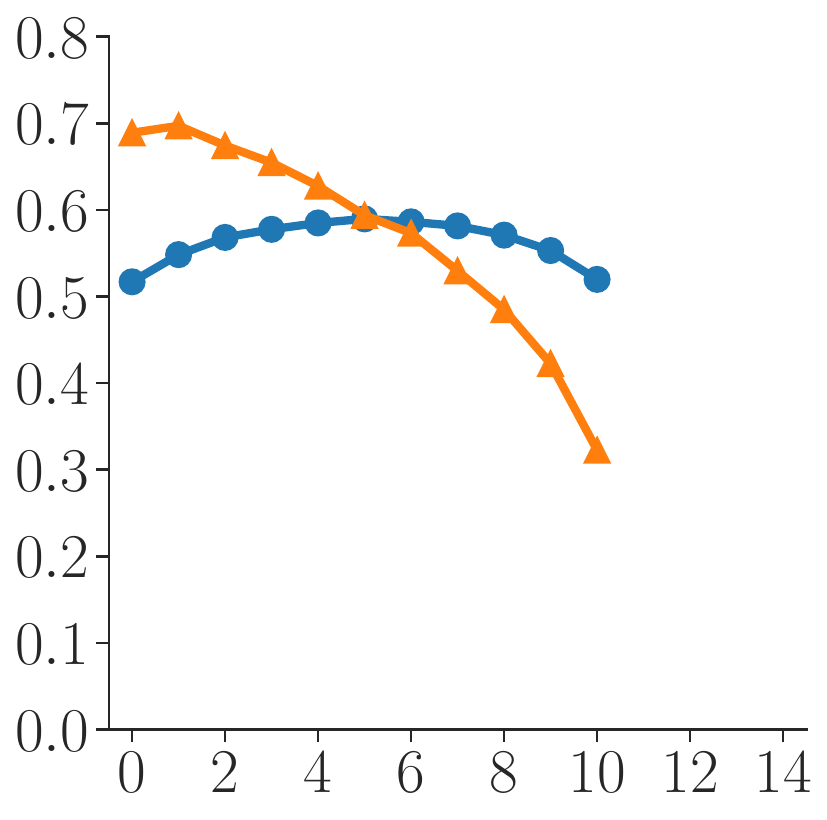} &
	\includegraphics[width=0.19\textwidth]{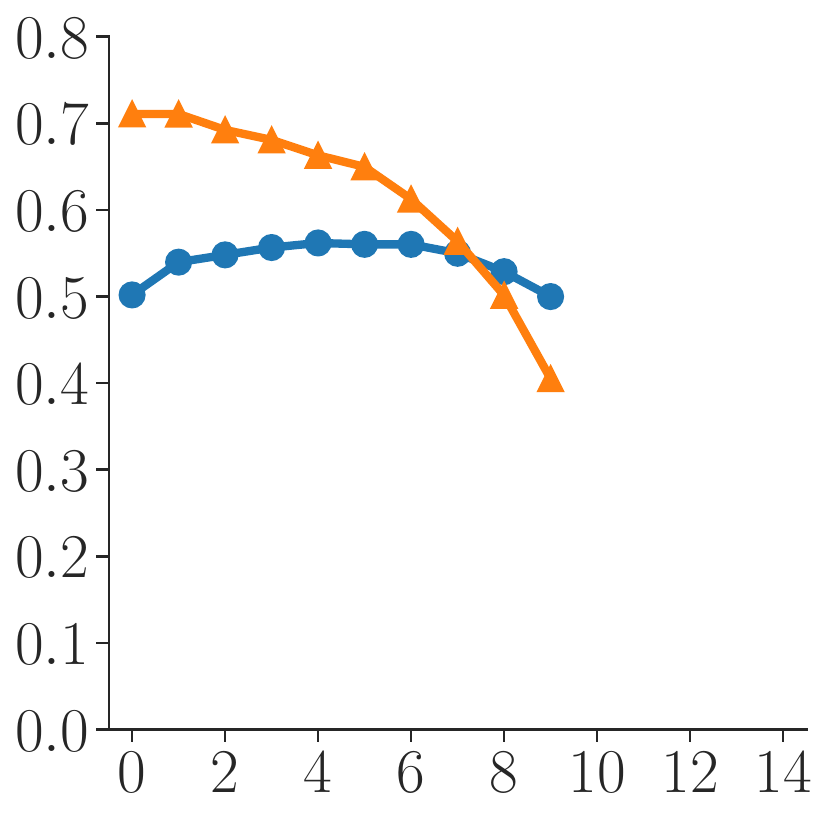} & 
	\includegraphics[width=0.19\textwidth]{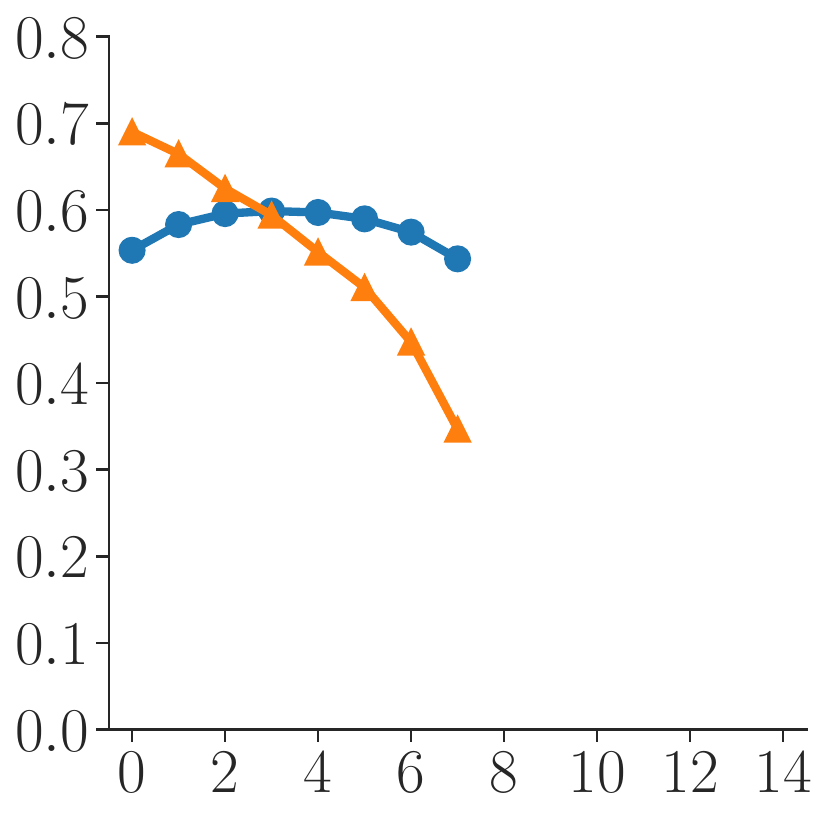} &
	\includegraphics[width=0.19\textwidth]{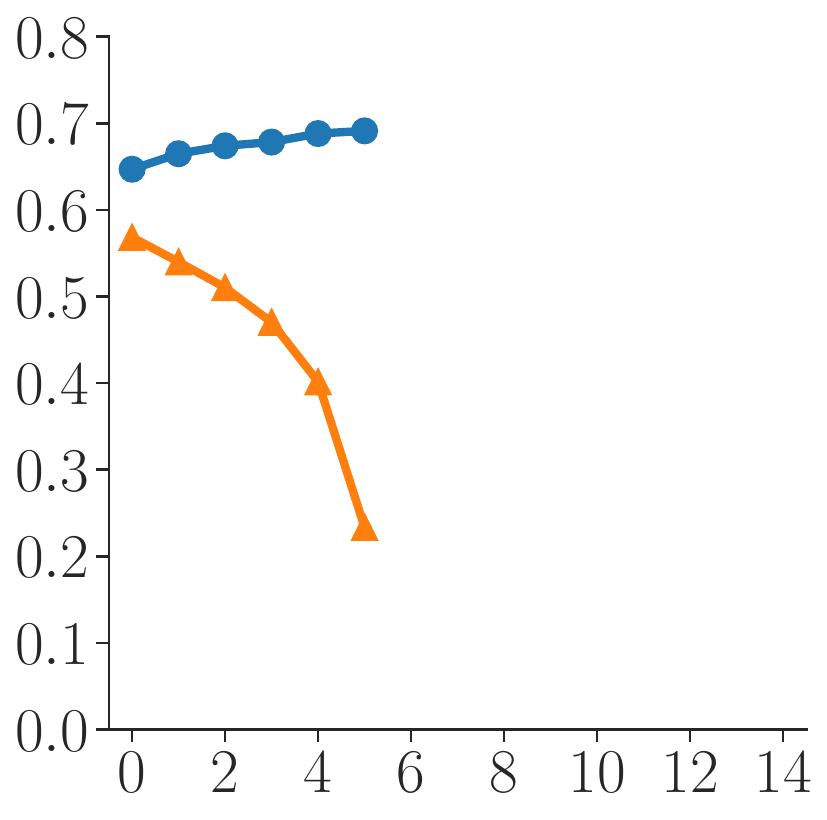} 
	\\	
	L-shape &
	Staircase &
	Staircase-Tri &
	Star &
	$\omega=1.5\pi$
	\end{tabular}
	\caption{Results of the trained $hp$-(AM$)^2$R policy deployed on five meshes, none of which were part of the meshes from the training regime. Element shade indicates $p$ order (white = 1; black=8). On the L-shaped and $\omega=1.5\pi$ disk domains we solve the Laplace problem~\cref{eq:LaplaceEqn}, and on the remaining domains we solve the Poisson problem~\cref{eq:PoissonEqn}.  The responsiveness of the trained policy to changes in the geometry and PDE setting is evident from the variety of policy actions observed.
}
	\label{fig:hp-deploy}
\end{figure}

In the ``testing'' row of~\Cref{tab:hp-results} we show the final error estimates for the five meshes from \Cref{fig:hp-deploy}.
The (AM$)^2$R policy produces a lower final global error estimate in every case except the Star mesh, for which it still produces an error estimate of equivalent order of accuracy.
\rev{For the Star mesh case, we speculate that the presence of re-entrant corners on opposing sides of the domain induces a pre-asymptotic response to refinement that is different from those of the other initial meshes; such effects could be explored further in future work.}
The example of a disk with $\omega=1.5\pi$ is notable; the solution to this problem has no singularities, and thus an optimal marking policy should move toward uniform $hp$-refinement quickly.
In agreement with this intuition, the (AM$)^2$R policy actions quickly move toward $\rho=0$ and the largest improvement over the fixed parameter AMR policy is observed.

\begin{figure}
\centering
\setlength\tabcolsep{0mm}
\begin{tabular}{ccccc} 
	\includegraphics[width=0.22\textwidth]{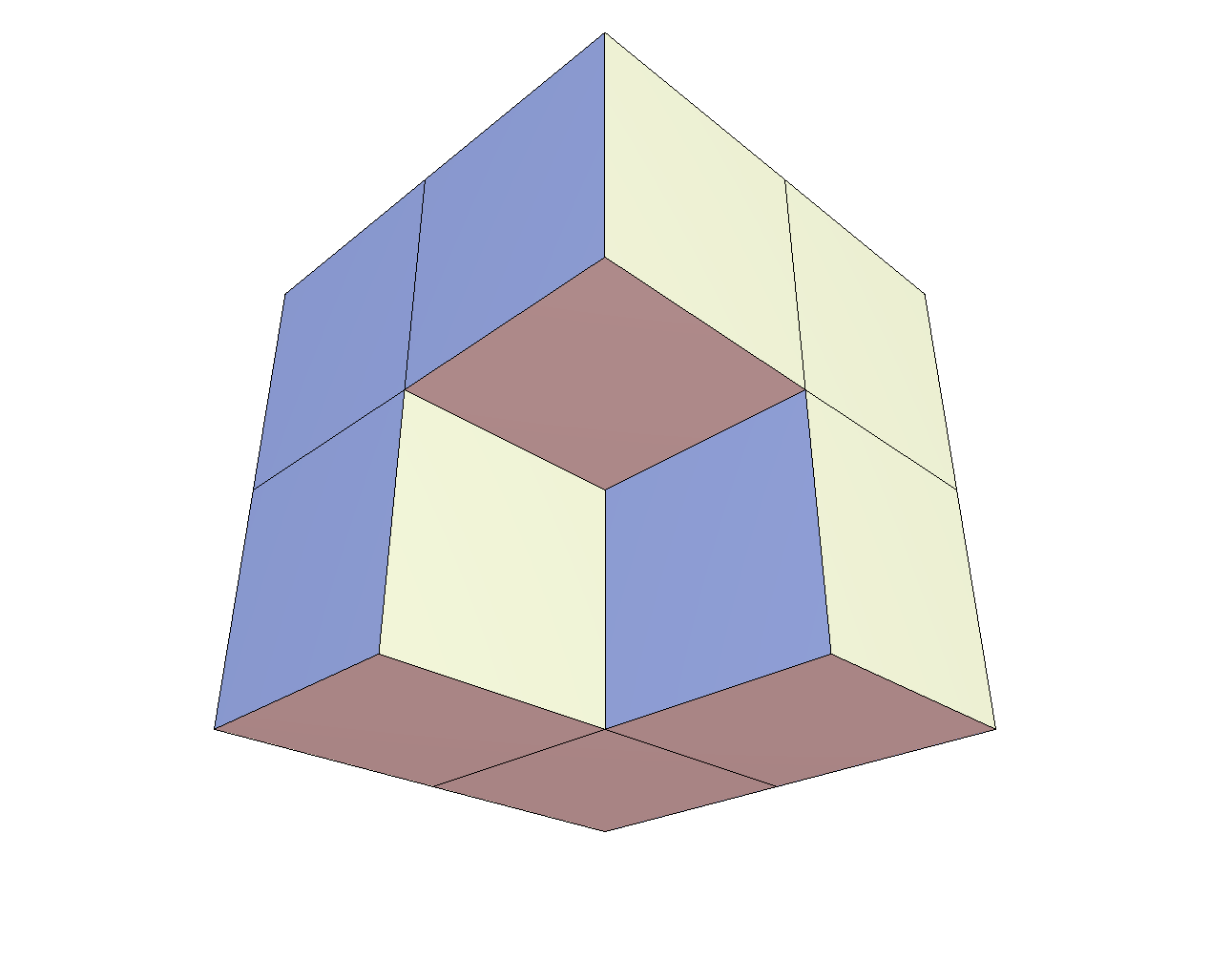} &
		\includegraphics[width=0.22\textwidth]{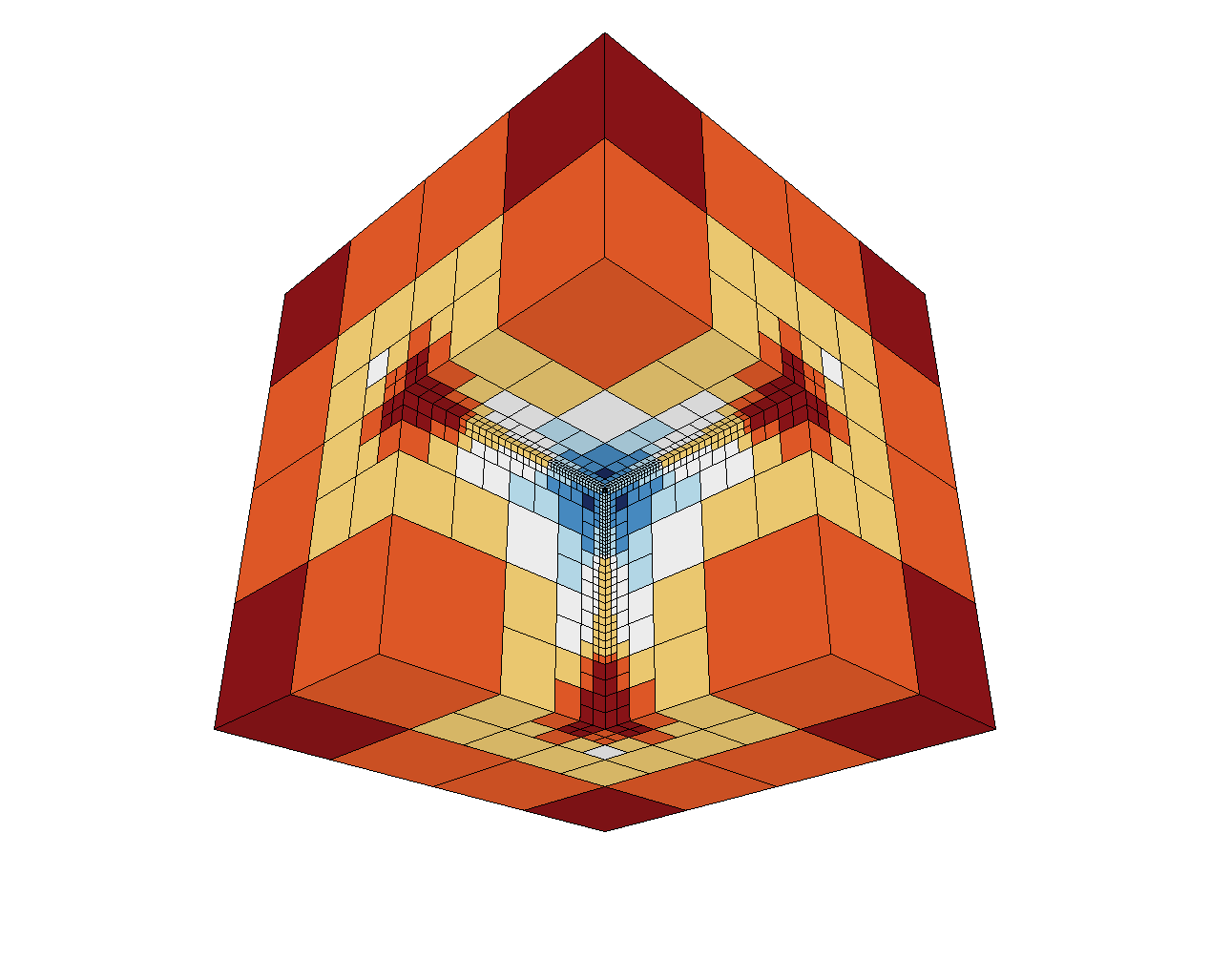} &
				~\includegraphics[width=0.02\textwidth]{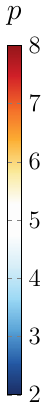}~ &
							\includegraphics[width=0.22\textwidth]{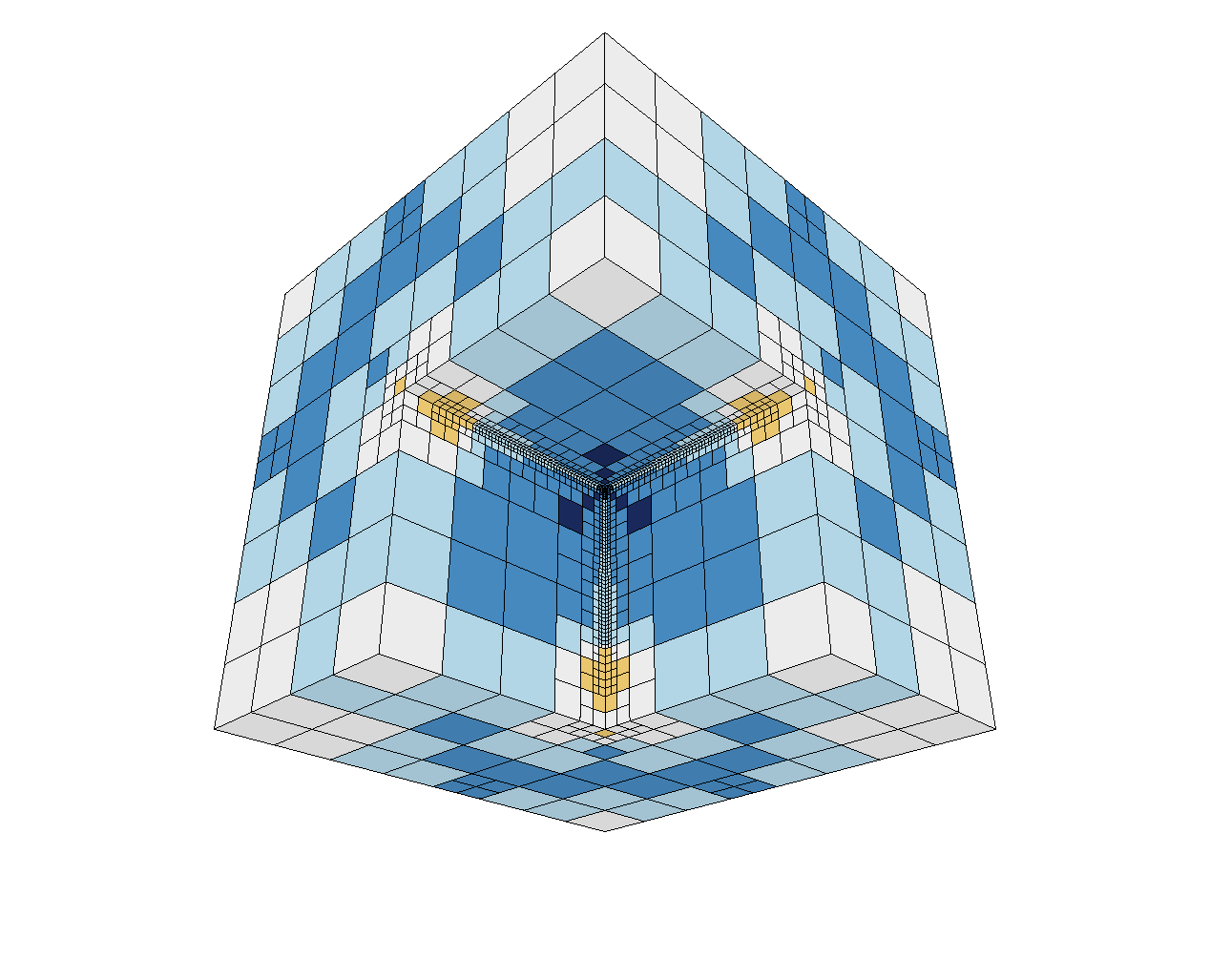} &~
	\includegraphics[width=0.19\textwidth]{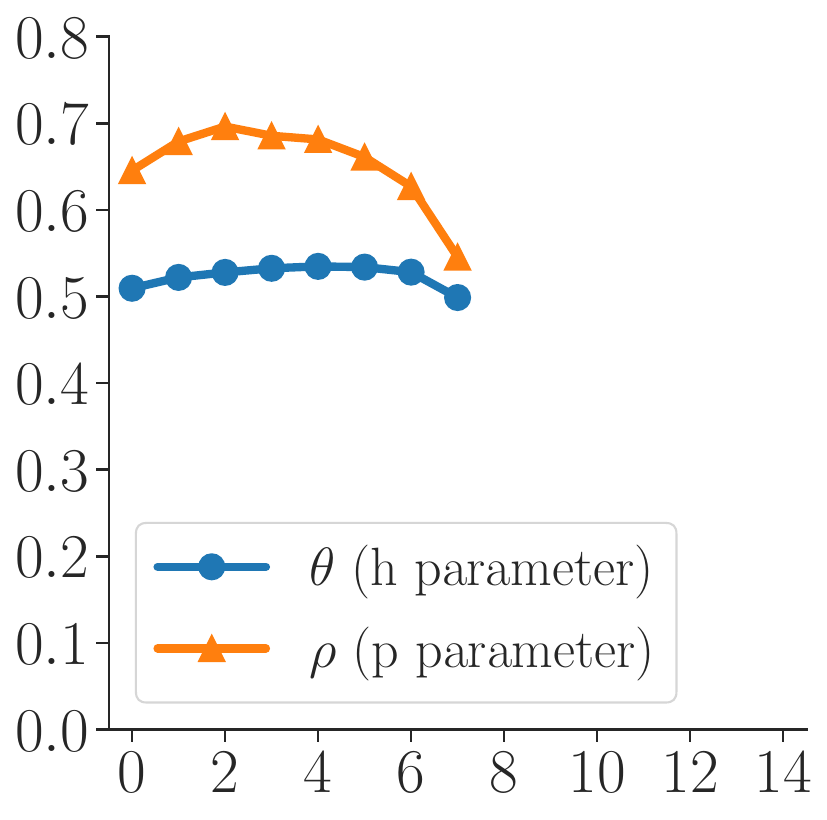}  \\
 Initial mesh & AMR policy & & (AM$)^2$R policy & Policy actions
\end{tabular}
	\caption{Solving the Poisson problem on a 3D mesh of the Fichera corner.  The visualization of the the final meshes indicates that the fixed-parameter AMR policy refines higher in $p$ than the adaptive-parameter (AM$)^2$R policy; the color bar indicates the order of $p$ ranging from $p=2$ in blue to $p=8$ in red.  The (AM$)^2$R policy attains a lower final error estimate than the AMR policy for the same cumulative dof threshold~(cf.~\Cref{tab:hp-results}).}
	\label{fig:hp-fichera}
\end{figure}

We carry out an additional experiment in \rev{generalizability} by deploying the (AM$)^2$R policy on a 3D mesh of a Fichera corner, consisting of seven cubes arranged to form a larger cube with one octant missing; see~\Cref{fig:hp-fichera}.
We again solve the Poisson problem \cref{eq:PoissonEqn}. 
To accommodate the faster rate of growth in dofs in 3D, we raise the cumulative dof threshold $J_\infty$ to $5\times 10^5$; all training and previous testing had $J_\infty=10^4$.
The benefit of not including geometric information in our observation space is now realized as the (AM$)^2$R policy immediately works without modification or additional training.
Furthermore, as indicated in the last row of~\Cref{tab:hp-results}, the (AM$)^2$R policy outperforms the optimal fixed-parameter policy with an improvement factor of 1.47.
\section{Discussion} \label{sec:future_research_directions}

In this work, we focused on learning a map from normalized statistics of local error estimates to marking parameters.
These statistics only partially characterize the space of discretization states such maps should aim to probe.
Therefore, future research may involve learning higher-dimensional maps involving more sophisticated simulation observables.
Doing so may lead to better performing marking policies for $hp$-refinement or new policies for, e.g., goal-oriented AMR.

Another important future research direction is the development of marking policies for time-dependent PDEs.
Unpublished experiments by the authors have shown that the approach presented here can generalize to such settings and the associated training can be performed using, e.g., policy gradient methods for infinite-horizon environments (cf. \cite[Section~13.6]{sutton2018reinforcement}).
Ongoing work by the authors is dedicated to developing refinement policies for time-dependent PDEs.

Finally, we believe the true benefits of this new AFEM paradigm lie in \emph{transfer learning}.
That is, training on inexpensive model problems with the aim of improving performance on more expensive target problems; cf.~\Cref{ssub:ex_transfer}.
Future work should focus in part on innovating new tools and techniques to efficiently train robust marking policies for more complicated transfer learning applications.

\section{Conclusion} \label{sec:conclusion}

In this work, we introduced a doubly adaptive AFEM paradigm that treats the selection of marking parameters as a state-dependent decision made by a marking policy which can optimized with policy gradient methods from the reinforcement learning literature.
We then demonstrated the potential of this novel paradigm for $h$- and $hp$-refinement applications via experiments on benchmark problems.

In our first experiment (cf.~\Cref{ssub:ex1a}), we focused on $h$-refinement with the well-studied L-shaped domain problem~\cite{mitchell2013collection}.
Here, we demonstrated that the efficiency of AFEM can be significantly improved through adaptive parameter selection.
In particular, we witnessed the superior efficiency of a pre-trained adaptive marking policy when compared against the best performing counterpart fixed-parameter policy.
In this experiment, we also witnessed nearly twice the efficiency when compared against the median-performing fixed-parameter policy.

In our second and third experiments (cf.~\Cref{ssub:ex2c,ssub:ex_transfer}, respectively), we considered learning a robust marking policy for $hp$-refinement over a distribution of model problems.
The first of these experiments demonstrated that our chosen observation space is expressive enough to deliver policies with superior average performance across a distribution of training problems.
The second of these experiments demonstrated that such robust policies can also deliver superior performance on \emph{unseen} model problems.
Indeed, after only training a marking policy on $2$D Poisson equations whose domains have a single re-entrant corner, we could apply the trained policy to far more complicated domains\textemdash{}even $3$D domains\textemdash{}without significant loss of efficiency.
For the purpose of reproduction and wider adoption, this work is accompanied by an open-source Python-based implementation \cite{Code}.

\subsection*{Acknowledgements} \label{sub:acknowledgements}

We thank Bob Anderson, Daniel Faissol, Brenden Petersen, and Tzanio Kolev for first considering adaptive mesh refinement as a Markov decision process, securing funding to explore the idea, and ultimately recruiting and encouraging us to further develop it.
We also thank Dylan Copeland, Tarik Dzanic, Ketan Mittal, and Jiachen Yang for countless helpful discussions along the way.
Finally, last but not least, we thank Justin Crum for his early numerical experiments with $hp$-refinement and Jennifer Zvonek for her careful reading of the manuscript.

\subsection*{Disclaimer} \label{sub:disclaimer}
{\small
This work was performed under the auspices of the U.S. Department of Energy by Lawrence Livermore National Laboratory under contract DE--AC52--07NA27344 and the LLNL-LDRD Program under Project tracking No.\ 21--SI--001.  
Release number LLNL--JRNL--837212.
}

\phantomsection\bibliographystyle{siamplain}
\bibliography{main.bib}

\end{document}